\documentclass{article} 
\usepackage{amsmath, amsthm, amssymb, newtxmath} 

\usepackage{subcaption}
\usepackage{changepage} 

\usepackage{graphicx}
\usepackage{tikz-cd}
\usepackage{bbm}
\usepackage[margin=1in]{geometry} 
\usepackage{microtype} 
\usepackage{parskip}   

\usepackage[colorinlistoftodos,prependcaption,textsize=small]{todonotes}

\usepackage{adjustbox}

\usepackage[backend=biber, style=alphabetic]{biblatex}  
\newtheorem{theorem}{Theorem}[section] 
\newtheorem{proposition}[theorem]{Proposition} 
\newtheorem{lemma}[theorem]{Lemma} 
\newtheorem{corollary}[theorem]{Corollary} 

\theoremstyle{definition}
\newtheorem{definition}[theorem]{Definition} 

\theoremstyle{remark}
\newtheorem{remark}[theorem]{Remark} 

\newtheorem{construction}[theorem]{Construction} 

\usepackage{amsthm}         
\usepackage{algpseudocode}  

\newtheorem{algo}{Algorithm}

\algrenewcommand\algorithmicrequire{\textbf{Input:}}
\algrenewcommand\algorithmicensure{\textbf{Output:}}
\algnewcommand{\Input}{\Statex \textbf{Input:} }
\algnewcommand{\Output}{\Statex \textbf{Output:} }

\newcommand{\twist}{\mathsf{w}_1}
\newcommand{\sdist}{d_{\mathbb{S}^{1}}}
\newcommand{\eps}{\varepsilon}

\usepackage[colorlinks=true, linkcolor=blue, citecolor=blue, urlcolor=blue]{hyperref}

\title{Discrete Approximate Circle Bundles\thanks{ This work was partially supported by the National Science Foundation through   CAREER award \# DMS-2415445.}}
\author{Brad Turow, Jose A. Perea}

\date{}
\begin{document}
\maketitle

\begin{abstract}
        In this paper, we introduce \emph{discrete approximate circle bundles}, a class of objects designed to serve as the data science analog of  circle bundles from algebraic topology. 
        We show that, under appropriate conditions, one can meaningfully and stably identify a discrete approximate circle bundle with an isomorphism class of true circle bundles. 
        We also describe two cohomology invariants which uniquely determine the isomorphism class of a circle bundle, and provide algorithms to   compute them given a discrete approximate representative. 
        Finally, we propose a novel methodology for coordinatization and dimensionality reduction of circle bundle data. 
        To illustrate the practical utility and viability of our algorithms, we present applications to both real and synthetic datasets from computer vision (e.g., modeling optical flow). 
        The paper is accompanied by an open-source software package, with full documentation and tutorials, enabling reproducible implementation of the proposed algorithms and experiments, including those used to generate the figures in this paper
\end{abstract}

\section{Background And Problem Statement}

\subsection{Motivation And Context}

It has been observed in recent years that high-dimensional datasets arising in fields such as computer vision, computational chemistry, and motion tracking often lie near highly non-linear, low-dimensional manifolds \cite{tenenbaum2000isomap, Roweis2000, Belkin2003}. 
These manifolds can exhibit complex geometric and topological structure, motivating the development of a range of techniques for identifying and leveraging such structures in tasks including dimensionality reduction, statistical analysis, pattern recognition, and data visualization \cite{Coifman2006, Sapiro2001}. 
By modeling the intrinsic manifold geometry, these methods have been shown to improve compression, denoising, and feature extraction in domains where traditional linear techniques fail to capture the underlying structure of the data \cite{JMLR:v24:21-0073}.\\

In practice, complex non-linear structure in large, high-dimensional datasets is often difficult to observe directly; 
it is therefore essential to have techniques for inferring global structure from purely local computations on small portions of the data.
In this paper, we focus on a particular class of datasets which can be modeled as \textit{circle bundles}. 
Roughly speaking, a circle bundle is a continuously-varying family of (topological) circles parameterized by another space, called the \textit{base space} of the bundle. 
The individual circles are called the \textit{fibers} of the bundle, and the union of the fibers is the \textit{total space}. 
The parametrization is given by a projection map $\pi:E\to B$ from the total space to the base space; 
in particular $E$ is locally homeomorphic to a product $U\times \mathbb{S}^{1}$, for $U\subseteq B$ open, but globally the structure may be twisted. 
For instance, the torus and the Klein bottle can each be characterized as the total space of a circle bundle with base space $\mathbb{S}^{1}$, but the two are not homeomorphic (see Figure~\ref{fig:klein_bundle_proj}). 
Similarly, the 3-manifold $SO(3)$ admits a non-trivial circle bundle structure over $\mathbb{S}^{2}$.
Indeed, consider the map $\pi:SO(3)\to \mathbb{S}^{2}$ which assigns an (oriented) axis of rotation to each element of $SO(3)$; the fibers of this map are copies of $SO(2)\cong\mathbb{S}^{1}$, but one can show that $SO(3)\ncong\mathbb{S}^{2}\times\mathbb{S}^{1}$.\\

\begin{figure}[h!]
\centering
\begin{tikzpicture}[scale=4]

\draw[thick] (0,0) rectangle (1,1);

\draw[->, thick] (0,1.05) -- (1,1.05); 
\draw[->, thick] (0,-0.05) -- (1,-0.05); 

\draw[->, thick] (-0.05,0) -- (-0.05,1); 
\draw[->, thick] (-0.05,0) -- (-0.05,0.95); 
\draw[<-, thick] (1.05,0) -- (1.05,1);   
\draw[<-, thick] (1.05,0.05) -- (1.05,1);   

\node at (-0.14,0.5) {K};

\fill[gray!40] (0.57,0) rectangle (0.70,1);
\draw[black, thick] (0.635,0) -- (0.635,1); 

\node[black, right] at (0.67,0.55) {$\pi^{-1}(U)$};

\draw[thick, ->] (0,-0.40) -- (1,-0.40);
\node at (-0.1,-0.40) {$\mathbb{S}^{1}$};

\draw[black, line width=2pt] (0.57,-0.40) -- (0.70,-0.40);
\node[black] at (0.635,-0.46) {$U$};

\fill[black] (0.635,-0.40) circle (0.015);

\draw[->, thick] (0.635,-0.02) -- (0.635,-0.38);
\node at (0.68,-0.22) {$\pi$};

\node[rotate=90] at (0.56,0.58) {Fiber $\mathbb{S}^{1}$};

\end{tikzpicture}
\caption{The Klein bottle $K$ as a circle bundle over $\mathbb{S}^{1}$. The highlighted region $\pi^{-1}(U)$ in $K$ shows all fibers above the neighborhood $U$ in the base space and is homeomorphic to $U\times\mathbb{S}^{1}$. Globally, however, $K$ is not homeomorphic to $\mathbb{S}^{1}\times\mathbb{S}^{1}$.}
\label{fig:klein_bundle_proj}
\end{figure}

On the machine learning side, various datasets related to computer vision and graphics have been shown to have circle bundle structures. 
For instance, in~\cite{opt_flow_torus}, Adams et al. use local measurements and a feature map to support a torus model for high-contrast optical flow data. 
Similarly, in~\cite{Klein_bottle}, it is shown that frequently-occurring high-contrast patches from natural images concentrate around a Klein bottle embedded in a high-dimensional Euclidean space.
It is worth noting that in each of these contexts, a direct persistent homology computation fails to satisfactorily confirm the topology of the proposed underlying model. 
Figure~\ref{fig: Wavy torus} shows an illustration of how this failure can occur, and shows how the local-to-global methodologies we propose in this paper successfully recover the underlying structure. \\ 

\begin{figure}[h!]
\centering

\begin{tikzpicture}
\node (A) {\begin{minipage}{0.90\textwidth}
    \centering
    \begin{subfigure}[t]{0.3\textwidth}
        \centering
        \includegraphics[width=\linewidth]{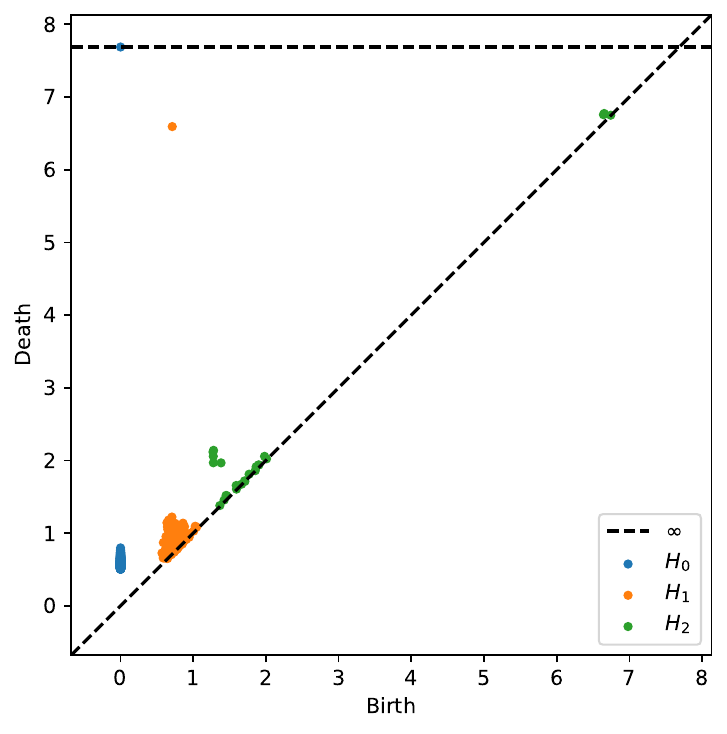}
        \caption{Persistence diagrams of the original toroidal dataset $X$}
    \end{subfigure}\hfill
    \begin{subfigure}[t]{0.28\textwidth}
        \centering
        \includegraphics[width=\linewidth]{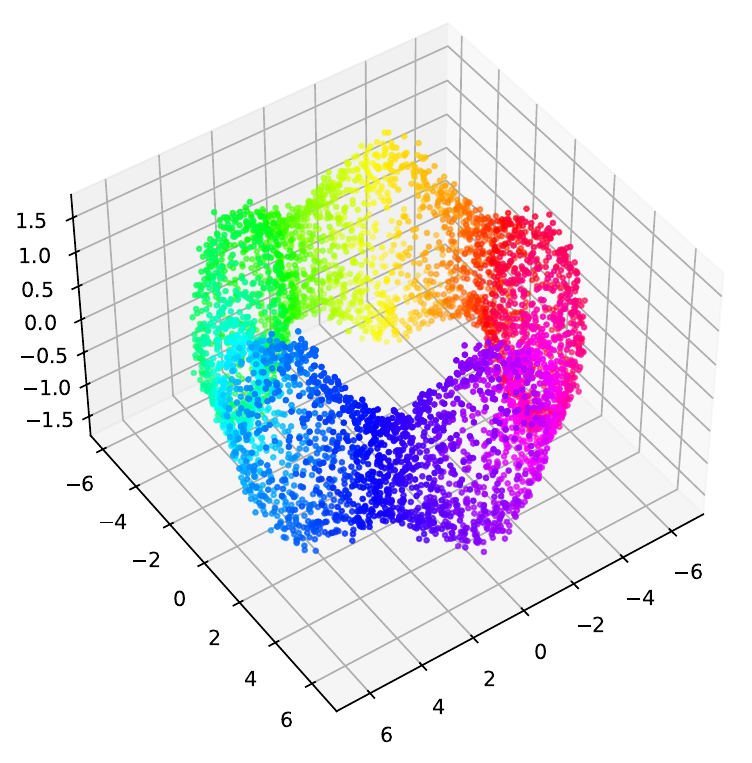}
        \caption{$X$ colored by the circular coordinate derived from its most persistent 1-dimensional cohomology class \cite{Sparse_CC}}
    \end{subfigure}\hfill
    \begin{subfigure}[t]{0.28\textwidth}
        \centering
        \includegraphics[width=\linewidth]{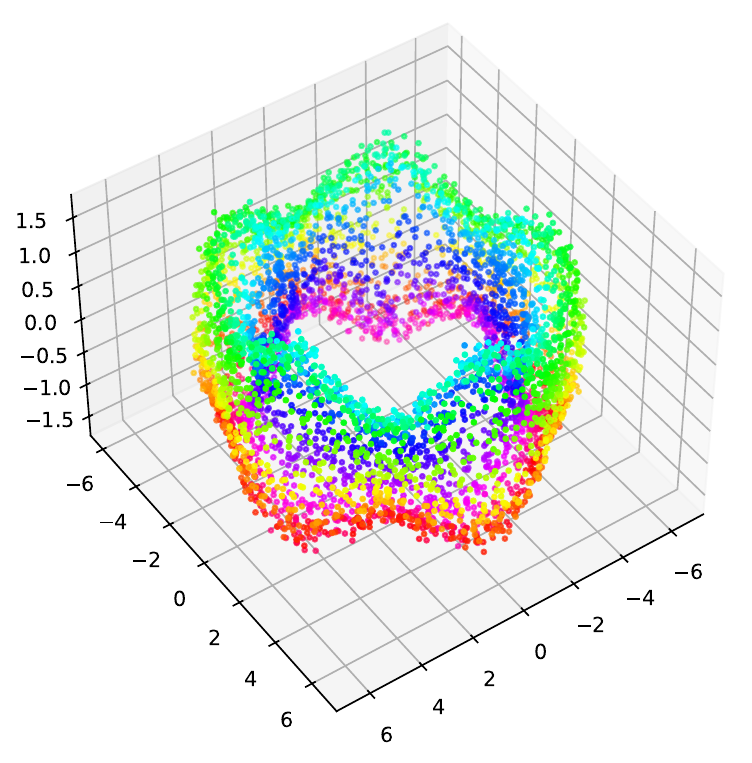}
        \caption{Fiberwise angle resulting from trivializing the circle bundle over $\mathbb{S}^1$ given by (b)}
    \end{subfigure}
\end{minipage}};

\node[anchor=north west, xshift=6pt, yshift = -10pt] at (A.north east)
    {\includegraphics[width=0.06\textwidth]{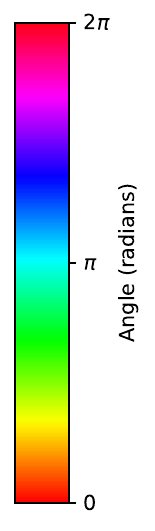}};
\end{tikzpicture}

\caption{A noisy sample $X$ from a torus whose fiber radius continuously oscillates. The persistence diagrams show only a single significant class in dimension 1 (recall that the non-zero Betti numbers for the torus are $\beta_{0} = 1$, $\beta_{1}=2$, $\beta_{2}=1$). Figure (b) shows the dataset colored according to circular coordinates computed with DREiMac~\cite{DREiMac}. Figure (c) shows the same dataset colored according to the fiber coordinates computed using the software package accompanying this paper.}
\label{fig: Wavy torus}
\end{figure}

Beyond these examples, datasets arising from 3D objects with rotational symmetries — such as projection images of macromolecules in cryo-electron microscopy — naturally give rise to circle bundle structures, where each fiber corresponds to rotations of the object about a fixed axis~\cite{cryoem_singer_sigworth}.
Here the underlying model is typically a 3-manifold such as $SO(3)$, and a direct persistence computation to capture the topology is simply intractable.  
On the other hand, one can infer the underlying manifold from purely local computations by leveraging the bundle structure, as we demonstrate with a synthetic dataset of 3D density functions in Section~\ref{sec: Densities}. \\

The different possible circle bundle structures over a topological space $B$ are completely described, up to isomorphism, by a pair of discrete invariants called \textit{characteristic classes}. 
In particular, a non-vanishing class indicates the bundle is non-trivial, so there is no hope of constructing a coordinatization for the whole space as a product. 
Additionally, characteristic classes are invariant under continuous perturbations, so one might hope to be able to stably recover them from a noisy finite sampling. 
Another key feature of characteristic classes is that they can be computed from local measurements alone, which makes them suitable for working with complex, high-dimensional data not amenable to more direct techniques. 
Lastly, the characteristic classes of a bundle can be used to construct a so-called \textit{bundle map} from the total space to a more familiar space which respects topological structure. This observation has recently been leveraged in works such as~\cite{LeeEtAl2025} and~\cite{Euclidean_Vector_Bundles} to develop non-linear dimensionality reduction techniques based on vector bundles. 
In this paper, we extend this approach to circle bundles.\\ 

\subsection{Our Contributions}

Here we summarize our contributions: \\

\begin{enumerate}
    \item We define notions of discrete approximate circle bundles (Definition~\ref{def: disc appr CB}) in terms of approximate local trivializations (Definition~\ref{def: disc local triv}), and show how these relate to other well-studied objects such as discrete approximate \v{C}ech cocycles (Proposition~\ref{prop: triv to cocycle}). 
    Building on results of~\cite{Tinarrage2022} and~\cite{Euclidean_Vector_Bundles}, we give conditions under which a discrete approximate circle bundle can be meaningfully identified with a true class of circle bundles (Theorem~\ref{thm: disc CB to true}).
    
    \item We show that every true circle bundle is uniquely characterized up to isomorphism by an orientation class and a twisted Euler class (Theorem~\ref{prop: general class thm}), and provide an algorithm for computing these classes from a system of discrete approximate local trivializations (Algorithm~\ref{alg:CharClasses}). We further introduce a metric on each space of discrete approximate local trivializations and show that two trivializations which are sufficiently close produce the same characteristic classes (Corollary~\ref{cor: char of true bundle}). In particular, we prove that if a discrete approximate local trivialization is sufficiently close to a true one, the characteristic classes obtained from our algorithm are the same (Corollary~\ref{prop: algorithm computes classes}). \\

    \item Given a discrete approximate circle bundle over a base space $B$ and an open cover $\mathcal{U}$ for $B$, we introduce a weights filtration on the nerve complex $\mathcal{N}(\mathcal{U})$ and a corresponding notion of persistence for characteristic class representatives (Section~\ref{sec: Weights Filt}). We then propose a scheme for coordinatization and dimensionality reduction which produces a map from the dataset to the total space of the associated bundle $V(2,d)\times_{O(2)}\mathbb{S}^{1}$, where $V(2,d)$ denotes the Stiefel manifold of 2-planes in $\mathbb{R}^{d}$ (Section~\ref{sec: Coordinatization}). 
    Our pipeline is designed to be compatible with any stage of the weights filtration on $\mathcal{N}(\mathcal{U})$, yielding multi-scale representations and coordinatizations of the data consistent with its multi-scale global topology. \\  

    \item We demonstrate implementations of our algorithms on several real and synthetic datasets related to imaging and optical flow. 
    In Section~\ref{sec: optical flow}, we use these tools to confirm the torus model proposed in~\cite{opt_flow_torus}. A more detailed analysis of our findings can be found in our companion paper~\cite{turow2026extended}. \\ 

    \item We provide an open-source software package implementing the algorithms introduced in this paper,  available at
\url{https://github.com/bradturow/Circle_Bundles}, with full documentation and tutorials hosted at
\url{https://circle-bundles.readthedocs.io} \\
\end{enumerate}

\subsection{Prior Work}

Recently, there has been an emergence of new approaches in local-to-global inference for analyzing geometrically and topologically complex datasets.
The basic principle in each case is to model small pieces of the dataset and analyze how these parts glue together to determine the global shape. 
Given a dataset $X$, a common technique is to choose a feature map $\pi:X\to S$ to some metric space $S$ (e.g., $\mathbb{R}$) and study how the level sets vary. 
The Mapper pipeline takes this approach to represent a point cloud as a graph with nodes representing local clusters~\cite{Mapper}. 
Similarly, in~\cite{scoccola2022fibered} the authors describe a pipeline called FibeRed which uses linear dimensionality reduction --- e.g., via PCA or ISOMAP --- on the fibers of a feature map $\pi:X\to S$ to a Riemannian manifold $S$, to obtain a low-dimensional representation of $X$ which captures both local geometry and global topology. \\

Another aspect of local-to-global analysis is the synchronization of local coordinate systems. Cohomological obstructions to synchronization are explored in works such as~\cite{gao_synchronization} using the language of principal $G$-bundles. We note, however, that the theory provided there makes the restrictive assumption that the underlying bundle is \textit{flat} (which is not the case in most of the applications we consider). In~\cite{Euclidean_Vector_Bundles}, the authors introduce a theory of discrete approximate Euclidean vector bundles which handles non-flat cases in the context of $O(d)$-synchronization. Building on the language and techniques developed in~\cite{Tinarrage2022}, they show that under appropriate conditions, one can uniquely and stably identify discrete approximate Euclidean vector bundles with true ones. 
They also provide stable algorithms for computing characteristic classes of vector bundles from discrete approximate cocycle representatives. \\

\subsection{Structure Of Paper}

The rest of the paper is structured as follows:\\

\begin{adjustwidth}{1.5em}{0pt}
\textbf{Section \ref{sec: Mathematical Background}: Mathematical Preliminaries.} 
We review the essential mathematical elements required for our results, including the general theory of fiber bundles, the classification of circle bundles and cohomology with local coefficients. This section establishes notation and key concepts that will be used throughout the paper.\\

\textbf{Section \ref{sec: Disc Apprx Circ Bundles}: Discrete Approximate Circle Bundles.} We introduce our proposed notion of discrete approximate circle bundles and explore its relationship with discrete approximate \v{C}ech cocycles and classifying maps as defined in~\cite{Euclidean_Vector_Bundles}. We also discuss how these discrete constructions relate to classical circle bundles.\\

\textbf{Section \ref{sec: Char Classes}: Algorithms for Characteristic Classes.} We present algorithms to compute characteristic classes from discrete approximate local trivializations, and prove that our algorithms are robust to input perturbations (i.e., stable).\\

\textbf{Section \ref{sec: Persistence}: Weights Filtration and Persistence.} Given a discrete approximate circle bundle over a metric space $B$ and an open cover $\mathcal{U}$ of $B$, we introduce a  filtration on the nerve complex $\mathcal{N}(\mathcal{U})$ and describe a type of persistent cohomology for characteristic class representatives with respect to this filtration. \\

\textbf{Section \ref{sec: Coordinatization}: Coordinatization Pipeline.} We provide a pipeline for coordinatization and dimensionality reduction of circle bundle data, which integrates the Principal Stiefel Coordinates algorithm~\cite{LeeEtAl2025}. The result is a function from the dataset into a Stiefel manifold which is, in some sense which we explain, compatible with a chosen stage of the associated weights filtration. \\

\textbf{Section \ref{sec: Implementations}: Implementations and Experiments.} We demonstrate our methods on both synthetic and real datasets, applying the sparse circular coordinates algorithm~\cite{Sparse_CC} to obtain discrete approximate local trivializations and showcasing the resulting computations of characteristic classes and coordinatization.\\

\textbf{Section \ref{sec: Future Work}: Future Work.} We conclude with a discussion of potential extensions, open questions, and directions for further exploration in both theory and applications.\\
\end{adjustwidth}

\section{Preliminaries}
\label{sec: Mathematical Background}

\subsection{Fiber Bundles}\label{sec: Fiber Bundles}

In this section, we briefly review the key aspects of the theory of fiber bundles which are relevant for us. Our exposition emphasizes a computational perspective. For additional details, see \cite{husemoller}.\\

\begin{definition}
    Let $F$ be a topological space. A \textbf{fiber bundle with fiber $F$} is a triple $\xi = (E,\pi,B)$, where $E$ and $B$ are topological spaces and $\pi:E\to B$ is a continuous surjection with the following property: for each $b\in B$, there exists an open  neighborhood $U_{b}\subseteq B$ of $b$ and a homeomorphism $\varphi:\pi^{-1}(U_{b})\to U_{b}\times F$ such that if $p_1: U_b\times F \to U_b$ is the projection onto the first factor, then the following diagram commutes: 

\begin{equation*}
\begin{tikzcd}
\pi^{-1}(U_{b}) \arrow[rr, "\varphi"] \arrow[dr, swap, "\pi"'] & & U_{b}\times F \arrow[dl, "p_{1}"] \\
& U_{b} &
\end{tikzcd}
\end{equation*}

    \noindent Any such map $\varphi$ is called a \textbf{local trivialization of $\xi$}. An open cover $\mathcal{U}$ of $B$ is called a $\textbf{trivializing cover}$ for $\xi$ if there exists a local trivialization of $\xi$ over each $U \in\mathcal{U}$. Two bundles $\xi=(E,\pi,B)$ and $\xi' = (E',\pi',B)$ are \textbf{isomorphic} if there exists a homeomorphism $f:E\to E'$ such that $\pi'\circ f = \pi$. Let $\text{Bun}_{F}(B)$ denote the set of isomorphism classes of fiber bundles over $B$ with fiber $F$. \\
\end{definition}

\begin{definition}
    Let $G$ be a topological group. A \textbf{principal $G$-bundle} is a fiber bundle $\xi = (E ,\pi, B)$  with fiber $G$, satisfying the following properties: $E$ has a free right $G$-action so that $\pi(eg) = \pi(e)$ for all $(e,g) \in E\times G$ --- i.e., the action is fiber-preserving --- and which is transitive when restricted to each fiber $\pi^{-1}(b)$, $b \in B$. Moreover the local trivializations $\varphi : \pi^{-1}(U_b) \to U_b \times G$ can be chosen to be $G$-equivariant, where the right $G$-action on $U_b\times G$ is group multiplication in $G$. 
    Two principal $G$-bundles $(E,\pi,B)$ and $(E', \pi' , B)$ are isomorphic (as principal $G$-bundles) if there exists a $G$-equivariant homeomorphism $f: E \to E'$ so that $\pi'\circ f= \pi$. Let $\text{Prin}_G(B)$ denote the set of isomorphism classes of principal $G$-bundles over $B$.\\
\end{definition}

\noindent \textbf{Transition Maps:} Given a fiber bundle $\xi = (E,\pi, B)$ with fiber $F$, a trivializing cover $\mathcal{U} = \{U_j\}_{j\in J}$ of $B$ and a family $\{\varphi_{j}:\pi^{-1}(U_{j})\to U_{j}\times F\}_{j\in J}$ of local trivializations subordinate to $\mathcal{U}$, there exist unique continuous maps $\Omega_{jk}:U_{j}\cap U_{k}\to \text{Homeo}(F)$ such that $\varphi_{j}\circ\varphi_{k}^{-1}(b,v) = (b, \Omega_{jk}(b)(v))$ for any $(b,v)\in (U_{j}\cap U_{k})\times F$. 
Here $\text{Homeo}(F)$ denotes the set of homeomorphisms from $F$ to $F$, endowed with the compact-open topology. 
The maps $\{\Omega_{jk}\}_{(jk)\in\mathcal{N}(\mathcal{U})}$,
indexed by the 1-simplices of the nerve of $\mathcal{U}$, are called the \textbf{transition maps} associated with the family of local trivializations. Note that the transition functions associated with any family of local trivializations must satisfy the \textbf{cocycle condition}:

\begin{equation}\label{eq: Cech cocycle condition}
    \Omega_{jk}(b)\circ \Omega_{kl}(b) = \Omega_{jl}(b) \hspace{1cm} (b\in U_j \cap U_k \cap U_l\; , \; (jkl)\in \mathcal{N}(\mathcal{U}))
\end{equation}

\noindent If we require the transition maps to be $G$-valued for some subgroup $G \leq \text{Homeo}(F)$, we say that $G$ is the $\textbf{structure group}$ of the bundle. \\

\noindent \textbf{Fiber Bundle Construction Lemma:} The following well-known lemma essentially states that a fiber bundle is completely determined up to isomorphism by the transition maps associated with a family of local trivializations. It also gives a computationally-friendly way to determine whether two bundles over the same space are isomorphic:\\

\textbf{Part A:} Let $B$ and $F$ be topological spaces, and let $G$ be a topological group with a continuous left action on $F$. 
Given an open cover $\mathcal{U} = \{U_{j}\}_{j\in J}$ of $B$ and a collection $\{\Omega_{jk}:U_{j}\cap U_{k}\to G\}_{(jk)\in\mathcal{N}(\mathcal{U})}$ of continuous maps satisfying the cocycle condition (\ref{eq: Cech cocycle condition}), there exists a fiber bundle $\pi:E\to B$ with fiber $F$ and structure group $G$ that is trivializable over $\mathcal{U}$ with transition maps $\{\Omega_{jk}\}_{(jk)\in\mathcal{N}(\mathcal{U})}$  \cite[Chapter 5, Theorem 3.2]{husemoller}. \\

\textbf{Part B:} Suppose $\xi' = (E',\pi',B)$ is another fiber bundle with fiber $F$ and structure group $G$ (endowed with the same left-action) which is trivializable over $\mathcal{U}$ with transition maps $\{\Omega_{jk}'\}_{(jk)\in\mathcal{N}(\mathcal{U})}$. If the action of $G$ on $F$ is faithful, then $\xi\cong\xi'$ if and only if there exist maps $\{\mu_{j}:U_{j}\to G\}_{j\in J}$ such that $\Omega_{jk}' = \mu_{j}^{-1}\Omega_{jk}\mu_{k}$ for all $(jk)\in\mathcal{N}(\mathcal{U})$ \cite[Chapter 5, Theorem 2.7]{husemoller}.\\

\noindent Note that if $\xi = (E,\pi,B)$ and $\xi'= (E',\pi', B)$ are fiber bundles with fiber $F$ over a common paracompact space $B$ with trivializing covers $\mathcal{U}$ and $\mathcal{U}'$, then any common refinement $\mathcal{V}$ of $\mathcal{U}$ and $\mathcal{U}'$ is a trivializing cover for $\xi$ and $\xi'$, so we can compare the two bundles using this refinement (and the result does not depend on the choice of refinement).\\

\noindent \textbf{Associated Bundles:} Let $\xi = (E,\pi,B)$ be a fiber bundle with fiber $F$ and structure group $G$.  
Let $\mathcal{U} =\{U_j\}_{j\in J}$ be a trivializing  cover for $\xi$ and let  
$\{\Omega_{jk} : U_j \cap U_k \to G\}_{(jk)\in \mathcal{N(U)}}$
be the transition maps of a local trivializing of $\xi$ over $\mathcal{U}$.
If $F'$ is a left $G$-space, then Part A of the Fiber Bundle Construction Lemma defines  a fiber bundle 
$\xi'$ over $B$ with fiber $F'$.
The association $\xi \mapsto \xi'$ is well-defined up to isomorphism of fiber bundles, and we call $\xi'$ the $\textbf{associated bundle}$ to $\xi$ with fiber $F'$.
The special case  $F' = G$ with left action by group multiplication is called the $\textbf{principal}$ $G-\textbf{bundle}$ associated to $\xi$. 
When $\xi = (E, \pi, B)$ is a principal $G$-bundle and $F'$ is a left $G$-space, one can construct $\xi'$ directly without appealing to transition functions or isomorphism classes as follows: 
The multiplication $\cdot : (E\times F')\times G \to E\times F'$
given by $(e,v')\cdot g = (eg \,,\, g^{-1}*v')$ defines a right $G$-action on  $E\times F'$, and the    orbit  space 
$E\times_G F'$ has a map $\pi' : E\times_G F' \to B$
given by $\pi'([e, v']) = \pi(e)$. 
One can show that $\xi' = (E\times_G F' , \pi', B)$ is a fiber bundle with fiber $F'$. 
Finally, we note that for any space $B$, the associated bundle construction gives a bijective correspondence between $\text{Bun}_F(B)$ and $\text{Prin}_{\text{Homeo}(F)}(B)$, where $F$ is understood to be endowed with the tautological $\text{Homeo}(F)$-action by evaluation.
\\ 

If $\xi = (E,\pi,B)$ is a principal $G$-bundle and $\varphi:G\to H$ is a continuous group homomorphism, then $H$ has a continuous left $G$-action given by $g\cdot h = \varphi(g) h$. The associated bundle, denoted by $\varphi_{*}\xi$, is then a principal $H$-bundle. In particular, if $\{\Omega_{jk}\}_{(jk)\in\mathcal{N}(\mathcal{U})}$ are transition maps for $\xi$ subordinate to some open cover $\mathcal{U}$ of $B$, then $\{\varphi\circ\Omega_{jk}\}_{(jk)\in\mathcal{N}(\mathcal{U})}$ are transition maps for $\varphi_{*}\xi$. \\

\noindent \textbf{\v{C}ech Cohomology:} The cocycle condition~\eqref{eq: Cech cocycle condition} and Part B of the Fiber Bundle Construction Lemma suggest a cohomological perspective to fiber bundle classification. This perspective can be articulated using the language of \v{C}ech cohomology, which we briefly describe below. See \cite[Chapter IX]{miranda1995algebraic} for definitions.\\

Let $\mathcal{U} = \{U_\alpha\}_{\alpha\in \Lambda}$  be an open cover of a topological space $B$ and let $\mathcal{F}$ be a presheaf of abelian groups on $B$. 
For $k\geq 0$, define $\check{C}^{k}(\mathcal{U};\mathcal{F})$ and $\delta^{k}:\check{C}^{k}(\mathcal{U};\mathcal{F})\to \check{C}^{k+1}(\mathcal{U};\mathcal{F})$ by

\begin{equation*}
    \check{C}^{k}(\mathcal{U};\mathcal{F})=\prod_{(\alpha_{0}, \ldots , \alpha_{k})}\mathcal{F}(U_{\alpha_0\cdots \alpha_k})
\end{equation*}

\begin{equation*}
    (\delta^{k}\Omega)_{\alpha_{0}\cdots\alpha_{k+1}} = \sum_{j=0}^{k+1}(-1)^{j}\Omega_{\alpha_{0}\cdots\widehat{\alpha_{j}}\cdots \alpha_{k+1}}|_{U_{\alpha_{0}\cdots \alpha_{k+1}}} 
\end{equation*}
where $U_{\alpha_0\cdots \alpha_k} = U_{\alpha_{0}}\cap\ldots \cap U_{\alpha_{k}}$ and $\alpha_0 \cdots \widehat{\alpha_j}\cdots \alpha_{k+1}$
is the result of removing $\alpha_j$ from the original tuple. 
Then $(\check{C}^{*}(\mathcal{U};\mathcal{F}),\delta^{*})$ is a cochain complex, and the associated cohomology groups $\check{H}^*(\mathcal{U}; \mathcal{F})$ are the \textbf{\v{C}ech cohomology subordinate to $\mathcal{U}$ with coefficients in $\mathcal{F}$}. \\

    \noindent Note that if $\mathcal{F}$ is a presheaf of non-abelian groups, we can still define $\check{C}^{k}(\mathcal{U};\mathcal{F})$ and $\delta^{k}$ as above, replacing the minus sign by group inversion in the coboundary formula. In general, however, we do not have $ \text{im }\delta^{k} \subseteq \ker\delta^{k+1}$ so  $\check{H}^{k}(\mathcal{U};\mathcal{F})$ is not well-defined for $k > 0$. For $k=1$, we still have a well-defined cohomology \textit{set} $\check{H}^{1}(\mathcal{U};\mathcal{F})$, though it does not have a natural group structure.\\ 

    \noindent Let $\mathcal{F}$ be a presheaf of groups over $B$, and let $\mathcal{U}$ be an open cover of $B$. 
    If $\mathcal{V}$ is a refinement of $\mathcal{U}$, one has an induced map $\check{H}^{*}(\mathcal{U};\mathcal{F})\to\check{H}^{*}(\mathcal{V};\mathcal{F})$ in cohomology. 
    The open covers of $B$ form a partially ordered set with respect to refinements, so the corresponding \v{C}ech cohomology groups (or sets) form a directed system (of groups or sets, respectively). 
    The direct limit is denoted $\check{H}^{*}(B;\mathcal{F})$ and called the $\textbf{\v{C}ech cohomology of $B$}$ with coefficients in $\mathcal{F}$. 
    If $B$ is paracompact, then   the Fiber Bundle Construction lemma implies a one-to-one correspondence between $\check{H}^{1}(B;\mathcal{C}_{\text{Homeo}(F)})$ and $\text{Bun}_{F}(B)$, where $\mathcal{C}_{G}$ denotes the sheaf of continuous $G$-valued functions on $B$ for a topological group $G$. 
    That is, for $U\subseteq B$ open, $\mathcal{C}_G(U) = \text{Maps}(U,G)$ with  multiplication in $G$.\\ 
    
    \noindent We state the following additional properties of \v{C}ech cohomology without proof. For further details, see~\cite{bredon1997sheaf}.\\

\begin{proposition}\label{prop: Cech properties}
    \v{C}ech cohomology has the following properties: \\

    \begin{enumerate}
        \item If $\mathcal{U}$ is a good open cover of $B$, then $\check{H}^{*}(\mathcal{U};\mathcal{C}_{G})\cong \check{H}^{*}(B;\mathcal{C}_{G})$ for any topological group $G$.\\
        
        \item If $B$ has the homotopy type of a CW complex and $G$ is an abelian group, then there is an isomorphism between \v{C}ech   and cellular cohomology,  $\check{H}^{*}(B;\underline{G})\cong H^{*}(B;G)$, where $\underline{G}$ denotes the sheaf of locally constant $G$-valued functions on $B$.\\

        \item If $\mathcal{F}$ is a flabby sheaf over a paracompact space $B$, then $\check{H}^{k}(B;\mathcal{F})=0$ for all $k > 0$.\\

        \item If $\mathcal{F}$ is a sheaf of abelian groups  over a paracompact space $B$, then there is an isomorphism between sheaf and \v{C}ech cohomology, $H^{k}(B;\mathcal{F})\cong \check{H}^{k}(B;\mathcal{F})$, for all $k\geq 0$.\\

        \item If $\varphi:G\to H$ is a continuous homomorphism of topological groups, there is an induced cochain map $\varphi^{\#}:\check{C}^{*}(\mathcal{U};\mathcal{C}_{G})\to \check{C}^{*}(\mathcal{U};\mathcal{C}_{H})$ given by $(\varphi^{\#}\Omega)_{\alpha_{0}\cdots\alpha_{k}}=\varphi(\Omega_{\alpha_{0}\cdots\alpha_{k}})$. Moreover, if $\varphi$ is a homotopy equivalence, then $\varphi^{\#}$ induces an isomorphism in cohomology. \\

        \item If $F\to G\to H$ is a central extension of topological groups, there is an induced short exact sequence $\mathcal{C}_{F}\to \mathcal{C}_{G}\to \mathcal{C}_{H}$ of sheaves, which in turn gives rise to an exact sequence in \v{C}ech cohomology. 
        Moreover, if $F$ is abelian, there is an induced connecting homomorphism $\Delta:\check{H}^{1}(B;\mathcal{C}_{H})\to \check{H}^{2}(B;\mathcal{C}_{F})$ which yields a long exact sequence.\\
    \end{enumerate}
    
\end{proposition}

\noindent \textbf{Pullback Bundles:} In addition to associated bundles, another key way to obtain new fiber bundles from existing ones is via the pullback construction. Given a fiber bundle $\xi = (E,\pi,B)$ and a continuous map $f:B'\to B$, the \textbf{pullback bundle} $f^{*}\xi$ is the triple $(f^{*}E, \pi', B')$, where

\begin{equation*}
    f^{*}E = \{(b',e)\in B'\times E: \ f(b') = \pi(e)\} \\ 
\end{equation*}    

\noindent and $\pi':f^{*}E\to B'$ is given by $\pi'(b',e) = b'$. Note that if $\xi$ has fiber $F$, then so does $f^{*}\xi$. Moreover, one can show that homotopic maps yield isomorphic pullback bundles.\\

\noindent \textbf{Universal Bundles:} A $\textbf{universal $G$-bundle}$ is a principal $G$-bundle $(EG,p, BG)$ such that $EG$ is (weakly) contractible. The space $BG$ is called a $\textbf{classifying space}$ for $G$. It can be shown that a universal $G$-bundle exists for any topological group $G$. The construction is essentially unique in the sense that any two universal $G$-bundles are isomorphic and any two classifying spaces for $G$ are (weakly) homotopy equivalent.\\

\noindent For any paracompact space $B$, there is a natural bijective correspondence between $[B,BG]$, the set of homotopy classes of maps from $B$ to $BG$, and $\text{Prin}_{G}(B)$.
The correspondence is given by $[f] \mapsto [f^{*}\xi_{G}]$, where $\xi_{G}$ is a universal $G$-bundle. Given any $[\xi]\in \text{Prin}_{G}(B)$, a map $f:B\to BG$ such that $\xi \cong f^{*}\xi_{G}$ is called a \textbf{classifying map} for $\xi$. \\

\noindent Lastly, we note that the classifying space construction can be made functorial in the following sense: if $\varphi:G\to H$ is a continuous homomorphism of topological groups, there is an induced map $B\varphi:BG\to BH$ with the property that if $f:B\to BG$ classifies a principal $G$-bundle $\xi$, then $B\varphi\circ f:B\to BH$ classifies $\varphi^{*}\xi$.\\

\subsection{Cohomology With Local Coefficients}

We provide a brief overview of cohomology with local coefficients which we use in the classification of circle bundles. For additional details, see \cite[Chapter V]{eilenberg1947homology} or \cite[Section 3.I]{hatcher2002algebraic}.\\  

\begin{definition}\label{def: local systems}
    Given a topological space $B$ and an abelian group $G$, a \textbf{local system} for $B$ with coefficients in $G$ is a locally constant sheaf $\mathcal{L}$ with stalk $G$.\\
\end{definition}

\noindent Note that if $B$ is path-connected and has a universal cover $p:\widetilde{B}\to B$, then a local system on $B$ can equivalently be defined as a sheaf $\mathcal{L}$ with stalk $G$ such that the pullback $p^{*}\mathcal{L}$ is a constant sheaf on $\widetilde{B}$. To see this, note that if $\mathcal{L}$ is a locally-constant sheaf, then $p^{*}\mathcal{L}$ is also locally-constant, but a locally-constant sheaf on a simply-connected space is constant. Conversely, if $\mathcal{L}$ is a sheaf on $B$ and $p^{*}\mathcal{L}$ is constant, then clearly $\mathcal{L}$ is locally-constant.\\

\noindent We will primarily be interested in the case where $G$ is a discrete group and $B$ a path-connected paracompact space. 
In this setting, a local system on $B$ admits several equivalent characterizations:\\

\begin{proposition}\label{prop: local systems}
    Suppose $G$ is a discrete group and $B$ is a path-connected, locally path-connected and semilocally simply connected space. 
    Endow the automorphism group $\text{Aut}(G)$ with the discrete topology, and let $\underline{\text{Aut}(G)}$ denote the associated locally-constant sheaf  on $B$. There are bijective correspondences between: 
    
\begin{enumerate}
    \item Isomorphism classes of local systems on $B$ with stalk $G$,\\

    \item Isomorphism classes of flat fiber bundles on $B$ with fiber $G$ and structure group $\text{Aut}(G)$,\\

    \item Conjugacy classes of homomorphisms $\rho:\pi_{1}(B)\to \text{Aut}(G)$ (called \textbf{monodromy representations}),\\

    \item Cohomology classes $[\omega]$ in the \v{C}ech cohomology set $\check{H}^{1}(B;\underline{\text{Aut}(G)})$.\\
\end{enumerate}

\end{proposition}

\begin{proof}

    \noindent \textbf{1 $\leftrightarrow $ 2:} Suppose $\mathcal{L}$ is a local system over $B$ with stalk $G$. Choose an open cover $\mathcal{U} = \{U_{j}\}_{j\in J}$ for $B$ and isomorphisms $\{\phi_{j}:\mathcal{L}|_{U_{j}}\to \underline{G}\}_{j\in J}$. Each map $\phi_{j}\circ\phi_{k}^{-1}:\underline{G}\to \underline{G}$ is locally-constant (hence constant) on $\underline{G}$, thus determines an element $g_{jk}\in\text{Aut}(G)$. The $g_{jk}$'s satisfy the cocycle condition, so the collection of all these maps can be interpreted as a \v{C}ech cocycle $g\in \check{Z}^{1}(\mathcal{U};\underline{\text{Aut}(G)})$. By the Fiber Bundle Construction Lemma, $g$ determines a flat fiber bundle with fiber $G$ and structure group $\text{Aut}(G)$ (and furthermore, the isomorphism class of the bundle does not depend on the choice of trivializing cover). Conversely, given a flat fiber bundle $\xi$ with fiber $G$ and structure group $\text{Aut}(G)$, the sheaf of local sections of $\xi$ is a local system on $B$ with stalk $G$. One can check that these constructions are inverses up to isomorphism. \\
    
    \noindent \textbf{2 $\leftrightarrow$ 3:} Suppose we have a representation $\rho:\pi_{1}(B)\to \text{Aut}(G)$. The universal cover $p:\widetilde{B}\to B$ of $B$ is a principal $\pi_{1}(B)$-bundle, where $\pi_{1}(B)$ acts on $\widetilde{B}$  from the left by deck transformations (and taking inverses yields the corresponding right action). 
    In particular, this bundle is flat since $\pi_{1}(B)$ is a discrete group. Since $\rho$ gives a left $\pi_{1}(B)$-action on $G$ by automorphisms, so we can form the associated bundle $\widetilde{B}\times_{\rho}G\to B$. Observe that if $\rho':\pi_{1}(B)\to \text{Aut}(G)$ is another representation in the same conjugacy class as $\rho$, then by the Fiber Bundle Construction Lemma, the resulting bundles are isomorphic. \\
    
    \noindent Conversely, let $\mathcal{L}$ be a local system on $B$ with stalk $G$. Choose any $b_{0}\in B$ and isomorphism $\mathcal{L}_{b_{0}}\cong G$, as well as a loop $\gamma:I\to B$ based at $b_{0}$. Since $I = [0,1]$ is contractible, the sheaf $\gamma^{*}\mathcal{L}$ is constant, and the restriction maps induce isomorphisms $(\gamma^{*}\mathcal{L})_{0}\to (\gamma^{*}\mathcal{L})(I)\to (\gamma^{*}\mathcal{L})_{1}$. We therefore obtain an induced map $T_{\gamma}\in \text{Aut}(G)$ via our identification of $(\gamma^{*}\mathcal{L})_{1} = (\gamma^{*}\mathcal{L})_{0} = \mathcal{L}_{b_{0}}$ with $G$. One can check that the automorphisms induced by homotopic loops coincide, and that the assignment $[\gamma]\longmapsto T_{\gamma}$ defines a representation $\rho:\pi_{1}(B,b_{0})\to \text{Aut}(G)$. Note that if $b_{1}\in B$, there exists a path $\widetilde{\gamma}:[0,1]\to B$ from $b_{0}$ to $b_{1}$, and we have $\pi_{1}(B,b_{1})\cong \pi_{1}(B,b_{0})$ via the assignment $[\gamma]\longmapsto [\widetilde{\gamma} *\gamma*\widetilde{\gamma}^{-1}]$. Similarly, different identifications $\mathcal{L}_{b_{0}}\cong G$ yield homomorphisms which differ by conjugation. Thus, an isomorphism class of local systems with stalk $G$ determines a representation $\rho:\pi_{1}(B)\to \text{Aut}(G)$ up to conjugacy.\\

    \noindent \textbf{2 $\leftrightarrow$ 4:} This equivalence is immediately implied by the Fiber Bundle Construction Lemma. \\
\end{proof}

\noindent We will primarily be interested in local systems with stalk $\mathbb{Z}$: \\

\begin{definition}\label{def: orientation sheaf}
    A local system with stalk $\mathbb{Z}$ is called an \textbf{orientation sheaf}.
    Given a path-connected, semilocally simply connected space $B$, an open cover $\mathcal{U}$ of $B$ and a cocycle $\omega\in \check{Z}^{1}(\mathcal{U};\underline{\mathbb{Z}_{2}})$ (where $\mathbb{Z}_2 = \{1,-1\} \cong \mathrm{Aut}(\mathbb{Z}))$, we will denote the orientation sheaf corresponding to $\omega$ by $\mathbb{Z}_{\omega}$.\\
\end{definition}

\begin{remark}
    If $B$ is a smooth manifold, the orientation sheaf corresponding to any representative of the Stiefel-Whitney class $w_{1}(TB)\in \check{H}^{1}(B;\underline{\mathbb{Z}_{2}})$  of its tangent bundle, is called an orientation sheaf for the manifold. 
    By Proposition~\ref{prop: local systems}, $B$ has an orientation sheaf $\mathcal{O}$, unique up to isomorphism.\\
\end{remark}

\noindent Let $\mathcal{U}$ be an open cover of a path-connected, locally path-connected, semilocally simply connected space $B$, and let $\mathbb{Z}_{\omega}$ denote the orientation sheaf on $B$ associated with a cocycle $\omega\in\check{Z}^{1}(\mathcal{U};\underline{\mathbb{Z}_{2}})$. Proposition~\ref{prop: local systems} implies we can think of $\mathbb{Z}_{\omega}$ as the sheaf of locally-constant sections of an associated bundle $B_{\omega}\times_{\mathbb{Z}_{2}}\mathbb{Z}$, where $B_{\omega}$ is a double cover of $B$ and $B_{\omega}\to B$ is a principal $\mathbb{Z}_{2}$-bundle. 
More generally, given a topological group $G$ with a $\mathbb{Z}_{2}$-action, we will be interested in sheaves of sections of associated bundles of the form $B_{\omega}\times_{\mathbb{Z}_{2}}G$:\\

\begin{definition} \label{def: twisted sheaf}
    Let $\omega\in \check{Z}^{1}(\mathcal{U};\underline{\mathbb{Z}_{2}})$ for some open cover $\mathcal{U}$ of a paracompact space $B$. Given a topological group $G$ with a $\mathbb{Z}_{2}$-action, let $\mathcal{P}_{G,\omega}$ be the presheaf over $B$ defined by

\begin{equation*}
    \mathcal{P}_{G,\omega}(V) = \{\text{families } \{f_{j}\in \mathcal{C}_{G}(U_{j}\cap V)\}: \ f_{j}|_{U_{jk} \cap V} = \omega_{jk}\cdot f_{k}|_{U_{jk} \cap V}\}
\end{equation*}

    \noindent The sheaf $\mathcal{C}_{G,\omega}$ of \textbf{continuous $G$-valued functions on $B$ twisted by $\omega$} is the sheafification of $\mathcal{P}_{G,\omega}$. Note that $\mathcal{C}_{G,\omega}$ is equivalently defined as the sheaf of sections of the associated bundle $B_{\omega}\times_{\mathbb{Z}_{2}} G$.\\
    
    \noindent \textbf{Notation:} If $G$ is a discrete group, we will write $G_{\omega}$ instead of $\mathcal{C}_{G,\omega}$.  \\
\end{definition}

\begin{proposition}\label{prop: isomorphic sheaves}
    Suppose $\omega,\omega'\in \check{Z}^{1}(\mathcal{U};\underline{\mathbb{Z}_{2}})$. Then, for any topological group $G$ with left $\mathbb{Z}_{2}$-action, the sheaves $\mathcal{C}_{G,\omega}$ and $\mathcal{C}_{G,\omega'}$ are isomorphic if any only if $[\omega] = [\omega']$. Explicitly, if  $\omega = \omega'\cdot \delta^{0}\phi$, then the map which sends $f_{j}\in\mathcal{C}_{G,\omega}(U_{j})$ to $\phi_{j}\cdot f_{j}\in\mathcal{C}_{G,\omega'}(U_{j})$ is an isomorphism of sheaves. \\
\end{proposition}

\begin{proof}
    Let $\{f_{j}\in \mathcal{C}_{G}(U_{j})\}_{j\in J}$ such that $f_{j}|_{U_{jk} } = \omega_{jk}\cdot f_{k}|_{U_{jk} }$ whenever $U_{j}\cap U_{k}$ is nonempty, and let $\{f'_{j}=\phi_{j}\cdot f_{j}\in\mathcal{C}_{G}(U_{j})\}_{j\in J}$. Since $\omega = \omega'\cdot \delta^{0}\phi$, we have 

\begin{equation*}
    \tilde{f}_{j}|_{U_{jk} }=\phi_{j}\cdot f_{j}|_{U_{jk} } = (\phi_{j}\omega_{jk})\cdot f_{k}|_{U_{jk} }=\phi_{k}\cdot f_{k}|_{U_{jk} } = \tilde{f}_{k}|_{U_{jk} }
\end{equation*}

\noindent whenever $U_{j}\cap U_{k}$ is nonempty. It follows that the assignment $\{f_{j}\}_{j\in J}\longmapsto\{\tilde{f}_{j}\}_{j\in J}$ induces a sheaf morphism $\mathcal{C}_{G,\omega}\to \mathcal{C}_{G,\omega'}$. Since the map on each open set is clearly a bijection, the resulting sheaf morphism is an isomorphism.\\
\end{proof}

\noindent For $G$ abelian, a base space $B$ and a local system $\mathcal{L}$ with stalk $G$, one can define cohomology groups $H^{k}(B;\mathcal{L})$, called $\textbf{cohomology with local coefficients}$ in several ways, which are equivalent under mild assumptions. Most directly (though most abstractly), $H^{k}(B;\mathcal{L})$ can be interpreted as sheaf cohomology groups defined in terms of derived functors. If $B$ is paracompact, then by Proposition~\ref{prop: Cech properties} we have $\check{H}^{k}(\mathcal{U};\mathcal{L})\cong H^{k}(B;\mathcal{L})$ for any good open cover $\mathcal{U}$ of $B$, so we can compute these groups with \v{C}ech cohomology. \\

\noindent One can also define a notion of singular (co)homology with local coefficients (in fact, any (co)homology theory can be promoted to a theory with local coefficients). 
The details are not important for us, but we note that when $B$ is paracompact, the singular and \v{C}ech cohomology groups with coefficients in $\mathcal{L}$ are isomorphic. 
We also have the following generalization of Poincar\'e duality  \cite[Theorem 2.4]{tshishiku-poincare-local-coeffs}: \\

\begin{proposition}
\label{prop: TwistedPoincare}
    Let $M$ be a closed connected $n$-manifold, and let $\mathcal{L}$ be a local system of abelian groups. Then $H^{k}(M;\mathcal{L})\cong H_{n-k}(M;\mathcal{L}\otimes\mathcal{O})$ for all $k$, where $\mathcal{O}$ denotes the orientation sheaf of $M$.\\ 
\end{proposition}

\subsection{Classification Of Circle Bundles}\label{sec: True Circle Bundles}

In this paper, we use the term \textbf{circle bundle} to mean a fiber bundle with fiber $\mathbb{S}^{1}$ (note that some authors reserve this term for \emph{principal} $\mathbb{S}^{1}$-bundles). 
A priori, the structure group for a circle bundle is $\text{Homeo}(\mathbb{S}^{1})$.
We have the following two central extensions of topological groups: 

\begin{equation}\label{eq: Homeo+ central ext}
    1\to \text{Homeo}_{+}(\mathbb{S}^{1})\to \text{Homeo}(\mathbb{S}^{1})\xrightarrow{\deg} \mathbb{Z}_{2}\to 1
\end{equation}

\begin{equation}\label{eq: exponential sequence}
    1\to \mathbb{Z}\to \widetilde{\text{Homeo}}_{+}(\mathbb{S}^{1})\xrightarrow{\widetilde{\exp}} \text{Homeo}_{+}(\mathbb{S}^{1})\to 1\\
\end{equation}

\noindent Here $\text{Homeo}_{+}(\mathbb{S}^{1})$ denotes the group of orientation-preserving homeomorphisms of $\mathbb{S}^{1}$, and $\widetilde{\text{Homeo}}_{+}(\mathbb{S}^{1})$ denotes the universal cover of $\text{Homeo}_{+}(\mathbb{S}^{1})$, which can be modeled explicitly as 

\begin{equation}
    \widetilde{\text{Homeo}}_{+}(\mathbb{S}^{1}) = \{f\in\text{Homeo}_{+}(\mathbb{R}): \ f(x+1)=f(x)+1 \ \text{ for all }x\}
\end{equation}

\noindent The projection map $\widetilde{\exp}:\widetilde{\text{Homeo}}_{+}(\mathbb{S}^1)\to \text{Homeo}_{+}(\mathbb{S}^{1})$ is defined by $\widetilde{\exp}(f) = \text{exp}\circ  f \circ \log$, where $\exp:\mathbb{R}\to \mathbb{S}^{1}$ is given by $\exp(\theta) = e^{2\uppi i\theta}$ and $\log$ denotes $\frac{1}{2\uppi}$ times any branch of the complex logarithm. \\

\noindent Notice that sequence~\eqref{eq: Homeo+ central ext} splits, and the section $s:\mathbb{Z}_{2}\to \text{Homeo}(\mathbb{S}^{1})$ induced by complex conjugation, gives rise to a semidirect decomposition $\text{Homeo}(\mathbb{S}^{1})\cong \text{Homeo}_{+}(\mathbb{S}^{1})\rtimes_{s}\mathbb{Z}_{2}$. Explicitly, we obtain a right $\mathbb{Z}_{2}$-action on $\text{Homeo}_{+}(\mathbb{S}^{1})$ via the assignment $\pm 1\longmapsto \text{Adj}_{s(\pm 1)}$ defined by $\text{Adj}_{s(\pm 1)}(\Lambda) = s(\pm1)\Lambda s(\pm1)^{-1}$. This induces a group multiplication law $(\Lambda_{1}, \omega_{1})\cdot (\Lambda_{2},\omega_{2}) = (\Lambda_{1}\text{Adj}_{s(\omega_{1})}(\Lambda_{2}),\omega_{1}\omega_{2})$. The map $\varphi:\text{Homeo}_{+}(\mathbb{S}^{1})\rtimes_{s}\mathbb{Z}_{2}\to \text{Homeo}(\mathbb{S}^{1})$ given by $\varphi(\Lambda,\omega)=\Lambda s(\omega)$ is then a group isomorphism. 

\noindent Thus, for any open cover $\mathcal{U}$ of a paracompact space $B$, $\varphi$ induces an isomorphism on cohomology, where the cocycle condition $(\Lambda_{jk},\omega_{jk})\cdot (\Lambda_{kl},\omega_{kl}) = (\Lambda_{jl},\omega_{jl})$  is given explicitly by
\begin{equation*}
\Lambda_{jk}\text{Adj}_{r(\omega_{jk})}(\Lambda_{kl}) = \Lambda_{jl}\quad\quad \omega_{jk}\omega_{kl}=\omega_{jl}\\
\end{equation*}

\noindent Consequently, for any $\omega\in\check{Z}^{1}(\mathcal{U};\underline{\mathbb{Z}_{2}})$, there is a natural bijection between cocycles $\Omega\in\check{Z}^{1}(\mathcal{U};\mathcal{C}_{\text{Homeo}(\mathbb{S}^{1})})$ satisfying $\deg_{*}(\Omega) = \omega$ and cochains $\Lambda\in\check{C}^{1}(\mathcal{U};\mathcal{C}_{\text{Homeo}_{+}(\mathbb{S}^{1})})$ satisfying $\Lambda_{jk}\text{Adj}_{r(\omega_{jk})}(\Lambda_{kl}) = \Lambda_{jl}$ for all $(jkl)\in\mathcal{N}(\mathcal{U})$. Alternatively, we can interpret $\omega$ as a local system and $\Lambda$ as a \v{C}ech cocycle in $\check{Z}^{1}(\mathcal{U};\mathcal{C}_{\text{Homeo}_{+}(\mathbb{S}^{1}),\omega})$:\\

\begin{equation}
    \check{Z}^{1}(\mathcal{U};\mathcal{C}_{\text{Homeo}(\mathbb{S}^{1})})=\coprod_{\omega\in\check{Z}^{1}(\mathcal{U};\underline{\mathbb{Z}_{2}})}\check{Z}^{1}(\mathcal{U};\mathcal{C}_{\text{Homeo}_{+}(\mathbb{S}^{1}),\omega})\\
\end{equation}

\noindent The advantage of this perspective is suggested by the following proposition: \\

\begin{proposition}\label{prop: twisted exp sequence}
    Suppose $\omega\in\check{Z}^{1}(\mathcal{U};\underline{\mathbb{Z}_{2}})$ for some open cover $\mathcal{U}$ of a paracompact space $B$. The following sequence of sheaves is short exact:

\begin{equation}\label{eq: twisted exp seq}
    1\to \mathbb{Z}_{\omega}\to \mathcal{C}_{\widetilde{\text{Homeo}}_{+}(\mathbb{S}^{1}),\omega}\xrightarrow{\widetilde{\exp}_{*}} \mathcal{C}_{\text{Homeo}_{+}(\mathbb{S}^{1}),\omega}\to 1
\end{equation}

    \noindent Moreover, the sheaf $\mathcal{C}_{\widetilde{\text{Homeo}}_{+}(\mathbb{S}^{1}),\omega}$ is soft, so the induced connecting homomorphism \[\Delta_{\omega}:\check{H}^{1}(\mathcal{U};\mathcal{C}_{\text{Homeo}_{+}(\mathbb{S}^{1}),\omega})\to \check{H}^{2}(\mathcal{U};\mathbb{Z}_{\omega})\] is an isomorphism. \\
\end{proposition}

\begin{proof}
    The $\mathbb{Z}_{2}$-actions on $\mathbb{Z}$, $\widetilde{\text{Homeo}}_{+}(\mathbb{S}^{1})$ and $\text{Homeo}_{+}(\mathbb{S}^{1})$  induce actions on the sheaves $\underline{\mathbb{Z}}$, $\mathcal{C}_{\widetilde{\text{Homeo}}_{+}(\mathbb{S}^{1})}$ and $\mathcal{C}_{\text{Homeo}_{+}(\mathbb{S}^{1})}$, repectively. Moreover, $\widetilde{\exp}$ is equivariant with respect to the $\mathbb{Z}_{2}$ actions on $\widetilde{\text{Homeo}}_{+}(\mathbb{S}^{1})$ and $\text{Homeo}_{+}(\mathbb{S}^{1})$ in the sense that $\widetilde{\exp}(\lambda \cdot \tilde{f}) = \lambda\cdot \widetilde{\exp}(\tilde{f})$ for any $\tilde{f}\in\widetilde{\text{Homeo}}_{+}(\mathbb{S}^{1})$ and $\lambda\in\{\pm 1\}=\mathbb{Z}_{2}$  (the inclusion of $\mathbb{Z}$ into $\widetilde{\text{Homeo}}_{+}(\mathbb{S}^{1})$ is also equivariant in the same sense). Thus, it descends to a map between the sheaves $\mathcal{C}_{\widetilde{\text{Homeo}}_{+}(\mathbb{S}^{1}),\omega}$ and $\mathcal{C}_{\text{Homeo}_{+}(\mathbb{S}^{1}),\omega}$. \\
    
    \noindent At the level of stalks, sequence~\eqref{eq: twisted exp seq} reduces to exact sequence~\eqref{eq: exponential sequence}, hence it is also exact as a sequence of sheaves of abelian groups.
    By Proposition~\ref{prop: Cech properties}, we therefore obtain a long exact sequence in \v{C}ech cohomology, and in particular, a connecting homomorphism $\Delta_{\omega}:\check{H}^{1}(\mathcal{U};\mathcal{C}_{\text{Homeo}_{+}(\mathbb{S}^{1}),\omega})\to \check{H}^{2}(\mathcal{U};\mathbb{Z}_{\omega})$. 
    Moreover, $\widetilde{\text{Homeo}}_{+}(\mathbb{S}^{1})$ is contractible, so  $\mathcal{C}_{\widetilde{\text{Homeo}}_{+}(\mathbb{S}^{1}),\omega}$  is soft, and Proposition~\ref{prop: Cech properties} implies $\Delta_{\omega}$ is an isomorphism by exactness.\\
\end{proof}

\noindent The proposition above immediately leads to the main classification result of this section: \\

\begin{theorem}\label{prop: general class thm}
    
    For any open cover $\mathcal{U}$ of a paracompact space $B$, there is a natural bijective correspondence 

\begin{equation*}
    \check{H}^{1}(\mathcal{U};\mathcal{C}_{\text{Homeo}(\mathbb{S}^{1})})\ \cong \ \coprod_{[\omega]\in\check{H}^{1}(\mathcal{U};\underline{\mathbb{Z}}_{2})}\check{H}^{2}(\mathcal{U};\mathbb{Z}_{\omega})
\end{equation*}    
    
\noindent via the assignment $[\Omega] \longmapsto \left(\twist(\Omega),\tilde{e}(\Omega)\right)$
given by $\twist(\Omega) = [\deg_{*}(\Omega)]$ and $\tilde{e}(\Omega) = \Delta_{\omega}([\Omega])$.\\

\end{theorem}

\noindent We therefore have a complete classification of circle bundles over paracompact spaces in terms of a pair of discrete invariants. As we will later show, these two invariants are closely connected to the first Stiefel-Whitney class and Euler class for real plane bundles, so we will refer to $\twist(\Omega)$ and $\tilde{e}(\Omega)$ respectively as the \textbf{Stiefel-Whitney class} (or \textbf{orientation class}) and the \textbf{twisted Euler class} of the circle bundle represented by $\Omega$.\\

\begin{remark}
    The notation in Theorem~\ref{prop: general class thm} is slightly abusive, since each cocycle $\omega$ within a cohomology class determines a different local system $\mathbb{Z}_{\omega}$. On the other hand, by Proposition~\ref{prop: isomorphic sheaves}, the local systems determined by any two cohomologous cocycles are isomorphic. In particular, if $\omega$ is a coboundary, then $\mathbb{Z}_{\omega}\cong\underline{\mathbb{Z}}$. \\
\end{remark}

\noindent Observe that when $B$ is a locally finite 1-dimensional CW complex (e.g., a graph), then Theorem~\ref{prop: general class thm} implies that circle bundles over $B$ are completely classified by the first Stiefel-Whitney class. In the case when $B$ is a compact connected surface, the classification is also especially simple and explicit: \\

\begin{proposition}\label{prop: 2-manifold class}
Let \( B \) be a closed connected 2-manifold, and let \( [\omega] \in H^1(B; \mathbb{Z}_2) \). Then 

\begin{equation*}
    H^{2}(B;\mathbb{Z}_{\omega})=\begin{cases}
    \mathbb{Z}, & \text{if } [\omega] = \twist(B), \\
    \mathbb{Z}_{2}, & \text{otherwise}.\\
  \end{cases}
\end{equation*}

\noindent We therefore have the following classification of circle bundles over $B$:

\begin{equation*}
    \text{Bun}_{\mathbb{S}^{1}}(B)=\begin{cases}
    \mathbb{Z}\ \coprod \ (\mathbb{Z}_{2})^{2^{2g}-1}, & \text{if } B=\Sigma_{g}\text{ (orientable surface of genus }g), \\
    \mathbb{Z}\ \coprod (\mathbb{Z}_{2})^{2^{k}-1}, & \text{if } B = \#_{k}\mathbb{RP}^{2} \text{ (non-orientable case)}\\
  \end{cases}
\end{equation*}
\end{proposition}

\begin{proof}
    Twisted Poincare duality \ref{prop: TwistedPoincare} asserts that for any compact connected $n$-manifold $M$ and locally constant sheaf $\mathcal{L}$ of Abelian groups on $M$ (with fiber $G$), one has $H^{*}(M;\mathcal{L})\cong H_{n-*}(M;\mathcal{L}\otimes\mathcal{O})$, where $\mathcal{O}$ denotes the orientation sheaf of $M$. 
    Moreover, $H_{0}(M;\mathcal{L})$ is isomorphic to the coinvariants of the monodromy action $\omega_{\mathcal{L}}:\pi_{1}(M)\to \text{Aut}(G)$ induced by $\mathcal{L}$, i.e., 

\begin{equation*}
    H_{0}(M;\mathcal{L})\cong G/\langle \omega_{\mathcal{L}}(\gamma)\cdot a-a \ | \ a\in G, \gamma\in\pi_{1}(M)\rangle
\end{equation*}
     
    \noindent Thus, for any compact connected surface $B$ and orientation character $[\omega]\in H^{1}(B;\mathbb{Z}_{2})$, one has 

\begin{equation*}
    H^{2}(B;\mathbb{Z}_{\omega})=\begin{cases}
    \mathbb{Z}, & \text{if } \mathbb{Z}_{\omega}\otimes\mathcal{O}\cong\underline{\mathbb{Z}}, \\
    \mathbb{Z}_{2}, & \text{otherwise}.\\
  \end{cases}
\end{equation*}

\noindent Since $\mathbb{Z}_{\omega}\otimes\mathcal{O}\cong\underline{\mathbb{Z}}$ if and only if $[\omega] + \twist (B) = 0$ (or, equivalently, $[\omega] = \twist(B)^{-1} = \twist(B)$), the proposition follows. \\
\end{proof}

\noindent In light of the proposition above, for any compact connected surface $B$ and $[\omega]\in H^{1}(B;\mathbb{Z}_{2})$, one can define a \textbf{twisted fundamental class} $[B]_{\omega}$ which generates $H_{2}(B;\mathbb{Z}_{\omega})$ -- this choice is unique when $[\omega]\neq \twist(B)$ and unique up to sign when $[\omega] = \twist(B)$. 
The pairing $\langle\tilde{e},[B]_{\omega}\rangle$ then gives an explicit bijection between $\check{H}^{2}(B;\mathbb{Z}_{\omega})$ and its image in $\mathbb{Z}$. We call the integer $\langle\tilde{e}(\Omega),[B]_{\omega}\rangle$ the $\textbf{twisted Euler number}$ of the bundle classified by $\Omega$. \\

\subsection{Reduction Of Structure Group}

The group $\text{Homeo}(\mathbb{S}^{1})$ is    disconnected, non-abelian and infinite-dimensional, making it challenging to work with from a computational perspective. 
In this section, we identify reductions of this structure group to simpler ones.\\

\begin{definition}
    We say that a circle bundle $\xi$ is \textbf{orientable} if any of the following equivalent conditions are met: \\

\begin{enumerate}
    \item $\xi$ admits a family of local trivializations with $\text{Homeo}_{+}(\mathbb{S}^{1})$-valued transition maps,\\

    \item The cohomology class $[\Omega_{\xi}]\in\check{H}^{1}(B;\mathcal{C}_{\text{Homeo}(\mathbb{S}^{1})})$ lifts to a class in $\check{H}^{1}(B;\mathcal{C}_{\text{Homeo}_{+}(\mathbb{S}^{1})})$,\\

    \item The orientation class $\twist([\Omega_{\xi}])\in\check{H}^{1}(B;\underline{\mathbb{Z}_{2}})$ vanishes.\\
\end{enumerate}
    
\end{definition}

 The next proposition implies another useful characterization of orientable circle bundles: \\

\begin{proposition}\label{prop: homotopy S1 to Homeo+}
    The inclusion $\iota:\mathbb{S}^{1}\to \text{Homeo}_{+}(\mathbb{S}^{1})$ given by multiplication in $\mathbb{S}^1$ is a homotopy equivalence. 
\end{proposition}

\begin{proof}
    Consider the map $\tau:\text{Homeo}_{+}(\mathbb{S}^{1})\to \mathbb{S}^{1}$ defined by $\tau(f) = f(1)$. Observe that $\tau$ is a continuous surjection, and the fiber of $\tau$ over each $z\in\mathbb{S}^{1}$ is comprised of the orientation-preserving homeomorphisms which map $1$ to $z$. 
    One can  check that each $\tau^{-1}(z)$ is contractible, and a theorem of Smale~\cite{Smale1957Vietoris} then implies that $\tau$ is a homotopy inverse for $\iota$.\\
\end{proof}

\noindent An immediate consequence of the proposition above is that one can always reduce the structure group of an orientable circle bundle $\xi$ to $\mathbb{S}^{1}$ itself, which means every orientable circle bundle admits a principal $\mathbb{S}^{1}$-structure. In fact, one always has two non-isomorphic choices for the principal structure on $\xi$ corresponding to the two choices of orientation, though the underlying fiber bundles are isomorphic. This is consistent with our earlier observation that the (twisted) Euler number of a circle bundle is only well-defined up to sign until an orientation class representative for the bundle is specified. \\

\noindent More generally, Proposition~\ref{prop: homotopy S1 to Homeo+} implies that the inclusion of $\mathbb{S}^{1}\rtimes\mathbb{Z}_{2} \cong \text{Iso}(\mathbb{S}^{1})$ into $\text{Homeo}(\mathbb{S}^{1})$ is a homotopy equivalence (where $\text{Iso}(\mathbb{S}^{1})$ denotes isometries of $\mathbb{S}^{1}$ with respect to geodesic distance). It follows that over paracompact spaces, any circle bundle, no matter how wild, admits a family of trivializations with $\text{Iso}(\mathbb{S}^{1})$-valued transition maps.\\

\noindent From a computational perspective, the structure group $\text{Iso}(\mathbb{S}^{1})\cong\mathbb{S}^{1}\rtimes\mathbb{Z}_{2}$ is vastly preferable to $\text{Homeo}(\mathbb{S}^{1})$. 
In particular, $\text{Iso}(\mathbb{S}^{1})$ is a 1-dimensional Lie group (with two connected components), and it has a natural identification with the group $O(2)$ of $2\times 2$ orthogonal  matrices, which gives rise to the following commutative diagram of split extensions:

\begin{equation}
    \begin{tikzcd}
SO(2) \arrow[r] \arrow[d, "\iota"] & O(2) \arrow[r, "\det"] \arrow[d, "\iota"] & \mathbb{Z}_{2} \arrow[d, equals] \\
\text{Homeo}_{+}(\mathbb{S}^{1}) \arrow[r] & \text{Homeo}(\mathbb{S}^{1}) \arrow[r, "\deg"] & \mathbb{Z}_{2}
    \end{tikzcd}\\
\end{equation}

\noindent Note that $\mathbb{S}^{1}$ is identified with $SO(2)$, giving rise to the following commutative diagram of exponential sequences:

\begin{equation}\label{eq: commutative exponential diagram}
    \begin{tikzcd}
    \mathbb{Z} \arrow[r] \arrow[d, equals] & \mathbb{R} \arrow[r, "\exp"] \arrow[d,"\iota"] & SO(2) \arrow[d, "\iota"] \\
    \mathbb{Z} \arrow[r] & \widetilde{\text{Homeo}}_{+}(\mathbb{S}^{1}) \arrow[r, "\widetilde{\exp}"] & \text{Homeo}_{+}(\mathbb{S}^{1})
    \end{tikzcd}
\end{equation}

\noindent where $\exp:\mathbb{R}\to SO(2)$ is given by $\exp(\theta) = \left(\begin{array}{cc}
\cos(2\uppi\theta) & -\sin(2\uppi\theta)\\
\sin(2\uppi\theta) & \cos(2\uppi\theta)\\
\end{array}\right)$. 
Passing to \v{C}ech cohomology, the Stiefel-Whitney and twisted Euler classes we defined in Theorem \ref{prop: general class thm} to classify circle bundles correspond precisely to their classical counterparts for vector bundles.
However, it is useful to know how these classes arise and can be explicitly computed from arbitrary $\text{Homeo}(\mathbb{S}^{1})$-valued transition data using~\eqref{eq: Homeo+ central ext} and~\eqref{eq: exponential sequence}. \\ 

\subsection{The Tautological Circle Bundle Over $\text{Gr}(2)$ And Dimensionality Reduction}

A classifying map can be used to identify the total space of a fiber bundle with a subspace of a more familiar object.
Specifically, if $G$ is a topological group and $\mu_{G}=(EG,p,BG)$ is a universal $G$-bundle, then for any principal $G$-bundle $\xi = (E,\pi , B)$ over a paracompact space $B$, we can construct a classifying map $f:B\to BG$ with $f^{*}\mu_{G} \cong \xi$. 
This map is covered by a bundle map $\tilde{f}:E \to EG$ in the sense that $p\circ \tilde{f} = f\circ \pi$,
and serves to map the total space of $\xi$ into $EG$. 
More generally, if $\xi' = (E\times_G F , \pi', B)$ is the associated bundle of $\xi$ with fiber a left $G$-space $F$, then $\xi'$ is isomorphic to the pullback of the associated bundle $EG\times_{G}F\to BG$ along the same classifying map $f$.
In this sense, we think of  $EG\times_{G}F\to BG$ as ``universal", though the total space is generally not contractible.  \\

\noindent In the context of circle bundles, the associated bundle construction gives bijective correspondences 

\begin{equation}\label{eq: Circ bndls as planes}
    \text{Bun}_{\mathbb{S}^{1}}(B) \quad \leftrightarrow \quad \text{Prin}_{O(2)}(B) \quad \leftrightarrow\quad \text{Vect}_{2,\mathbb{R}}(B)  
\end{equation}

\noindent for any paracompact space $B$ (once one fixes the action of $O(2)$ on $\mathbb{S}^{1}$). 
It follows that up to homotopy, every circle bundle $\xi$ over $B$ can be realized (up to isomorphism) as the sphere bundle of a rank-2 real Euclidean vector bundle $\xi_{\text{vect}}$. \\

\noindent A model for the classifying space of real rank-2 vector bundles is the infinite Grassmann manifold $\text{Gr}(2)=\text{Gr}_{2}(\mathbb{R}^{\infty})$ consisting of all 2-planes in $\mathbb{R}^{\infty}$, topologized as the direct limit $\text{Gr}_{2}(\mathbb{R}^{2})\subset \text{Gr}_{2}(\mathbb{R}^{3})\subset \cdots $ (see \cite{MilnorStasheff} for more details). 
Identifying each plane with its orthogonal projection matrix gives an embedding of $\text{Gr}(2)$ into $\mathbb{R}^{\infty\times\infty}$. 
The total space of the universal plane bundle $\mu$ over $\text{Gr}(2)$ can be modeled by the Stiefel manifold $V(2) = V(2,\infty)$: 

\begin{equation*}
    V(2,\infty) = \{A\in \mathbb{R}^{\infty\times 2}: A^{T}A = I_2\}
\end{equation*}

\noindent The columns of each $A\in V(2)$ form an orthonormal basis for a 2-plane in $\text{Gr}(2)$ whose orthogonal projection matrix is given by $AA^{T}$. 
Thus, the universal plane bundle is then defined by the surjection 
\begin{equation}
\label{eq: PiGr}
  \begin{array}{rccl}
     \pi_{Gr} :& V(2) & \longrightarrow & \text{Gr}(2)   \\
     & A & \mapsto &  AA^T
\end{array}  
\end{equation}

\noindent Given a circle bundle $\xi$ over a paracompact space $B$, the associated plane bundle $\xi_{\text{vect}}$ (\ref{eq: Circ bndls as planes}) can be expressed as the pullback of $\mu$ along a classifying map $f:B\to \text{Gr}(2)$. 
It follows that $\xi$ is isomorphic to the pullback of the associated circle bundle $\mu_{\circ}=V(2)\times_{O(2)}\mathbb{S}^{1}\to \text{Gr}(2)$ along $f$. 
Thus, $\mu_{\circ}$ is ``universal" in the sense described above, and $\xi$ is classified by the same map as $\xi_{\text{vect}}$. \\

\noindent From the perspective of data science and dimensionality reduction, a major advantage of choosing $\mu_{\circ}$ as our model for the universal circle bundle is that we can apply known dimensionality reduction techniques in $V(2)$ to obtain low-dimensional representations of point clouds in $V(2,d)\times_{O(2)}\mathbb{S}^{1}$, for small $d$. 
Our coordinatization pipeline for discrete approximate circle bundles (Section \ref{sec: Coordinatization}) integrates Principal Stiefel Coordinates \cite{LeeEtAl2025} for this purpose. \\ 

\section{Discrete Approximate Circle Bundles}
\label{sec: Disc Apprx Circ Bundles}

In this section, we introduce discrete approximate notions of local trivializations, \v{C}ech cohomology and classifying maps for circle bundles. 
The purpose of these objects is to represent true circle bundles with finite data which can be algorithmically, efficiently and stably computed from a point cloud and a feature map. \\

\subsection{Discrete Approximate Local Trivializations}

\begin{definition}
   Let $f: X \to Y$ be a function between metric spaces. 
    The $\textbf{distortion}$ of $f$ is defined by 

\begin{equation*}
    \text{dist}(f) = \sup_{x,x' \in X}|d_{Y}(f(x), f(x')) - d_{X}(x,x')|
\end{equation*}    
    
    Given functions $f:X\to Y$ and $g:Y\to X$, the $\textbf{codistortion}$ of $f$ and $g$ is defined by

\begin{equation*}
    \text{codist}(f,g) = \sup_{x\in X, y\in Y}|d_{Y}(f(x), y) - d_{X}(x,g(y))|
\end{equation*}    

\end{definition}
\vspace{5mm}

\begin{definition}
    Let $\pi:X\to B$ and $\pi':X'\to B$ be functions between metric spaces. 
    A \textbf{discrete $(\beta,\eps)$-approximate bundle map} between $(X,\pi,B)$ and $(X',\pi',B)$ is a function $\varphi:X\to X'$ with the following properties: \\

    \begin{enumerate}
        \item There exists a function $\psi:X'\to X$ such that 
        
        \[
        \max\{\text{dist}(\varphi), \ \text{dist}(\psi), \ \text{codist}(\varphi,\psi) \} \leq \eps,
        \]
        
        \item The functions $\varphi$ and $\psi$ satisfy

\begin{equation*}
    \max \{\sup_{x\in X}d_{B}(\pi(x),\pi'\circ \varphi(x)), \quad \sup_{x'\in X'}d_{B}(\pi'(x'),\pi\circ \psi(x'))\} \leq \frac{\beta}{2}
\end{equation*}
            
    \end{enumerate}

    Any function $\psi$ satisfying the properties above is called an $\textbf{almost-inverse}$ for $\varphi$. 
    Note that $\psi$ is also a discrete $(\beta,\eps)$-approximate bundle map. \\
\end{definition}

We note several immediate consequences of the definition above for later use: \\

\begin{proposition}\label{lemma: approx bundle map props}
    Suppose $\varphi:(X,\pi,B)\to (X',\pi',B)$ is a discrete $(\beta,\eps)$-approximate bundle map and $\psi$ is an almost-inverse for $\varphi$. \\

\begin{enumerate}
    \item \textbf{Almost-Inverse Property:} 

    \begin{equation*}
        \sup_{x\in X}d_{X}(\psi\circ \varphi(x),x)\leq \eps \hspace{5mm} \sup_{x'\in X'}d_{X}(\varphi\circ \psi(x'),x')\leq \eps  
    \end{equation*}
    
    \item \textbf{Restriction:} Let $A = \pi^{-1}(U)\subseteq X$ and $A' = (\pi')^{-1}(U)\subseteq X'$ for some set $U\subseteq B$. 
    If $\varphi(A)\subseteq A'$ and $\psi(A')\subseteq A$, then $\varphi|_{A}$ and $\psi|_{A'}$ are almost-inverse discrete $(\beta,\eps)$-approximate bundle maps. \\

    \item \textbf{Composition:} If $\varphi':X'\to X''$ is a discrete $(\beta',\eps')$-approximate bundle map from $(X',\pi',B)$ to $(X'',\pi'',B)$ with almost-inverse $\psi'$, then $\varphi'\circ \varphi$ is a discrete $(\beta+\beta', \eps + \eps')$-approximate bundle map from $(X,\pi,B)$ to $(X'',\pi'',B)$ and $\psi\circ \psi'$ is an almost-inverse.\\
\end{enumerate}
\end{proposition}

Recall that in the classical setting, a local trivialization for a circle bundle $\pi:X\to B$ is a bundle isomorphism between $(\pi^{-1}(U),\pi|_{\pi^{-1}(U)},U)$ and $(U\times \mathbb{S}^{1}, p_{U}, U)$ for   $U\subseteq B$ open (where $p_{U}$ denotes projection onto the first factor). 
We therefore turn our attention to discrete approximate bundle maps into product spaces of the form $U\times\mathbb{S}^{1}$.\\

\begin{definition}\label{def: circle geodesic}
    Let $\sdist$ denote geodesic distance on $\mathbb{S}^{1}$. 
    Note that for $s,t\in[0,2\uppi]$ one has 

    \[
    \sdist(e^{is},e^{it})=\min\{|s-t|,\,2\uppi-|s-t|\}\in[0,\uppi]
    \]

    More generally, for any $r > 0$, let $\sdist^{r}$ denote geodesic distance scaled by a factor of $r$ -- that is, $\sdist^{r}(z,z')=r\sdist(z,z')$ for all $z,z\in\mathbb{S}^{1}$. 
    We will denote the metric space $(\mathbb{S}^{1},\sdist^{r})$ by $\mathbb{S}^{1}_{r}$.
    Finally, given any metric space $B$, we denote the supremum product metric on $B\times\mathbb{S}^{1}_{r}$ by $d_{\infty}^{r}$.\\
\end{definition}

\begin{remark}
    Since key results related to discrete approximate $O(2)$ \v{C}ech cohomology are most easily stated in terms of the Frobenius norm on $O(2)$ (which is closely related to the Euclidean norm on $\mathbb{S}^{1}$), we note the following identity which we will use frequently:

\begin{equation}\label{eq: S1 dist to Euclidean}
    |z - w| = 2\sin^{-1}\left(\frac{\sdist(z,w)}{2}\right),\hspace{1cm}z,w\in\mathbb{S}^{1}
\end{equation}
\end{remark}

\begin{definition}\label{def: disc local triv}
    Suppose $\pi:X\to B$ is a function between metric spaces. 
    For $\beta,\eps\geq 0$ and $r > \frac{16}{\uppi}(\eps + \beta)$, a \textbf{discrete $(r,\beta,\eps)$-approximate local trivialization} of $\pi$ is a discrete $(\beta,\eps)$-approximate bundle map $\varphi:(\pi^{-1}(U),\pi,U)\to (U\times\mathbb{S}^{1}_{r_{0}},p_{B},U)$ for some open set $U\subseteq B$ and $r_{0} \geq r$.\\
\end{definition}

\begin{remark}
    Note that as $\eps$ and $\beta$ approach $0$ (for fixed $r$), discrete $(r,\beta,\eps)$-approximate local trivializations are forced to be continuous local trivializations in the classical sense. 
    The technical requirement $r>\frac{16}{\uppi}(\eps +\beta)$ in Definition~\ref{def: disc local triv} captures the relative scales of $r$, $\beta$ and $\eps$ necessary for key results to hold (in particular, see the proof of Lemma~\ref{lemma: chart pair approx transition}). \\
\end{remark}

We now introduce the notion of a discrete approximate circle bundle: \\

\begin{definition}\label{def: disc appr CB}
    A map $\pi:X\to B$ between metric spaces is called a \textbf{discrete $(r,\beta,\eps)$-approximate circle bundle} if there is a good open cover $\mathcal{U}=\{U_{j}\}_{j\in J}$ of $B$ (called a $\textbf{trivializing cover}$ for $(X,\pi,B)$) with the following properties: \\

    \begin{enumerate}
        \item For each $(jkl)\in\mathcal{N}(\mathcal{U})$, 
        with $j,k,l\in J$ not necessarily distinct, there exists some $x\in X$ such that $\bar{B}_{\beta}(\pi(x))\subseteq U_{jkl}$.\\
        
        \item For each $j\in J$, there exists a discrete $(r,\beta,\eps)$-approximate local trivialization $\varphi_{j}:\pi^{-1}(U_{j})\to U_{j}\times\mathbb{S}^{1}_{r_{j}}$ for some $r_{j}\geq r$. \\
    \end{enumerate}

    The family $\{\varphi_{j}\}_{j\in J}$ is called a $\textbf{discrete $(r,\beta,\eps)$-approximate trivialization of $\pi$}$, and the set of all such families subordinate to a fixed trivializing cover $\mathcal{U}$ is denoted by $T_{r,\beta,\eps}(\pi,\mathcal{U})$.\\
\end{definition}

Note that any $\Phi = \{\varphi_{j}:\pi^{-1}(U_{j})\to U_{j}\times \mathbb{S}_{r_{j}}^{1}\}_{j\in J}\in T_{r,\beta,\eps}(\pi,\mathcal{U})$ has an associated system of local circular coordinate functions $\mathfrak{F}(\Phi)=\{f_{j}:\pi^{-1}(U_{j})\to \mathbb{S}^{1}\}_{j\in J}$ given by $f_{j}=p_{\mathbb{S}^{1}}\circ\varphi_{j}$.
We will see that, as in the classical setting, the relationships between local circular coordinates for a discrete approximate circle bundle encode its global topological structure. \\

\begin{definition}\label{def: disc appr local angle functions}
    Let $\pi:X\to B$ be a function between metric spaces, and let $\mathcal{U} = \{U_{j}\}_{j\in J}$ be an open cover of $B$.  
    For $\eps > 0$ and $0 \leq \delta < 1$, a system of \textbf{discrete $(\eps,\delta)$-approximate local circular coordinates} for $\pi$ is a family
     $\mathfrak{F}=\{f_{j}:\pi^{-1}(U_{j})\to \mathbb{S}^{1}\}_{j\in J}$ with the following properties:\\

    \begin{enumerate}
        \item There exists a simplicial cochain $\Omega\in C^{1}(\mathcal{N}(\mathcal{U});O(2))$ such that 

\begin{equation}\label{eq: disc appr triv cond}
    |f_{j}(x) - \Omega_{jk}f_{k}(x)| < \eps
\end{equation}

    \noindent whenever $x\in \pi^{-1}(U_{jk})$ and $(jk)\in\mathcal{N}(\mathcal{U})$.\\
    
    \item If $d_H$  denotes the Hausdorff distance in $(\mathbb{C}, |\cdot|)$, then $d_{H}(f_{l}(\pi^{-1}(U_{jkl})),\mathbb{S}^{1}) < \delta$ for every $(jkl)\in \mathcal{N}(\mathcal{U})$. \\
    \end{enumerate}
     
     \noindent The set of discrete $(\eps,\delta)$-approximate systems of local angle functions for $\pi$ subordinate to $\mathcal{U}$ is denoted by $\Gamma_{\varepsilon,\delta}(\pi, \mathcal{U})$. 
     Any $\Omega$ satisfying~\eqref{eq: disc appr triv cond} is said to be a $\textbf{witness}$ that $\mathfrak{F}\in \Gamma_{\varepsilon,\delta}(\pi, \mathcal{U})$. 
     Given a system of discrete approximate local circular coordinates, we will refer to the quantity \[\frac{\varepsilon}{1-\delta}\] as the \textbf{roughness} of the system and write 
     \[ \Gamma_{\alpha}(\pi, \mathcal{U}) = 
\bigcup\limits_{\frac{\eps}{1-\delta} < \alpha}\Gamma_{\varepsilon,\delta}(\pi, \mathcal{U})
\] for any $\alpha > 0$. 
We define the following metric on $\Gamma_{\infty}(\pi, \mathcal{U})$:

\begin{equation}\label{eq: DT metric}
    d_{\Gamma}(\{f_{j}\}_{j\in J}, \{g_{j}\}_{j\in J}) 
    \;\;=\;\; 
    \sup_{j\in J}\sup_{x\in\pi^{-1}(U_{j})}|f_{j}(x) - g_{j}(x)|\\
\end{equation}
\end{definition}

In practice, it will often be more convenient to work with systems of discrete approximate local circular coordinates than discrete approximate local trivializations. 
Proposition~\ref{prop: trivs to ang functions} below shows that discrete approximate local trivializations always give rise to systems of discrete approximate local circular coordinates. 
This result is directly implied by the following lemma, 
which we use again in Section~\ref{sec: Relation To True} to compare systems of local circular coordinates induced by different choices of discrete approximate local trivializations subordinate to the same open cover.
The proof is deferred to Appendix~\ref{lemma: Appendix chart pair approx transition}.\\ 

\begin{lemma}\label{lemma: chart pair approx transition}
    Let $\pi:X\to B$ be a function between metric spaces. 
    Suppose $\{\varphi_{j}:(\pi^{-1}(U_{j}),\pi,U_{j})\to (U_{j}\times\mathbb{S}^{1}_{r_{j}},p_{B},U_{j})\}_{j=1}^{2}$ are discrete $(r,\beta,\eps)$-approximate local trivializations for some $r_{1},r_{2} \geq r$ and open sets $U_{1},U_{2}\subseteq B$ such that $\bar{B}_{\beta}(\pi(x))\subseteq U_{12}=U_{1}\cap U_{2}$ for some $x\in X$. Then, 

\[
    |r_{1}-r_{2}|\leq \frac{2}{\pi}(\eps + \beta)
\]
    
    and there exists some $\Omega_{12}\in O(2)$ such that 

\[
    |f_{1}(x) - \Omega_{12} f_{2}(x)| \quad \leq \quad \tilde{\tau}(\varphi_{1},\varphi_{2})
\]
 
    for all $x\in \pi^{-1}(U_{12})$, where $f_{j}=p_{\mathbb{S}^{1}}\circ\varphi_{j}$ for $j = 1,2$ and 

\begin{equation}\label{eq: pair eta and tau}
\tau(\varphi_{1},\varphi_{2}) \ = \ \frac{1}{\min\{r_{1},r_{2}\}}(\text{diam}(U_{12}) + 10\eps + 9\beta),\hspace{1cm}\widetilde{\tau}(\varphi_{1},\varphi_{2})\ = \ 2\sin^{-1}\left(\frac{\tau(\varphi_{1},\varphi_{2})}{2}\right)\\
\end{equation}
\end{lemma}
\hspace{1cm}

The following result is immediate:\\

\begin{proposition}\label{prop: trivs to ang functions}
    Suppose $\mathcal{U}$ is a trivializing cover for a discrete $(r,\beta,\eps)$-approximate circle bundle $\pi:X\to B$ and $\Phi = \{\varphi_{j}:\pi^{-1}(U_{j})\to U_{j}\times\mathbb{S}^{1}_{r_{j}}\}_{j\in J}\in T_{r,\beta,\eps}(\pi,\mathcal{U})$. 
    The functions $\{f_{j} = p_{\mathbb{S}^{1}}\circ \varphi_{j}\}_{j\in J}$ form a system of discrete $(\tilde{\tau}(\Phi),2\sin^{-1}(\frac{\eps}{2r}))$-approximate local circular coordinates for $\pi$, where

\begin{equation}\label{eq: family tau}
    \tau(\Phi)= \sup_{(jk)\in\mathcal{N}(\mathcal{U})}\tau (\varphi_{j},\varphi_{k}),\hspace{1cm}\tilde{\tau}(\Phi)=2\sin^{-1}\left(\frac{\tau(\Phi)}{2}\right)
\end{equation}
\end{proposition}

\begin{proof}
    Suppose $(jk)\in\mathcal{N}(\mathcal{U})$. 
    If $j < k$, let $\Omega_{jk}\in O(2)$ be any matrix which satisfies 
\[
|f_{j}(x) - \Omega_{jk}f_{k}(x)| \quad \leq \quad \widetilde{\tau}(\varphi_{j},\varphi_{k})\quad \leq \quad \widetilde{\tau}(\Phi)
\]
    for all $x\in \pi^{-1}(U_{jk})$ (such a matrix is guaranteed to exist by Lemma~\ref{lemma: chart pair approx transition}).
    If $j>k$, define $\Omega_{jk}=\Omega_{kj}^{-1}$ and observe
\[
|f_{j}(x) - \Omega_{jk}f_{k}(x)| \quad = \quad |\Omega_{jk}^{-1}f_{j}(x)-f_{k}(x)|\quad = \quad |\Omega_{kj}f_{j}(x) - f_{k}(x)|\quad \leq \quad  \widetilde{\tau}(\varphi_{k},\varphi_{j})\quad \leq \quad \widetilde{\tau}(\Phi)
\]
    Lastly, define $\Omega_{jj}=\text{Id}$ for all $j\in J$.
    Evidently, the matrices $\{\Omega_{jk}\}_{(jk)\in\mathcal{N}(\mathcal{U})}$ can be interpreted as a \v{C}ech cochain $\Omega\in\check{C}^{1}(\mathcal{U};\underline{O(2)})$, and we have 
\[|f_{j}(x) - \Omega_{jk}f_{k}(x)|\ \leq \ \widetilde{\tau}(\Phi)\]

    whenever $x\in \pi^{-1}(U_{jk})$ for some $(jk)\in\mathcal{N}(\mathcal{U})$.\\

    Now, suppose $U_{jkl}\neq \emptyset$, and let $z\in\mathbb{S}^{1}$ be arbitrary. 
    Let $x\in X$ be so that  $\bar{B}_{\beta}(\pi(x))\subseteq U_{jkl}$ and let $b = \pi(x)$.
    Then,
\[
\sdist^{r_{j}}(f_{j}(x),z)\quad \leq \quad d_{\infty}^{r_{j}}(\varphi_{j}(x),  (b,z))\quad \leq \quad \eps 
\]
    so it follows that $d_{H}(f_{j}(\pi^{-1}(U_{jkl})), \mathbb{S}^{1})\leq \frac{\eps}{r_{j}}\leq \frac{\eps}{r}$ with respect to unscaled geodesic distance on $\mathbb{S}^{1}$.
    Applying identity~\eqref{eq: S1 dist to Euclidean}, the Hausdorff bound becomes $2\sin^{-1}(\frac{\eps}{2r})$, so the proposition follows.\\
\end{proof}

\begin{remark}
      In light of the proposition above, we define the $\textbf{roughness}$ of a family $\Phi\in T_{r,\beta,\eps}(\pi,\mathcal{U})$ by

\begin{equation}\label{eq: alpha trivializations}
    \alpha (\Phi) = \frac{\widetilde{\tau}(\Phi)}{1 - 2\sin^{-1}\left(\frac{\eps}{2r}\right)}
\end{equation}

      By the previous result, $\alpha(\Phi)$ is equal to the roughness of $\mathfrak{F}(\Phi)$.
      Note that our definition of discrete approximate local trivializations requires $r_{j} > \frac{16\eps}{\pi}$ for each $j$, so

\begin{equation}\label{eq: Hausdorff bound}
 2\sin^{-1}\left(\frac{\eps}{2r}\right)\quad < \quad  2\sin^{-1}\left(\frac{\pi}{32}\right) \quad < \quad  \frac{1}{5}
\end{equation}

for any discrete approximate circle bundle.
It follows that $\alpha(\Phi) < \frac{4}{5}\widetilde{\tau}(\Phi)$ for any family of discrete approximate local trivializations. \\
\end{remark}

\subsection{Discrete Approximate \v{C}ech Cohomology}

\noindent Here, we introduce discrete and approximate relaxations of classical \v{C}ech cohomology. 
As we have seen, \v{C}ech cohomology provides a useful framework for computing obstructions and classifications of circle bundles, so these objects will play a central role in our classification and coordinatization algorithms.  \\ 

\begin{definition}
    Given a collection  $\mathcal{U} = \{U_{j}\}_{j\in J}$ of open subsets of a space $B$ and a metric group $G$, the set of $G$-valued \textbf{$\varepsilon$-approximate \v{C}ech cocycles} subordinate to $\mathcal{U}$ is defined by 

\begin{equation}\label{eq: approx cocycles def}
    \check{Z}^{k}_{\varepsilon}(\mathcal{U};\mathcal{C}_{G}) = \left\{\Omega\in\check{C}^{k}(\mathcal{U};\mathcal{C}_{G}): \ d_{Z}(\delta^{k}\Omega, \mathbbm{1}) < \varepsilon\right\}
\end{equation}

    \noindent where $d_{Z}$ is the supremum metric on $\check{C}^{k+1}(\mathcal{U};\mathcal{C}_{G})$: 

\begin{equation}\label{eq: cochain metric}
    d_{Z}(\Omega 
    ,\Lambda) = \sup_{ \sigma \in\mathcal{N}(\mathcal{U})^{(k+1)}}
    \;\sup_{b\in U_{\sigma}}d_{G}(\Omega_{\sigma}(b), \Lambda_{\sigma}(b))
\end{equation}    
    \noindent and $\mathbbm{1} \in \check{C}^{k+1}(\mathcal{U}; \mathcal{C}_G)$ is the cochain which is constant and equal to the identity of $G$.
  The set of $G$-valued \textbf{discrete $\varepsilon$-approximate \v{C}ech cocycles} subordinate to $\mathcal{U}$ is defined by  

\begin{equation}\label{eq: disc approx cocycles def}
    \check{Z}^{k}_{\varepsilon}(\mathcal{U};\underline{G}) = \check{Z}^{k}_{\varepsilon}(\mathcal{U};\mathcal{C}_{G})\cap \check{C}^{k}(\mathcal{U};\underline{G})
\end{equation}
    
    \noindent For $k=1$, we define the \textbf{$\varepsilon$-approximate cohomology set} $\check{H}^{1}_{\varepsilon}(\mathcal{U};\mathcal{C}_{G})$ as the orbit space of $\check{Z}^{1}_{\varepsilon}(\mathcal{U};\mathcal{C}_{G})$ with respect to the left action $\check{C}^{0}(\mathcal{U};\mathcal{C}_{G})\curvearrowright\check{C}^{1}(\mathcal{U};\mathcal{C}_{G})$ given by $(\phi\cdot \Omega)_{jk} =  \phi_{j}\Omega_{jk}\phi_{k}^{-1}$.\\
    
\end{definition}

\noindent We will primarily be interested in the case where $G = O(2)$, since every circle bundle over a paracompact space $B$ is represented by a class in $\check{H}^{1}(B;\mathcal{C}_{O(2)})$. 
Here we endow $O(2)$ with the Frobenius metric, given its predominance in algorithmic linear algebra and optimization.  \\

\subsection{Approximate Classifying Maps}

Finally, we introduce an approximate version of classifying maps. 
These will play an important role in our coordinatization pipeline. \\

\begin{definition}
    Given a  space $B$ and $\varepsilon\geq 0$, an \textbf{$\varepsilon$-approximate classifying map} is a map $f:B\to \text{Gr}(2)^{\varepsilon}$, where 
    $\text{Gr}(2)^0 = \text{Gr}(2)$, and $\text{Gr}(2)^{\varepsilon}\subseteq\mathbb{R}^{\infty\times\infty}$   for $\varepsilon >0$ is the $\varepsilon$-thickening   

\begin{equation*}
    \text{Gr}(2)^{\varepsilon} = \{A\in  \mathbb{R}^{\infty \times \infty}: \|A - P\|_{F} < \varepsilon \text{ for some }P\in \text{Gr}(2)\}\\
\end{equation*}
\end{definition}
\vspace{1mm}

\subsection{Relationships Between Notions}

In this section, we describe two key constructions which produce functions 

\begin{equation*}
    \text{wit}:\Gamma_{\alpha}(\pi,\mathcal{U})\to \check{Z}^{1}_{(3\sqrt{2})\alpha}\left(\mathcal{U};\underline{O(2)}\right)\hspace{1cm}\text{cl}:\check{Z}^{1}_{\varepsilon}(\mathcal{U};\mathcal{C}_{O(2)})\to \text{Maps}\left(B, \text{Gr}(2)^{\sqrt{2}\varepsilon}\right)\\
\end{equation*}

\noindent We first show that systems of discrete approximate local angle functions are witnessed by discrete approximate cocycles: \\

\begin{proposition}\label{prop: triv to cocycle}
    Suppose $\Omega$ witnesses that $\mathfrak{F}\in \Gamma_{\alpha}(\pi, \mathcal{U})$ for some $\alpha > 0$. 
    Then $\Omega\in \check{Z}^{1}_{3\alpha\sqrt{2}} \left(\mathcal{U};\underline{O(2)}\right)$. \\
\end{proposition}

\begin{proof}
    Assume $\mathfrak{F}=\{f_{j}\}_{j\in J}\in \Gamma_{\eps,\delta} (\pi, \mathcal{U})$, where $\alpha = \frac{\eps}{1-\delta}$. 
    Suppose $x\in \pi^{-1}(U_{jkl})$ and let $v_{st} = \Omega_{st}f_{t}(x) - f_{s}(x)$ for $s,t\in \{j,k,l\}$. Then,

\begin{equation*}
    \left|\Omega_{jk}\Omega_{kl}f_{l}(x) - \Omega_{jl}f_{l}(x)\right| \quad =\quad  \left|\Omega_{jk}(v_{kl} + f_{k}(x)) - \Omega_{jl}f_{l}(x)\right|
\end{equation*}

\begin{equation*}
    = \quad \left|\Omega_{jk}v_{kl} + (f_{j}(x) + v_{jk}) - (f_{j}(x) - v_{jl})\right| \quad =\quad  \left|\Omega_{jk}v_{kl} + v_{jk} - v_{jl}\right|
\end{equation*}
        
\begin{equation*}
    \leq \quad |v_{kl}| + |v_{jk}| + |v_{jl}|\quad  < \quad 3\eps\\
\end{equation*}

\noindent The result then follows from applying Proposition~\ref{prop: Frobenius delta surjectivity} to each matrix $(\Omega_{jk}\Omega_{kl} - \Omega_{jl})$.\\
\end{proof}

Below we describe an explicit construction which produces a witness for a system of discrete approximate local angle functions.\\

\begin{construction}\label{const: maxmin Procrustes}
    Let $X$ be any nonempty set, and let $f,g:X\to\mathbb{S}^{1}$. 
    For each $x\in X$, define $e^{i\alpha_{x}} = f(x)$ and $e^{i\beta_{x}} = g(x)$, then define $\gamma_{x} = \alpha_{x} - \beta_{x}$ and $\tilde{\gamma}_{x} = \alpha_{x} + \beta_{x}$. 
    Let $e^{i\theta}$ denote the midpoint of the shortest arc $C$ in $\mathbb{S}^{1}$ containing $\{e^{i\gamma_{x}}\}_{x\in X}$, and let $e^{i\widetilde{\theta}}$ denote the midpoint of the shortest arc $\widetilde{C}$ containing $\{e^{i\widetilde{\gamma}_{x}}\}_{x\in X}$. 
    Define

\begin{equation*}
    d_{\text{rot}}=\sup_{x\in X}|f(x) - R_{\theta}g(x)|, \hspace{1cm} d_{\text{ref}}=\sup_{x\in X}|f(x) - R_{\widetilde{\theta}}g(x)|
\end{equation*}

    If $d_{\text{rot}} < d_{\text{ref}}$, set $\Omega^{*} = R_{\theta}$; otherwise, set $\Omega^{*} = R_{\widetilde{\theta}}r$, where $R_{\theta} = \left(\begin{array}{cc}
    \cos(\theta) & -\sin(\theta) \\
    \sin(\theta) & \cos(\theta)\\
    \end{array}\right)$ and $r = \left(\begin{array}{cc}
    1 & 0 \\
    0 & -1\\
    \end{array}\right)$. \\
\end{construction}
\vspace{5mm}
\begin{proposition}
    If  $\text{diam}(\{e^{i\gamma_{x}}\}_{x\in X}) < \uppi$ in Construction~\ref{const: maxmin Procrustes} above, then $\Omega^{*}\in O(2)$ is the unique minimizer of $\min\limits_{\Omega\in O(2)}\sup\limits_{x\in X}|f(x) - \Omega g(x)|$.
\end{proposition}

\begin{proof}
    Let $L:O(2)\to \mathbb{R}_{\geq 0}$ be defined by $L(\Omega) = \sup\limits_{x\in X}|f(x) - \Omega g(x)|^{2}$. 
    We have
    
\begin{align*}
    |f(x) - \Omega g(x)|^{2} \quad &= \quad |f(x)|^{2} + |g(x)|^{2} - 2f(x)\cdot\Omega g(x)\\
    &= \quad 2 - 2f(x)\cdot\Omega g(x)
\end{align*}

    for all $x$ (here we are identifying $\mathbb{R}^{2}$ with $\mathbb{C}$ in the usual way). 
    If $\Omega = R_{\theta}$ is a counter-clockwise rotation by $\theta$, then $f(x)\cdot \Omega g(x) = \cos(\gamma_{x} - \theta)$, so to minimize $L$ over $SO(2)$ we need to identify $\theta\in \mathbb{S}^{1}$ for which the angular distance to the farthest $\gamma_{x}$ is minimized. 
    If the measure of $C$ is less than $\pi$, then this is precisely achieved at the midpoint of the shortest arc $C$ containing $\{e^{i\gamma_{x}}\}_{x\in X}$ (furthermore, the requirement that $\text{diam}(\{e^{i\gamma_{x}}\}_{x\in X}) < \pi$ guarantees that $C$ is unique). 
    To see this, first note that $e^{i\theta}$ certainly lies in $C$ for any optimal $\theta$, and for any fixed $e^{i\theta}\in C$, $\sdist(e^{i\gamma_{x}},e^{i\theta})$ is smallest when $e^{i\gamma_{x}}$ is an endpoint of $C$. 
    By symmetry, the $\theta$ which maximizes the minimum of the values $\{\sdist(e^{i\gamma_{x}},e^{i\theta})\}_{x\in X}$ must be half way between the two end points (otherwise, a small perturbation of $\theta$ towards the midpoint would move $\theta$ closer to the more distant endpoint while still keeping $\theta$ nearer to the closer endpoint. 
    This would yield a more optimal choice for $\theta$). 
    
    Analogously, if $\Omega = R_{\widetilde{\theta}}r$ for some $\widetilde{\theta}\in\mathbb{R}$, one can check that $f(x)\cdot \Omega g(x) = \cos(\widetilde{\gamma}_{x} - \theta)$. 
    Thus, to minimize $L$ over the component of $O(2)$ with matrices of negative determinant, we need to identify $\widetilde{\theta}\in \mathbb{S}^{1}$ for which the angular distance to the farthest $\widetilde{\gamma}_{x}$ is minimized. 
    Finally, comparing $d_{\text{rot}}$ and $d_{\text{ref}}$ determines which component of $O(2)$ contains the optimal $\Omega$. 
    To see that these two values cannot be equal, observe that if $\text{diam}(\{e^{i\gamma_{x}}\}_{x\in X}) < \pi$, then $\text{diam}(\{e^{i\widetilde{\gamma}_{x}}\}_{x\in X}) > \pi$. 
    It follows that then $d_{\text{rot}} < \frac{\pi}{2}$ and $d_{\text{ref}} > \frac{\pi}{2}$. 
    Similarly, if $\text{diam}(\{e^{i\widetilde{\gamma}_{x}}\}_{x\in X}) < \pi$, then $\text{diam}(\{e^{i\gamma_{x}}\}_{x\in X}) > \pi$, so $d_{\text{ref}} < d_{\text{rot}}$.\\
\end{proof}

\begin{corollary}
        If $\mathfrak{F}\in T_{\eps,\delta}(\pi,\mathcal{U})$ for $\eps < \sqrt{2}$ and $|\pi^{-1}(U_{jk})| > 1$ for all $(jk)\in\mathcal{N}(\mathcal{U})$, then Construction~\ref{const: maxmin Procrustes} yields a witness. \\
\end{corollary}

\begin{proof}
    Let $\mathfrak{F} = \{f_{j}\}_{j\in J}$.
    The hypotheses imply 

\begin{equation*}
    |f_{j}(x) - \Omega_{jk}f_{k}(x)|^{2} < \eps^{2}
\end{equation*}

    or, equivalently, 

\begin{equation*}
    \gamma^{jk}_{x} = (f_{j}(x))\cdot (\Omega_{jk}f_{k}(x)) > 1 - \frac{\eps^{2}}{2}
\end{equation*}

    for all $x\in U_{jk} $ and $ (jk)\in \mathcal{N}(\mathcal{U})$. 
    Then $|\gamma^{jk}_{x}| > \cos^{-1}(1 - \frac{\eps^{2}}{2})$ for all $x\in U_{jk}$, so $\text{diam}(\{\gamma^{jk}_{x}\}) < 2\cos^{-1}(1 - \frac{\eps^{2}}{2})$.  We conclude that if $\cos^{-1}(1 - \frac{\eps^{2}}{2}) < \frac{\pi}{2}$ (or, equivalently, if $\eps < \sqrt{2}$), then each $\Omega_{jk}$ obtained using Construction~\ref{const: maxmin Procrustes} is the unique minimizer of $\max\limits_{x\in \pi^{-1}(U_{jk})}|f_{j}(x) - f_{k}(x)|$. Thus by Proposition~\ref{prop: triv to cocycle}, $\Omega$ is a witness that $\mathfrak{F}\in T_{\eps,\delta}(\pi,\mathcal{U})$.\\
\end{proof}

\noindent We state the following technical result for later reference: \\

\begin{proposition}\label{prop: witnesses of close trivs}
    Let $\delta < 1$. Suppose $\Omega,\Lambda\in \check{C}^{1}\left(\mathcal{U};\underline{O(2)}\right)$  witness $\{f_{j}\}_{j\in J}\in \Gamma_{\eps_{f},\delta}(\pi, \mathcal{U})$ and $\{g_{j}\}_{j\in J}\in \Gamma_{\eps_{g},\delta}(\pi,\mathcal{U})$, respectively, and $d_{H}(f_{k}(\pi^{-1}(U_{jk})),\mathbb{S}^{1}) < \delta_{1}$ in $(\mathbb{C},|\cdot|)$ for every ordered 1-simplex $(jk)\in\mathcal{N}(\mathcal{U})$. 
    If $d_{\Gamma}(\{f_{j}\}_{j\in J},\{g_{j}\}_{j\in J}) < \delta_{2}$, then $d_{Z}(\Omega,\Lambda) < \frac{\sqrt{2}}{1-\delta_{1}}(\eps_{f} + \eps_{g} + 2\delta_{2})$. \\
\end{proposition}

\begin{proof}
     Given any $(jk)\in\mathcal{N}(\mathcal{U})$ and $x\in\pi^{-1}(U_{jk})$, the triangle inequality implies 

    \begin{equation*}
        |\Omega_{jk}f_{k}(x) - \Lambda_{jk}f_{k}(x)| \quad \leq 
    \end{equation*}
    
    \begin{equation*}
        |\Omega_{jk}f_{k}(x) - f_{j}(x)| \ + \ |f_{j}(x) - g_{j}(x)| \ + \ |g_{j}(x) - \Lambda_{jk}g_{k}(x)| \ + \ |\Lambda_{jk}g_{k}(x) - \Lambda_{jk}f_{k}(x)|
    \end{equation*}
    
    \begin{equation*}
        < \ \eps_{f} + \delta_{2} + \eps_{g} + \delta_{2} \\
    \end{equation*}

    \noindent The result then follows immediately from Proposition~\ref{prop: Frobenius delta surjectivity}. \\
\end{proof}

\begin{corollary}\label{prop: two trivs of same witness}
    If $\Omega,\Lambda\in\check{C}^{1}\left(\mathcal{U};\underline{O(2)}\right)$ both witness $\mathfrak{F}\in \Gamma_{\alpha}(\pi,\mathcal{U})$, then $d_{Z}(\Omega,\Lambda)<  2\sqrt{2}\, \alpha$.\\
\end{corollary}

\begin{corollary}\label{cor: witnesses of close trivs}
        Suppose $\Omega,\Lambda\in \check{C}^{1}\left(\mathcal{U};\underline{O(2)}\right)$ witness $\mathfrak{F},\mathfrak{G}\in T_{\alpha}(\pi, \mathcal{U})$, respectively. 
        If $d_{\Gamma}(\mathfrak{F},\mathfrak{G}) < \delta$, then $d_{Z}(\Omega,\Lambda) < 2\sqrt{2}(\alpha + \delta)$. \\
\end{corollary}

\noindent The following construction and subsequent result are due to \cite[Construction 4.15]{Euclidean_Vector_Bundles}:\\

\begin{construction}\label{const: approx cocycle to approx cl map}
    Suppose $\Omega\in\check{Z}^{1}_{\eps}(\mathcal{U};\mathcal{C}_{O(2)})$ for some locally-finite open cover $\mathcal{U} = \{U_{j}\}_{j\in \mathbb{N}}$ of a paracompact space $B$. Choose a partition of unity $\{\rho_{j}\}_{j\in \mathbb{N}}$ subordinate to $\mathcal{U}$. For each $j\in \mathbb{N} = \{0,1,\ldots\}$, define $\Psi_{j}:U_{j}\to V(2)$ by 

\begin{equation}\label{eq: Phi maps}
    \Psi_{j}(b) = \left(\begin{array}{ccc}
    \sqrt{\rho_{0}(b)} \Omega_{0j}(b)\\
    \sqrt{\rho_{1}(b)} \Omega_{1j}(b)\\
    \vdots\\
    \end{array}\right)
\end{equation}

    \noindent Note that when $b\in U_{j}\cap U_{k}^{c}$, $\rho_{k}(b) = 0$, so this map is well-defined. 
    Let $f:B\to \mathbb{R}^{\infty\times \infty}$ be defined by 

\begin{equation}
    \widetilde{f}(b) = \sum_{j =0}^\infty\rho_{j}(b)\Psi_{j}(b)\Psi_{j}(b)^{T}
\end{equation}

    \noindent Since each $b$ is contained in a finite number of sets in $\mathcal{U}$, the sum above is finite and hence $\tilde{f}$ is well-defined. \\

\end{construction}

\noindent For a proof of the theorem below, see \cite[Lemmas 4.16, 4.17, 4.18 and Theorem 4.21]{Euclidean_Vector_Bundles}:\\

\begin{theorem}\label{prop: approx cocycle to approx cl map}
    For $\eps \geq 0$ and any partition of unity $\{\rho_{j}\}_{j\in \mathbb{N}}$, Construction~\ref{const: approx cocycle to approx cl map} yields a map $\text{cl}:\check{Z}^{1}_{\eps}(\mathcal{U};\mathcal{C}_{O(2)})\to \text{Maps}(B,\text{Gr}(d)^{\sqrt{2}\eps})$ with the following properties: \\

\begin{enumerate}
    \item The induced map $\text{cl}:\check{Z}^{1}_{\eps}(\mathcal{U};\mathcal{C}_{O(2)})\to \left[B,\text{Gr}(d)^{\sqrt{2}\eps}\right]$ factors through $\check{H}^{1}_{\eps}(\mathcal{U};\mathcal{C}_{O(2)})$ and does not depend on  the choice of $\{\rho_{j}\}_{j\in \mathbb{N}}$.\\

    \item If $\Omega,\Lambda\in\check{Z}^{1}_{\eps}(\mathcal{U};\mathcal{C}_{O(2)})$ satisfy $d_{Z}(\Omega,\Lambda)< \delta$, then $\text{cl}(\Omega)$ and $\text{cl}(\Lambda)$ become equal in $\left[B,\text{Gr}(d)^{\sqrt{2}(\delta + \eps)}\right]$.\\
\end{enumerate}    
\end{theorem}

\subsection{Relationship To True Circle Bundles}\label{sec: Relation To True}

In this section, we show that for sufficiently small $\alpha > 0$, any system of discrete approximate local angle functions $\mathfrak{F}\in \Gamma_{\alpha}(\pi,\mathcal{U})$ uniquely corresponds to a true circle bundle. 
From this we deduce conditions under which a discrete approximate circle bundle can be identified with an isomorphism class of classical circle bundles.
We build on previous results of~\cite{Tinarrage2022} and~\cite{Euclidean_Vector_Bundles} which provide conditions under which approximate $O(2)$-cocycles and classifying maps can be uniquely identified with true ones. \\

The following construction and subsequent result~\cite[Lemma 1]{Tinarrage2022} describe a procedure for obtaining a classifying map from an approximate one: \\

\begin{construction}\label{const: Gr projection}
    Given $A\in\mathbb{R}^{n\times n}$,  compute its symmetrization $A_{s} = \frac{A + A^{T}}{2}$, and an orthogonal diagonalization $A_s = \Lambda D\Lambda^{T}$ such that the eigenvalues of $A_s$ appear in decreasing order along the diagonal of $D$. 
    Finally, set $\Pi(A) = \Lambda J_{2}\Lambda^{T}$, where $J_{2}$ is the $n\times n$ matrix whose only non-zero entries are the first two diagonal entries, both equal to 1. Note that this construction naturally extends to the case $n = \infty$. \\
\end{construction}

\begin{proposition}
    For $\varepsilon \leq \frac{\sqrt{2}}{2}$,  
     Construction~\ref{const: Gr projection} yields a deformation retraction $\Pi: \text{Gr}(2)^{\varepsilon} \to \text{Gr}(2)$. 
     Hence, the induced map  $\Pi_{*}:\text{Maps}(B, \text{Gr}(2)^{\varepsilon})\to\text{Maps}(B,\text{Gr}(2))$ descends to a natural bijection $\Pi_{*}:[B, \text{Gr}(2)^{\varepsilon}]\to [B,\text{Gr}(2)]$. \\
\end{proposition}

For later use, we also record the following additional construction and result from the same paper: \\

\begin{construction}\label{const: Stiefel projection}
    Let $S_{n} = \{(P,A)\in \text{Gr}(2,n)\times V(2,n): \text{rank}(PA) = 2\}$. Given any $(P,A)\in S_{n}$, let $\widehat{\Pi}_{n}(P,A)\in V(2,n)$ denote the $U$ matrix from the polar decomposition $UH = PA$. Explicitly, $H = ((PA)^{T}(PA))^{\frac{1}{2}}$ and $U = (PA)H^{-1}$.\\
\end{construction}

\begin{proposition}
    Construction~\ref{const: Stiefel projection} defines a map $\widehat{\Pi}_{n}:S_{n}\to V(2,n)$ for each $n$, thereby inducing a map $\widehat{\Pi}:S\to V(2)$ (where $S$ is the direct limit of the sequence $S_{1}\subset S_{2}\subset\cdots$). Furthermore, $\widehat{\Pi}(P,A)$ is the point in $\pi_{Gr}^{-1}(P)$ closest to $A$ with respect to the Frobenius metric.\\
\end{proposition}

Finally, we have the following construction and subsequent result~\cite[Theorem 4.27]{Euclidean_Vector_Bundles}:\\

\begin{construction}\label{const: cocycle projection}
    Suppose $\Omega\in\check{Z}^{1}_{\eps}(\mathcal{U};\mathcal{C}_{O(2)})$ for $\eps < \frac{1}{2}$. Let $\{\Psi_{j}:U_{j}\to V(2)\}_{j\in \mathbb{N}}$ be the maps defined by~\eqref{eq: Phi maps}, and let $\Pi$ be as described in Construction~\ref{const: Gr projection}.
    Define $f_{\Omega} = \Pi_{*}\circ\text{cl}(\Omega)$, and for each $j$  define $\widetilde{\Psi}_{j}:U_{j}\to V(2)$ by $\widetilde{\Psi}_{j}(b) = \hat{\Pi}(f_{\Omega}(b), \Psi_{j}(b))$. 
    Finally, define $\widetilde{\Omega}\in\check{Z}^{1}(\mathcal{U};\mathcal{C}_{O(2)})$ by $\widetilde{\Omega}_{jk}(b) = \widetilde{\Psi}_{j}(b)^T\widetilde{\Psi}_{k}(b)$. \\
\end{construction}

\begin{proposition}\label{prop: cocycle projection}
    
    For $\eps < \frac{1}{2}$, Construction~\ref{const: cocycle projection} yields a map $\Pi_{Z}:\check{Z}^{1}_{\varepsilon}(\mathcal{U};\mathcal{C}_{O(2)})\to \check{Z}^{1}(\mathcal{U};\mathcal{C}_{O(2)})$ such that the following diagram commutes: 

\begin{equation*}
\begin{tikzcd}
\check{Z}^{1}_{\varepsilon}(\mathcal{U};\mathcal{C}_{O(2)}) \arrow[r, "\text{cl}"] \arrow[d, "\Pi_{Z}"'] & \text{Maps}(B,\text{Gr}(2)^{\sqrt{2}\varepsilon}) \arrow[d, "\Pi_{*}"] \\
\check{Z}^{1}(\mathcal{U};\mathcal{C}_{O(2)}) \arrow[r, "\text{cl}"] & \text{Maps}(B,\text{Gr}(2))
\end{tikzcd}
\end{equation*}

    \noindent We also have an induced commutative diagram at the level of cohomology and homotopy classes of maps. Furthermore, if $\Omega\in \check{Z}^{1}_{\varepsilon}(\mathcal{U};\mathcal{C}_{O(2)})$ for $\varepsilon < \frac{\sqrt{2}}{4}$, then $d_{Z}(\Omega, \Pi_{Z}(\Omega)) < 9\varepsilon$.\\
\end{proposition}

The following corollary is immediate: 

\begin{corollary}\label{prop: approx cocycle or map determines true}
    If $\varepsilon < \frac{\sqrt{2}}{2}$, any $\varepsilon$-approximate classifying map uniquely determines a true classifying map, hence can be identified with an isomorphism class of true circle bundles. Similarly, if $\varepsilon < \frac{1}{2}$, any $\varepsilon$-approximate $O(2)$ \v{C}ech 1-cocycle $\Omega$ uniquely determines a true circle bundle classified by $\Pi_{*}\circ \text{cl}(\Omega)$.\\ 
\end{corollary}

\begin{proposition}\label{prop: triv determines circle bundle}
    Suppose $\mathfrak{F}\in \Gamma_{\alpha}(\pi,\mathcal{U})$ for $\alpha < \frac{\sqrt{2}}{20}$. 
    Then any witness determines a true circle bundle, and the circle bundles determined by any two witnesses coincide. 
    In this sense, $\mathfrak{F}$ uniquely determines a true circle bundle.
\end{proposition}

\begin{proof}
    Any witness $\Omega$ lies in $\check{Z}^{1}_{(3\sqrt{2})\alpha}\left(\mathcal{U};\underline{O(2)}\right)$ by Proposition~\ref{prop: triv to cocycle}. 
    Thus, if $\alpha < \frac{1}{6\sqrt{2}}$, then $(3\sqrt{2})\alpha < \frac{1}{2}$, so by Proposition~\ref{prop: approx cocycle or map determines true} $\Omega$ determines a true circle bundle classified by $\Pi_{*}\circ \text{cl}(\Omega)$. \\
    
    \noindent By Corollary~\ref{prop: two trivs of same witness}, any two witnesses $\Omega,\Omega'\in \check{Z}^{1}_{(3\sqrt{2})\alpha}\left(\mathcal{U};\underline{O(2)}\right)$ satisfy $d_{Z}(\Omega,\Omega ') < 2\varepsilon\sqrt{2}$. 
    Thus, by Proposition~\ref{prop: approx cocycle to approx cl map}, $\Pi_{*}\circ \text{cl}(\Omega) = \Pi_{*}\circ \text{cl}(\Omega')$ as long as $\sqrt{2}(3\varepsilon\sqrt{2}+2\varepsilon\sqrt{2}) = 10\varepsilon\leq \frac{\sqrt{2}}{2}$, or equivalently, $\varepsilon\leq \frac{\sqrt{2}}{20}$.\\
\end{proof}

We now discuss conditions under which a discrete approximate circle bundle can be uniquely and stably identified with a class of true circle bundles. \\  
\begin{proposition}\label{prop: close trivs same class}
    Suppose $\mathfrak{F},\mathfrak{G}\in \Gamma_{\alpha}(\pi, \mathcal{U})$ and $d_{\Gamma}(\mathfrak{F},\mathfrak{G}) < \delta$. 
    If $5\alpha + 2\delta < \frac{\sqrt{2}}{4}$, then $\mathfrak{F}$ and $\mathfrak{G}$ both determine the same true circle bundle. \\
\end{proposition}

\begin{proof}
    By Corollary~\ref{cor: witnesses of close trivs}, any two witnesses $\Omega_{1},\Omega_{2}\in\check{C}^{1}\left(\mathcal{U};\underline{O(2)}\right)$ for  $\mathfrak{F},\mathfrak{G}\in \Gamma_{\alpha}(\pi,\mathcal{U})$ satisfy $d_{Z}(\Omega_{1},\Omega_{2})< 2\sqrt{2}(\alpha + \delta)$. 
    By Proposition~\ref{prop: triv to cocycle}, we have $\Omega_{1},\Omega_{2}\in\check{Z}^{1}_{(3\sqrt{2})\alpha}\left(\mathcal{U};\underline{O(2)}\right)$, so Proposition~\ref{prop: cocycle projection} implies $[\text{cl}(\Omega_{1})]$ and $[\text{cl}(\Omega_{2})]$ coincide in $\left[B, \text{Gr}(2)^{\sqrt{2}\eps}\right]$, where $\eps = (3\sqrt{2})\alpha + 2\sqrt{2}(\alpha + \delta)$. 
    Since $\Pi$ induces a bijection from $\left[B, \text{Gr}(2)^{\sqrt{2}\eps}\right]$ to $[B, \text{Gr}(2)]$ by Proposition~\ref{prop: approx cocycle or map determines true}, $\Omega_{1}$ and $\Omega_{2}$ determine the same true circle bundle as long as

\begin{equation*}
    \sqrt{2}\left((3\sqrt{2})\alpha + 2\sqrt{2}(\alpha + \delta)\right) < \frac{\sqrt{2}}{2}
\end{equation*}

    or equivalently, if $5\alpha + 2\delta < \frac{\sqrt{2}}{4}$.\\
\end{proof}

\begin{corollary}\label{cor: homotopic trivs}
    Suppose $\mathfrak{F}=\{f_{j}\}_{j\in J},\mathfrak{G}=\{g_{j}\}_{j\in J}\in \Gamma_{\alpha}(\pi,\mathcal{U})$ for some $\alpha < \frac{\sqrt{2}}{20}$, and that there are homotopies $\{h_{j}:U_{j}\times I\to \mathbb{S}^{1}\}_{j\in J}$ such that $h_{j}(\cdot, 0) = f_{j}$, $h_{1}(\cdot, 1) = g_{j}$ and $\{h_{j}( \cdot, t)\}_{j\in J}\in \Gamma_{\alpha}(\pi,\mathcal{U})$ for all $t\in I$. 
    Then $\mathfrak{F}$ and $\mathfrak{G}$ determine the same true circle bundle. 
\end{corollary}

\begin{proof}
    Suppose $\{h_{j}:\pi^{-1}(U_{j})\times I\to \mathbb{S}^{1}\}_{j\in J}$ are continuous maps such that $\{h_{j}(\cdot,t)\}_{j\in J}\in \Gamma_{\alpha}(\pi,\mathcal{U})$ for all $t$, and $h_{j}(x,0) = f_{j}(x)$ and $h_{j}(x,1) = g_{j}(x)$. 
    Proposition~\ref{prop: close trivs same class} implies that for any $\alpha < \frac{\sqrt{2}}{20}$, there exists some $\delta > 0$ such that for any $\Delta t > 0$, $\{h_{j}(\cdot,t)\}_{j\in J},\{h_{j}(\cdot, t+ \Delta t)\}_{j\in J}\in \Gamma_{\alpha}(\pi,\mathcal{U})$ determine the same true bundle as long as $d_{\Gamma}(\{h_{j}(\cdot, t)\}_{j\in J}, \{h_{j}(\cdot, t+ \Delta t)\}_{j\in J}) < \delta$. 
    By continuity, it follows that $\{h_{j}(\cdot, t)\}_{j\in J}$ determines the same true bundle for all $t\in [0,1]$, and in particular that the true bundles determined by $\{f_{j}\}_{j\in J}$ and $\{g_{j}\}_{j\in J}$ coincide. \\ 
\end{proof}

\begin{remark}
    Discrete approximate local angle functions satisfying the hypotheses of Corollary~\ref{cor: homotopic trivs} above for some $\alpha$ are said to be \textbf{homotopic through} $\Gamma_{\alpha}(\pi,\mathcal{U})$.\\
\end{remark}

Given a function $\pi:X\to B$ between metric spaces and an open cover $\mathcal{U} = \{U_{j}\}$ of $B$, we define $\Gamma_{0}(\pi,\mathcal{U})$ to be the set of families $\{f_{j}:U_{j}\to \mathbb{S}^{1}\}_{j\in J}$ for which there exists some \v{C}ech cocycle $\Omega\in \check{Z}^{1}(\mathcal{U};\mathcal{C}_{O(2)})$ satisfying $f_{j}(b) = \Omega_{jk}(b)f_{k}(b)$ for all $U_{jk}\in \mathcal{N}(\mathcal{U})$ and $b\in U_{jk}$. 
Note that $\Omega$ need not be locally constant. 
We provide a construction for producing a family $\{\widetilde{f}_{j}\}_{j\in J}\in \Gamma_{0}(\pi,\mathcal{U})$ from a family $\{f_{j}\}_{j\in J}\in \Gamma_{\alpha}(\pi,\mathcal{U})$ when $\alpha$ is sufficiently small and $\mathcal{U} = \{U_j\}_{j\in J}$ is locally finite. 
Our construction uses the notion of a Frechet mean~\cite{Karcher1977} for angular values in $\mathbb{S}^{1}$, which generalizes the notion of a centroid in Euclidean space:\\

\begin{definition}\label{def: Karcher mean}
    Let $M$ be a complete metric space, and let $m_{1},...,m_{N}\in M$. For any $p\in M$, the \textbf{(weighted) Frechet variance} of the set $\{m_{1},...,m_{N}\}$ at $p$ with weights $\{w_{1},...,w_{N}\}$ is defined by

\begin{equation*}
    F_{V}(p) = \sum_{n=1}^{N}w_{n}d_{M}(p,x_{n})^{2}
\end{equation*}

    The point $p$ is called a $\textbf{(weighted) Karcher mean}$ of $\{m_{1},...,m_{N}\}$ if it minimizes $F_{V}$ on $M$. If $p$ is the \textit{unique} weighted Karcher mean of $\{m_{1},...,m_{N}\}$, it is called the $\textbf{(weighted) Frechet mean}$ of $\{m_{1},...,m_{N}\}$.\\
    
\end{definition}

\begin{construction}\label{const: triv projection}
    Suppose $\{f_{j}\}_{j\in J}\in \Gamma_{\alpha}(\pi,\mathcal{U})$ for some locally-finite good open cover $\mathcal{U}$. 
    Let $\Omega\in\check{C}^{1}\left(\mathcal{U};\underline{O(2)}\right)$ be any witness, and let $\widetilde{\Omega} = \Pi_{Z}(\Omega)$. 
    Choose a partition of unity $\{\rho_{j}\}_{j\in J}$ subordinate to $\mathcal{U}$, and for each $j$, define $\widetilde{f}_{j}:\pi^{-1}(U_{j})\to \mathbb{S}^{1}$ by

\begin{equation}\label{eq: Karcher mean}
    \widetilde{f}_{j}(x) \quad = \quad \exp\left(\sum_{k \in J} \rho_{k}(\pi(x))\log\left(\widetilde{\Omega}_{jk}(\pi(x))f_{k}(x)\right)\right)
\end{equation}

    \noindent Here we fix a branch of log for each $x\in X$ and each $j\in J$ as follows. 
    Since  $\mathcal{U}$ is locally finite, then  the set $\left\{\widetilde{\Omega}_{jk}(\pi(x))f_{k}(x)\in \mathbb{S}^{1} \, : \, k\in J\right\}$ is finite, and thus is contained in the domain of some branch of log. 
    Note that if $\pi(x)\notin U_{jk}$ then $\rho_{k}(\pi(x)) = 0$, so $\tilde{f}_j(x)$ is well-defined. \\
\end{construction}

\begin{proposition}\label{prop: triv projection}
    Suppose $\Omega$ witnesses that $\mathfrak{F}=\{f_{j}\}_{j\in J}\in \Gamma_{\alpha}(\pi,\mathcal{U})$ for some good open cover $\mathcal{U}$. 
    Then the functions $\{\widetilde{f}_{j}\}_{j\in J}$ described in Construction~\ref{const: triv projection} form a system  $\widetilde{\mathfrak{F}}\in \Gamma_{0}(\pi,\mathcal{U})$.
\end{proposition}

\begin{proof}
    Let $\exp:\mathbb{R}\to SO(2)$ be as defined in~\eqref{eq: commutative exponential diagram} and define $r:\mathbb{Z}_{2}\to O(2)$ by $r(\pm1) = \left(\begin{array}{cc}
    1 & 0\\
    0 & \pm 1\\
    \end{array}\right)$.     
    Let $\widetilde{\Omega} = \Pi_{Z}(\Omega)$. 
    Define $\omega = \det_{*}(\Omega)$, and for each $U_{jk}\in \mathcal{N}(\mathcal{U})$, let $R_{jk}:U_{jk}\to SO(2)$ be defined by $R_{jk}(b) = \Omega_{jk}(b) r(\omega_{jk})$. 
    Since $\mathcal{U}$ is a good open cover, each intersection $U_{jk}$ is contractible, so each $R_{jk}$ has a continuous lift $\Theta_{jk}:U_{jk}\to \mathbb{R}$ satisfying $\exp (\Theta_{jk}(b)) = R_{jk}(b)$. 
    Thus, $\Omega$ decomposes as $\Omega_{jk}(b) = \exp(\Theta_{jk}(b))r(\omega_{jk})$. 
    For each $x\in X$ and $j\in J$ such that $\pi(x)\in U_{j}$, choose a branch of $\log$ whose domain contains the (finite) set $\{\tilde{\Omega}_{jk}(\pi(x))f_{k}(x)\in \mathbb{S}^{1}: \pi (x)\in U_{k}\}$.
    Note that whenever $\pi (x) = b\in \pi^{-1}(U_{j}\cap U_{k})$ for some $j,k\in J$, we have
\[
    \log\left(\widetilde{\Omega}_{jk}(b)f_{k}(x)\right) \ = \  \kappa_{jk}(b) + \Theta_{jk}(b) + \omega_{jk}\log \left(f_{k}(x)\right)    
\]
    for some $\kappa_{jk}(\pi(x))\in \mathbb{Z}=\ker(\exp)$, so 

    \begin{align*}
        \tilde{\Omega}_{jk}(b)\tilde{f}_{k}(x) \quad &= \quad \tilde{\Omega}_{jk}(b)\exp\left(\sum_{l}\rho_{l}(b) \log \left(\tilde{\Omega}_{kl}f_{l}(x)\right)\right)\\
        &= \quad \exp(\Theta_{jk}(b))r(\omega_{jk})\exp\left(\sum_{l}\rho_{l}(b) \log \left(\tilde{\Omega}_{kl}(b)f_{l}(x)\right)\right)\\
        &= \quad \exp\left(\Theta_{jk}(b) + \omega_{jk}\sum_{l}\rho_{l}(b) \log \left(\tilde{\Omega}_{kl}(b)f_{l}(x)\right)\right)\\
        &= \quad \exp\left(\sum_{l}\rho_{l}(b)\left(\Theta_{jk}(b) + \omega_{jk} \log \left(\tilde{\Omega}_{kl}(b)f_{l}(x)\right)\right)\right)\\ 
        &= \quad \exp\left(\sum_{l}\rho_{l}(b) \log \left(\tilde{\Omega}_{jk}(b)\tilde{\Omega}_{kl}(b)f_{l}(x)\right)\right)\\
        &= \quad \exp\left(\sum_{l}\rho_{l}(b) \log \left(\tilde{\Omega}_{jl}(b)f_{l}(x)\right)\right)\\
        &= \quad \tilde{f}_{j}(x)
    \end{align*}
    
\end{proof}

\begin{proposition}\label{prop: Karcher means}
    Suppose $\Omega$ witnesses that $\{f_{j}\}_{j\in J}\in T_{\alpha}(\pi,\mathcal{U})$ and $d_{Z}(\Omega,\widetilde{\Omega}) < \delta$, where $\widetilde{\Omega} = \Pi_{Z}(\Omega)$. 
    If $\delta +\alpha < \sqrt{2}$, then for any partition of unity $\{\rho_{j}\}_{j\in J}$ subordinate to $\mathcal{U}$ and $x\in U_{j}$ for some $j\in J$, $\widetilde{f}_{j}(x)$ as defined by~\eqref{eq: Karcher mean} is equal to the Frechet mean of $\{\widetilde{\Omega}_{jk}(\pi(x))f_{k}(x)\}_{k\in K_{x}}$ with weights $\{\rho_{k}(\pi(x))\}_{k\in K_{x}}$, where $K_{x} = \{k\in J: x\in U_{jk}\}$. 
    Furthermore $d_{\Gamma}(\{f_{j}\}_{j\in J},\{\widetilde{f}_{j}\}_{j\in J}) < \alpha + \delta$, and the functions $\{\widetilde{f}_{j}\}_{j\in J}$ obtained by different choices of partition of unity are homotopic through $\Gamma_{0}(\pi,\mathcal{U})$. \\
\end{proposition}

\begin{proof}
    For any $(jk)\in \mathcal{N}(\mathcal{U})$ and $x\in \pi^{-1}(U_{jk})$, we have

\begin{align*}
    |f_{j}(x) - \widetilde{\Omega}_{jk}(\pi(x))f_{k}(x)| \quad &\leq \quad |f_{j}(x) - \Omega_{jk}f_{k}(x)| + |\Omega_{jk}f_{k}(x) - \widetilde{\Omega}_{jk}(\pi(x))f_{k}(x)|\\
    & < \quad \alpha + \|\Omega_{jk} - \widetilde{\Omega}_{jk}(\pi(x))\|_{F}\\
    & \leq \quad \alpha + \delta
\end{align*}

    Thus, by Proposition~\ref{prop: Karcher}, $\widetilde{f}_{j}(x)$ as defined in~\eqref{eq: Karcher mean} is the unique weighted Karcher mean of the set $\{\widetilde{\Omega}_{jk}(\pi(x))f_{k}(x)\}_{k\in K_{x}}$ as long as $\alpha +\delta < \sqrt{2}$. 
    Clearly we have $d_{\Gamma}(\{f_{j}\}_{j\in J}, \{\widetilde{f}_{j}\}_{j\in J})\leq \alpha + \delta$ in this case. \\

    \noindent Now, suppose $\{\rho'_{j}\}_{j\in J}$ is another partition of unity subordinate to $\mathcal{U}$. 
    Each partition of unity can be viewed as the component functions of maps $\rho,\rho':B\to \Delta^{J}$ to the full simplex on $J$. 
    Since $\Delta^{J}$ is convex, we can form the straight line homotopy $h_{t}:B\to \Delta^{J}$ from $\rho$ to $\rho'$.
    Replacing the components of $\rho$ with those of $h_{t}$ in~\eqref{eq: Karcher mean} for each $t\in [0,1]$ gives a continuously varying family in $\Gamma_{0}(\pi,\mathcal{U})$, so $\{f_{j}\}_{j\in J}$ and $\{\widetilde{f}_{j}\}_{j\in J}$ are homotopic through $\Gamma_{0}(\pi,\mathcal{U})$.\\ 
\end{proof}

\begin{corollary}
    For $\alpha < \frac{1}{28}$, Construction~\ref{const: triv projection} defines a function $\Pi_{\Gamma}:\Gamma_{\alpha}(\pi,\mathcal{U})\to \Gamma_{0}(\pi,\mathcal{U})$ such that for any $\mathfrak{F}\in \Gamma_{\alpha}(\pi,\mathcal{U})$, the systems $\mathfrak{F}$ and $\Pi_{\Gamma}(\mathfrak{F})$ determine the same true bundle. \\
\end{corollary}

\begin{proof}
    Suppose $\mathfrak{F}\in \Gamma_{\alpha}(\pi,\mathcal{U})$ for $\alpha < \frac{1}{28}<\frac{\sqrt{2}}{28} < \frac{\sqrt{2}}{20}$. 
    By Proposition~\ref{prop: triv to cocycle}, any witness satisfies $\Omega\in\check{Z}_{(3\sqrt{2})\alpha}\left(\mathcal{U};\underline{O(2)}\right)$, and Proposition~\ref{prop: cocycle projection} implies $\delta = d_{Z}(\Omega,\widetilde{\Omega}) < (27\sqrt{2})\alpha$, where $\widetilde{\Omega} = \Pi_{Z}(\Omega)$. 
    Thus, 
\[
\alpha + \delta \ < \  (1+ 27\sqrt{2})\alpha \ < \ \frac{1 + 27\sqrt{2}}{28} \ < \ \sqrt{2}
\]
    so the result follows by Proposition~\ref{prop: Karcher means}.\\ 
\end{proof}

To describe conditions under which a discrete approximate circle bundle can be uniquely identified with a class of true bundles, we introduce the following notation:\\

\begin{definition}
    For any $r>0$, $\beta,\eps\geq 0$ and collection of open sets $\mathcal{U}$, define 
    \begin{equation}\label{eq: general tau}
    \tau(r, \beta,\eps,\mathcal{U})=  \frac{1}{r}\left(\sup_{j\in J}\ \text{diam}(U_{j})+10\eps+9\beta\right),\hspace{1cm}\tilde{\tau}(r,\beta,\eps,\mathcal{U})=2\sin^{-1}\left(\frac{\tau(r,\beta,\eps,\mathcal{U})}{2}\right)
\end{equation}

    Observe that if $\pi:X\to B$ is a discrete $(r,\beta,\eps)$-approximate circle bundle and $\Phi\in T_{r,\beta,\eps}(\pi,\mathcal{U})$, then 
\[
\tau(\Phi) \ \leq \ \tau(r,\beta,\eps,\mathcal{U}),\hspace{1cm}\widetilde{\tau}(\Phi) \ \leq \ \widetilde{\tau}(r,\beta,\eps,\mathcal{U})
\]

\end{definition}

The following result shows that witnesses to systems of discrete approximate local angle functions emerging from  trivializations subordinate to the same cover are roughly related by a change of gauge:\\

\begin{lemma}\label{lemma: approximate gauge change}
    Let $\pi:X\to B$ be a discrete $(r,\beta,\eps)$-approximate circle bundle for   $r>0 $, $\beta,\eps\geq 0$ with trivializing cover $\mathcal{U} = \{U_{j}\}_{j\in J}$, and let $\Omega,\Omega'\in \check{C}^{1}(\mathcal{U};\underline{O(2)})$ respectively witness that $\mathfrak{F}(\Phi),\mathfrak{F}(\Phi')\in \Gamma_{\alpha}(\pi,\mathcal{U})$ for some $\alpha \leq \max\{\alpha(\Phi),\alpha (\Phi')\}$.    
    There exists some $\zeta\in\check{C}^{0}(\mathcal{U};\underline{O(2)})$ such that 

\[
\|\Omega_{jk}' - \zeta_{j}\Omega_{jk}\zeta_{k}^{-1}\|_{F} \quad \leq \quad 5\sqrt{2}\ \tilde{\tau}(r,\beta,\eps,\mathcal{U})
\]

    for all $(jk)\in\mathcal{N}(\mathcal{U})$.

\end{lemma}

\begin{proof}
    Define
\[
    \tau(\Phi,\Phi')= \sup_{j\in J}\ \tau (\varphi_{j},\varphi_{j}'),\hspace{1cm}\tilde{\tau}(\Phi,\Phi')=2\sin^{-1}\left(\frac{\tau(\Phi,\Phi')}{2}\right)
\]

and observe

\[
    \tau (\Phi,\Phi')\ \leq \ \tau (r,\beta,\eps,\mathcal{U})\hspace{1cm}\tilde{\tau}(\Phi,\Phi')\ \leq \ \tilde{\tau}(r,\beta,\eps,\mathcal{U})
\]

    Lemma~\ref{lemma: chart pair approx transition} implies that for each $j$, there exists some $\zeta_{j}\in O(2)$ such that $d_{\infty}(f_{j}',\zeta_{j}f_{j}) \leq \tau (\Phi,\Phi')$. 
    Thus, for any $x\in \pi^{-1}(U_{jk})$,

\begin{align*}
    |f_{j}'(x) -  \zeta_{j}\Omega_{jk}\zeta_{k}^{-1}f_{k}'(x)| \quad &\leq \quad |f_{j}'(x) -  \zeta_{j}f_{j}(x)| \ + \ |\zeta_{j}f_{j}(x) -  \zeta_{j}\Omega_{jk}f_{k}(x)| \ + \  |\zeta_{j}\Omega_{jk}f_{k}(x) -  \zeta_{j}\Omega_{jk}\zeta_{k}^{-1}f_{k}'(x)|\\
    &= \quad |f_{j}'(x) - \zeta_{j}f_{j}(x)| \ + \ |f_{j}(x) - \Omega_{jk}f_{k}(x)| \ + \ | \zeta_{k}f_{k}(x) - f_{k}'(x)|\\
    &\leq \quad \tilde{\tau}(\Phi,\Phi') + \tilde{\tau}(\Phi) + \tilde{\tau}(\Phi',\Phi)\\
    &=\quad  2\tilde{\tau}(\Phi,\Phi') + \tilde{\tau}(\Phi)
\end{align*}

    which implies

\begin{align*}
    |\Omega_{jk}'f_{k}'(x) - \zeta_{j}\Omega_{jk}\zeta_{k}^{-1}f_{k}'(x)| \quad &\leq \quad |\Omega_{jk}'f_{k}'(x) - f_{j}'(x)| + |f_{j}'(x), \zeta_{j}\Omega_{jk}\zeta_{k}^{-1}f_{k}'(x)|\\
    &\leq \quad \tilde{\tau}(\Phi') + (2\tilde{\tau}(\Phi,\Phi') + \tilde{\tau}(\Phi))\\
    &\leq \quad 4\tilde{\tau}(r,\beta,\eps,\mathcal{U})
\end{align*}

    for all $x\in \pi^{-1}(U_{jk})$.
     Using~\eqref{eq: Hausdorff bound} and applying Proposition~\ref{prop: Frobenius delta surjectivity} to the matrix $\Omega_{jk}' - \xi_{j}\Omega_{jk}\xi_{k}'$ , we obtain

\[
\|\Omega_{jk}' - \xi_{j}\Omega_{jk}\xi_{k}'\|_{F}\quad \leq \quad \frac{4\sqrt{2}}{1 - 2\sin^{-1}\left(\frac{\eps}{2r}\right)}\ \tilde{\tau}(r,\beta,\eps,\mathcal{U})\quad \leq \quad 5\sqrt{2}\ \tilde{\tau}(r,\beta,\eps,\mathcal{U})
\]

\end{proof}

We now arrive at the main result of this section: \\

\begin{theorem}\label{thm: disc CB to true}
    Let $\xi = \pi:X\to B$ be a discrete $(r,\beta,\eps)$-approximate circle bundle for some $r,\beta,\eps\geq 0$, and let $\mathcal{U} = \{U_{j}\}_{j\in J}$ be a trivializing cover.
    If $\widetilde{\tau}(r,\beta,\eps,\mathcal{U}) \leq \frac{\sqrt{2}}{36}$, then for any $\Phi\in T_{r,\beta,\eps}(\pi,\mathcal{U})$,
     the cocycle $\Pi_{Z}\circ \text{wit}(\mathfrak{F}(\Phi))\in \check{Z}^{1}(\mathcal{U};\mathcal{C}_{O(2)})$ is well-defined, and its cohomology class does not depend on $\Phi$.\\
\end{theorem}

\begin{proof}
    By Proposition~\ref{prop: trivs to ang functions}, we have $\mathfrak{F}(\Phi)\in \Gamma_{\alpha(\Phi)}(\pi, \mathcal{U})$, where $\alpha(\Phi)$ is defined by~\eqref{eq: alpha trivializations}.
    Thus, Proposition~\ref{prop: triv to cocycle} implies $\text{wit}(\mathfrak{F}(\Phi))\in\check{Z}^{1}_{(3\sqrt{2})\alpha(\Phi)}(\mathcal{U};\mathcal{C}_{O(2)})$. 
    By Proposition~\ref{prop: cocycle projection}, $\Pi_{Z}\circ \text{wit}(\mathfrak{F}(\Phi))\in\check{Z}^{1}(\mathcal{U};\mathcal{C}_{O(2)})$ is well-defined as long as $(3\sqrt{2})\alpha(\Phi) \leq \frac{1}{2}$, which is satisfied as long as 
$\widetilde{\tau}(r,\beta,\eps,\mathcal{U}) \leq \frac{\sqrt{2}}{15}$.
    
    Now, suppose we have some other $\Phi'\in T_{r,\eps,\beta}(\pi,\mathcal{U})$.
    Let $\Omega = \text{wit}(\mathfrak{F}(\Phi))$ and let $\Omega'=\text{wit}(\mathfrak{F}(\Phi'))$
    Proposition~\ref{prop: approx cocycle or map determines true} implies that if $3\sqrt{2}\max\{\alpha(\Phi),\alpha(\Phi')\} + d_{H}([\Omega],[\Omega'])\leq \frac{1}{2}$, then $[\Pi_{Z}(\Omega')] = [\Pi_{Z}(\Omega)]$.    
    By Lemma~\ref{lemma: approximate gauge change}, $d_{H}([\Omega], [\Omega'])\leq 5\sqrt{2}\ \tilde{\tau}(r,\beta,\eps,\mathcal{U}),$ so

\[
    3\sqrt{2}\max\{\alpha(\Phi),\alpha(\Phi')\} + d_{H}([\Omega],[\Omega'])\quad \leq \quad \frac{5}{4}(3\sqrt{2})\widetilde{\tau}(\Phi) + (5\sqrt{2})\widetilde{\tau}(r,\beta,\eps,\mathcal{U})\quad \leq \quad (9\sqrt{2})\tilde{\tau}(r,\beta,\eps,\mathcal{U})
\]

    We conclude that $[\Pi_{Z}(\Omega)] = [\Pi_{Z}(\Omega')]$ as long as $\widetilde{\tau}(r,\beta,\eps,\mathcal{U})\leq \frac{1}{18\sqrt{2}}=\frac{\sqrt{2}}{36}$.\\
\end{proof}

\begin{remark}
    The cohomology class assigned to $\xi$ in the proposition above is also independent of the trivializing cover $\mathcal{U}$ following sense: suppose $\Phi'\in T_{r,\beta,\eps}(\pi,\mathcal{V})$ for some other trivializing cover $\mathcal{V}=\{V_{j'}\}_{j'\in J'}$ satisfying $\widetilde{\tau}(r,\beta,\eps,\mathcal{V}) < \frac{\sqrt{2}}{36}$.
    If $\mathcal{W}$ is a common refinement of $\mathcal{U}$ and $\mathcal{V}$, then the images of $[\Pi_{Z}\circ\text{wit}(\mathfrak{F}(\Phi))]$ and $[\Pi_{Z}\circ\text{wit}(\mathfrak{F}(\Phi'))]$ with respect to the induced inclusions coincide in $\check{H}^{1}(\mathcal{W};\mathcal{C}_{O(2)})$.\\ 
\end{remark}

\section{Effective Computation Of Characteristic Classes}\label{sec: Char Classes}

In this section, we give explicit algorithms for computing the characteristic classes of a circle bundle from a discrete approximate cocycle representative. One can then use Construction~\ref{const: maxmin Procrustes} to compute classes from discrete approximate trivializations. We also introduce a notion of persistence for these classes. \\

We first note some simplifications which occur in the computation of the twisted Euler class when the transition functions are $O(2)$-valued (rather than $\text{Homeo}(\mathbb{S}^{1})$-valued). Let $r:\mathbb{Z}_{2}\to \text{Homeo}(\mathbb{S}^{1})$ denote the section of~\eqref{eq: Homeo+ central ext} described in Section~\ref{sec: True Circle Bundles}, and let $\varphi:\text{Homeo}_{+}(\mathbb{S}^{1})\rtimes_{r}\mathbb{Z}_{2}\to \text{Homeo}(\mathbb{S}^{1})$ be the group isomorphism described there given by $\varphi(\Lambda,\omega)=\Lambda r(\omega)$. Relative to the reduced structure group $O(2)$, we have $r(\pm 1) = \left(\begin{array}{cc}
    1 & 0\\
    0 & \pm 1\\
    \end{array}\right)$. Moreover,
    Proposition~\ref{prop: reflections} implies $\text{Adj}_{r(\omega)}(\Lambda) = \Lambda^{r(\omega)}$ for any $\Lambda\in O(2)$ and $\omega = \pm 1\in\mathbb{Z}_{2}$, so the group multiplication law in $\mathbb{S}^{1}\rtimes_{r}\mathbb{Z}_{2}$ simplifies to

\begin{equation}\label{eq: semidirect group law}
    (\Lambda_{1},\omega_{1})\cdot (\Lambda_{2},\omega_{2}) = (\Lambda_{1}\Lambda_{2}^{\omega_{1}}, \omega_{1}\omega_{2})
\end{equation}

\noindent The cocycle condition for $(\mathbb{S}^{1}\rtimes\mathbb{Z}_{2})$-valued cocycles is therefore given explicitly by $(\Lambda_{jk},\omega_{jk})\cdot (\Lambda_{kl},\omega_{kl}) = (\Lambda_{jl},\omega_{jl})$, or, equivalently,

\begin{equation}\label{eq: semidirect cocycle cond}
\Lambda_{jk}\Lambda_{kl}^{\omega_{jk}} = \Lambda_{jl}\quad\quad \omega_{jk}\omega_{kl}=\omega_{jl}\\
\end{equation}

Thus, for any $[\Lambda] \in \check{H}^{1}(\mathcal{U};\mathcal{C}_{O(2),\omega})$,
the connecting homomorphism $\Delta_{\omega}$ described in Proposition~\ref{prop: twisted exp sequence} returns a class $\Delta_{\omega}([\Lambda])\in \check{H}^{2}(\mathcal{U};\underline{\mathbb{Z}})$ represented by the cocycle whose $(jkl)$-component is given by

\begin{equation}\label{eq: twisted Euler component}
    \omega_{jk}\Theta_{kl}(b) - \Theta_{jl}(b) + \Theta_{jk}(b)
\end{equation}

for all $b\in U_{j}\cap U_{k}\cap U_{l}$, where $\Theta\in\check{C}^{1}(\mathcal{U};\mathcal{C}_{\mathbb{R}})$ is a lift of $\Lambda$ via $\exp_{*}$.\\

For later use, we define 
\begin{equation}
\begin{array}{rccl}
    \text{abs}:  & O(2) & \longrightarrow & SO(2) \\
     & A&\mapsto &  A r(\det(A)) 
\end{array}
\end{equation}
and note that $\varphi^{-1}(A) = (\text{abs}(A), \det (A))$. In fact, it is easy to see that if $\Omega\in\check{Z}^{1}_{\varepsilon}(\mathcal{U};\mathcal{C}_{O(2)})$ for some $\varepsilon \geq 0$ and $\omega = \det_{*}(\Omega)$, then the cochain $\Lambda\in \check{C}^{1}(\mathcal{U};\mathcal{C}_{SO(2)})$ given by $\Lambda_{jk}(b) = \text{abs}(\Omega_{jk}(b))$ for all $U_{j}\cap U_{k}\in \mathcal{N}(\mathcal{U})$ and $b\in U_{j}\cap U_{k}$ lies in $\check{Z}^{1}_{\varepsilon}(\mathcal{U};\mathcal{C}_{SO(2),\omega})$. We will denote the cochain $\Lambda$ so defined by $\text{abs}(\Omega)$.\\

Given a real number $x$, we use the notation $\langle x\rangle$ to denote the nearest integer to $x$ (if the integer is not unique, we take the larger one). Similarly, for any cochain $\lambda\in \check{C}^{1}(\mathcal{U};\mathcal{C}_{\mathbb{R}})$, we denote the nearest integer-valued cochain (with respect to $d_{Z}$) by $\langle \lambda\rangle$.\\

Below we give an explicit algorithm for computing characteristic classes from a discrete approximate representative. To prove stability, we first prove a generalization of a technical result from~\cite{Euclidean_Vector_Bundles} (Theorem 6.6). Let $1\to F\to G\xrightarrow{\psi} H\to 1$ be a central extension of Lie groups with H compact and F discrete. Fix a bi-invariant Riemannian metric on $H$ and use it to metrize $H$ and $G$ with the geodesic distance and $F$ with the distance inherited from $G$. Suppose $F$, $G$ and $H$ have left $\Gamma$-actions by isometries for some abelian group $\Gamma$, and that the maps in the central extension are $\Gamma$-equivariant. Let $\mathcal{U}$ be an open cover of a topological space $B$ such that each set and non-empty binary intersection is locally path connected and simply connected. Let $\text{sys}(H)$ denote the systole of $H$. \\

\begin{proposition}\label{prop: char class stability}
 For $\varepsilon \leq \frac{\text{sys}(H)}{8}$ and any $\omega\in\check{Z}^{1}(\mathcal{U};\underline{\Gamma})$, we have a well-defined map $\beta_{\omega}:\check{Z}_{\varepsilon}^{1}(\mathcal{U};\mathcal{C}_{H,\omega})\to \check{H}^{2}(\mathcal{U};F_{\omega})$ given by $\Delta_{\omega}(\Lambda)=[\langle\beta^{1}_{\omega}\Theta\rangle]$, where $\Theta\in\check{C}^{1}(\mathcal{U};\mathcal{C}_{G,\omega})$ is any lift of $\Lambda$ respect to $\psi_{*}$ and $\langle\cdot\rangle$ denotes the unique nearest integer 
 Furthermore, $\Delta_{\omega}$ descends to a map $\Delta_{\omega}:\check{H}^{1}_{\varepsilon}(\mathcal{U};\mathcal{C}_{H,\omega})\to \check{H}^{2}(\mathcal{U};F_{\omega})$ which is stable in the sense that if $[\Lambda],[\Lambda']\in\check{H}^{1}_{\varepsilon}(\mathcal{U};\mathcal{C}_{H,\omega})$ satisfy $d_{H}([\Lambda],[\Lambda'])<\frac{\text{sys}(H)}{8}$, then $\Delta_{\omega}([\Lambda]) = \Delta_{\omega}([\Lambda'])$.\\
\end{proposition}

\begin{remark}
    If $G$ is a metric group and $\Gamma$ is an abelian group which acts on $G$ by isometries, the $d_{Z}$ metric on $\check{C}^{1}(\mathcal{U};\mathcal{C}_{G})$ naturally extends to a metric on $\check{C}^{1}(\mathcal{U};\mathcal{C}_{G,\omega})$ for any $\omega\in\check{Z}^{1}(\mathcal{U};\underline{\Gamma})$. We then define $\check{Z}^{1}_{\varepsilon}(\mathcal{U};\mathcal{C}_{G,\omega})$ and $\check{H}^{1}_{\varepsilon}(\mathcal{U};\mathcal{C}_{G,\omega})$ for all $\varepsilon\geq 0$ in the obvious way. \\
\end{remark}

\begin{proof}
    The proof is essentially the same as that of Theorem 6.6 in~\cite{Euclidean_Vector_Bundles}. The fundamental difference is that we are doing cohomology with local coefficients, so the coboundary maps are slightly different. On the other hand, because $\Gamma$ acts by isometries, the same arguments apply. In particular, for any $\Lambda\in\check{Z}^{1}_{\varepsilon}(\mathcal{U};\mathcal{C}_{H,\omega})$, one can use the same steps shown in~\cite{Euclidean_Vector_Bundles} to verify the 2-cochain given by $\beta_{jkl} = \langle (\delta_{\omega}^{1}\Theta)_{jkl}\rangle$ is a genuine cocycle, and the stability arguments are identical. \\   
\end{proof}

\noindent An immediate result of the proposition above is the following: \\

\begin{corollary}\label{cor: char of true bundle}
    For $\varepsilon\leq 1$, we have a well-defined function 
    \[
    \text{Cl}:\check{H}_{\varepsilon}^{1}(\mathcal{U};\mathcal{C}_{O(2)})\to\coprod_{[\omega]\in\check{H}^{1}(\mathcal{U};\underline{\mathbb{Z}_{2}})}\check{H}^{2}(B;\mathbb{Z}_{\omega})
    \]
    given by $\Omega\longmapsto([\det_{*}\Omega], \Delta_{\omega}(\text{abs}(\Omega)))$. This map is 1-stable in the sense that $\text{Cl}(\Omega)=\text{Cl}(\Omega')$ whenever $d_{H}(\Omega,\Omega') < 1$. Moreover, if $\varepsilon < \frac{1}{9}$, then  $\text{Cl}(\Omega)$ returns the first Stiefel-Whitney class and twisted Euler class of the true circle bundle represented by $\Pi_{Z}(\Omega)$. \\
\end{corollary}

\begin{proof}
    Suppose $\Omega\in\check{Z}^{1}_{\varepsilon}(\mathcal{U};\mathcal{C}_{O(2)})$ for some $\varepsilon\leq 1$. Lemma 1.3 of~\cite{Euclidean_Vector_Bundles} implies $\omega = \det_{*}(\Omega)$ is a cocycle and that $\det_{*}(\Omega) = \det_{*}(\Omega')$ for any other $\Omega'\in\check{Z}^{1}_{1}(\mathcal{U};\mathcal{C}_{O(2)})$ such that $d_{Z}(\Omega,\Omega')<2$. If $\varepsilon < 1/9$, then by Proposition~\ref{prop: cocycle projection}, there exists a true cocycle $\Omega'$ such that $d_{Z}(\Omega,\Omega')< 1$, so $\det_{*}(\Omega)=\det_{*}(\Omega')$ for any such $\Omega'$. We conclude that the first Stiefel-Whitney class of any such $\Omega'$ is the same. \\

    \noindent Now, recall central extension~\eqref{eq: commutative exponential diagram}. Let $SO(2)_{g}$ denote the Lie group $SO(2)$ endowed with the geodesic distance metric (via the natural identification of $SO(2)$ with $\mathbb{S}^{1}$). The induced metrics on $\mathbb{R}$ and $\mathbb{Z}$ are then the usual ones, and all three groups have $\mathbb{Z}_{2}$-actions by isometries such that the inclusion and $\exp$ maps are equivariant. Thus, Proposition~\ref{prop: char class stability} gives a well-defined map $\Delta_{\omega}:\check{Z}^{1}_{\varepsilon}(\mathcal{U};\mathcal{C}_{SO(2)_{g},\omega})\to \check{H}^{2}(\mathcal{U};\mathbb{Z}_{\omega})$ for all $\varepsilon\leq \frac{sys(SO(2)_{g})}{8}$, where $\omega=\det_{*}(\Omega)$. By Lemma A.18 from~\cite{Euclidean_Vector_Bundles}, $\frac{sys(SO(2)_{g})}{8} = \frac{2\pi\sqrt{2}}{8} = \frac{\pi\sqrt{2}}{4}$ with respect to geodesic distance. To obtain a result in terms of the Frobenius norm, we note Corollary A.17 from~\cite{Euclidean_Vector_Bundles} implies we have an inclusion $\check{Z}_{1}^{1}(\mathcal{U};\mathcal{C}_{SO(2),\omega})\subset \check{Z}^{1}_{\frac{\pi\sqrt{2}}{4}}(\mathcal{U};\mathcal{C}_{SO(2)_{g},\omega})$. Thus, $\text{abs}(\Omega)\in\check{Z}^{1}_{1}(\mathcal{U};\mathcal{C}_{SO(2),\omega})$, so $\Delta_{\omega}(\text{abs}(\Omega))\in\check{H}^{2}(\mathcal{U};\mathbb{Z}_{\omega})$ (and therefore $\text{Cl}(\Omega)$) is well-defined. \\

    \noindent Finally, if $\Omega\in\check{Z}^{1}_{1/9}(\mathcal{U};\mathcal{C}_{O(2)})$, then by Proposition~\ref{prop: cocycle projection} there exists a true cocycle $\widetilde{\Omega} \in \check{Z}^1(\mathcal{U}; \mathcal{C}_{O(2)})$ such that $d_{Z}(\Omega,\widetilde{\Omega})<1$, and Proposition~\ref{prop: char class stability} implies  $\Delta_{\omega}(\widetilde{\Omega}) = \Delta_{\omega}(\Omega)$ for any such $\widetilde{\Omega}$. It follows that $\text{Cl}(\Omega) = \text{Cl}(\widetilde{\Omega})$, and in particular that $[\det_{*}(\Omega)]$ and $\Delta_{\omega}(\text{abs}(\Omega))$ are the first Stiefel-Whitney class and twisted Euler class of the true circle bundle represented by any such $\widetilde{\Omega}$.\\
\end{proof}

\begin{remark}
    Because the isomorphism class of a circle bundle is completely determined by its characteristic classes, Proposition~\ref{prop: char class stability} implies that if $\varepsilon < \frac{1}{9}$, then $\textit{any}$ true cocycle $\widetilde{\Omega}$ such that $d_{Z}(\Omega,\widetilde{\Omega})< 1$ must satisfy $\text{Cl}(\widetilde{\Omega}) = \text{Cl}(\Omega)$.\\
\end{remark}

\noindent Below is an explicit algorithm for computing characteristic class representatives: \\

\begin{algo}[Computation Of Characteristic Classes]\label{alg:CharClasses}
\begin{algorithmic}[1]
  \vspace{5mm}
    \Statex\vspace{1ex}

  \Require A cochain $\Omega\in\check{C}^{1}\!\left(\mathcal{U};\underline{O(2)}\right)$ subordinate to a good, finite open cover $\mathcal{U}$.
  
  \vspace{5mm}
  \Ensure A pair of cochains $\omega\in\check{C}^{1}(\mathcal{U};\underline{\mathbb{Z}_{2}})$ and $\tilde{e}\in \check{C}^{2}(\mathcal{U};\mathbb{Z}_{\omega})$

\vspace{5mm}

    \State Compute $\omega_{jk} = \det(\Omega_{jk})$ for all $(jk)\in\mathcal{N}(\mathcal{U})$.

\vspace{5mm}
  
  \State For each $(jk)\in\mathcal{N}(\mathcal{U})$, set
        $\Lambda_{jk} = \Omega_{jk}\, r(\omega_{jk})$ with
        $r(\pm 1)=\begin{bmatrix}1&0\\[2pt]0&\pm 1\end{bmatrix}$,
        then choose a lift $\Theta_{jk}\in\mathbb{R}$ via the $\exp$ map in~\eqref{eq: commutative exponential diagram}.

\vspace{5mm}
        
  \State For each $(jkl)\in\mathcal{N}(\mathcal{U})$, compute
        $\tilde{e}_{jkl} = \big\langle\, \omega_{jk}\Theta_{kl} - \Theta_{jl} + \Theta_{jk}\,\big\rangle$.
\end{algorithmic}
\end{algo}

\vspace{5mm}

\begin{corollary}\label{prop: algorithm computes classes}
    If $\mathcal{U}$ is a good, finite open cover and $\Omega\in\check{Z}^{1}_{\varepsilon}\left(\mathcal{U};\underline{O(2)}\right)$ for some $\varepsilon < \frac{1}{9}$, the outputs $\omega$ and $\tilde{e}$ of Algorithm~\ref{alg:CharClasses} represent the characteristic classes associated with $[\Pi_{Z}(\Omega)]\in \check{H}^{1}(\mathcal{U};\mathcal{C}_{O(2)})$.\\
\end{corollary}

\begin{proof}
    By Corollary~\ref{cor: char of true bundle}, the classes $[\det_{*}(\Omega)]$ and $\Delta_{\omega}(\text{abs}(\Omega))$ are respectively the first Stiefel-Whitney class and twisted Euler class of the true circle bundle represented by $\Pi_{Z}(\Omega)$. Expression~\eqref{eq: twisted Euler component} shows that $\widetilde{e}$ (as defined in~\ref{alg:CharClasses} above) represents $\Delta_{\omega}(\text{abs}(\Omega))$.\\ 
\end{proof}

\noindent In the case where $B$ is a closed connected 2-manifold, the following describes how to explicitly compute the twisted Euler number associated with a twisted Euler class: \\

\begin{algo}[Computation Of Twisted Euler Number]\label{alg:EulerNumber}
\begin{algorithmic}[1]
  \Statex
  \vspace{5mm}  
  
  \Require A finite good open cover $\mathcal{U}$ of a closed connected surface $B$ and cocycles $\omega\in \check{Z}^{1}(\mathcal{U};\underline{\mathbb{Z}_{2}})$, $\tilde{e}\in \check{Z}^{2}(\mathcal{U};\mathbb{Z}_{\omega})$.
  
  \vspace{5mm}
  \Ensure The twisted Euler number $\tilde{\eta}$ associated to $\tilde{e}$ (up to choice of sign).

\vspace{5mm}

  \State For $p=2,3$, construct the matrix representation $D_{\omega,p}$ for the twisted boundary operator $\partial_{\omega,p}:\check{C}_{p+1}(\mathcal{U};\mathbb{Z}_{\omega})\to \check{C}_{p}(\mathcal{U};\mathbb{Z}_{\omega})$ with respect to some total orderings on the simplices in each dimension (e.g., lexicographical). Explicitly, the non-zero entries of $D_{\omega,p}$ are given by

\begin{equation}\label{eq: twisted boundary matrix1}
    (D_{\omega,1})_{jkl}^{kl} = \omega_{jk}, \quad (D_{\omega,1})_{jkl}^{jl} = -1, \quad (D_{\omega,1})_{jkl}^{jk} = 1 
\end{equation}

    \noindent for each $(jkl)\in\mathcal{N}(\mathcal{U})$ and

\begin{equation*}\label{eq: twisted boundary matrix2}
    (D_{\omega,2})_{jklm}^{klm} = \omega_{jk}, \quad (D_{\omega,2})_{jklm}^{jlm} = -1, \quad (D_{\omega,2})_{jklm}^{jkm} = 1, \quad (D_{\omega,2})_{jklm}^{jkl} = -1 
\end{equation*}

    for each $(jklm)\in\mathcal{N}(\mathcal{U})$.

\vspace{5mm}
  
  \State Compute the Smith normal form $L D_{\omega,2} R = S$ of $D_{\omega,2}$ over $\mathbb{Z}$ and let $k$ denote the number of zero columns in $S$). If $k = 1$, set $\mu_{B,\omega}$ denote the last column of $R$ and proceed to Step 4. Otherwise, let $K$ be the submatrix of $R$ containing the last $k$ columns, and let $B$ equal the last $k$ rows of the product $R^{-1}D_{\omega,3}$.

\vspace{5mm}
        
  \State Compute the Smith normal form $L'BR' = S'$ of $B$. Let $v$ denote the last column of $R'$, and set $\mu_{B,\omega} = Kv$.

\vspace{5mm}

    \State Compute $\tilde{\eta} = \sum_{(jkl)\in\mathcal{N}(\mathcal{U})}\tilde{e}_{jkl}\cdot (\mu_{B,\omega})_{jkl}$.
    
\end{algorithmic}
\end{algo}

\vspace{5mm}

\begin{remark}
    Although Čech theory is typically formulated in terms of cohomology, there is a natural dual notion of Čech homology with the property that $\check{H}_{*}(\mathcal{U};\mathcal{L})\cong H_{*}(B;\mathcal{L})$ for any good open cover $\mathcal{U}$ of a space $B$ and any local system $\mathcal{L}$ of abelian groups (in particular, given a good open cover $\mathcal{U}$, we can \textit{define} $\check{C}_{k}(\mathcal{U};\mathcal{L}) = \text{Hom}_{\mathbb{Z}}(\check{C}^{k}(\mathcal{U};\mathcal{L}), \mathbb{Z})$ and consider the associated homology groups -- these groups then coincide with the corresponding singular homology. See~\cite{EilenbergSteenrod1952} for more details). We can therefore interpret $D_{\omega}$ in the construction above as the matrix representation of a twisted boundary operator whose kernel is generated by the associated twisted fundamental class. \\
\end{remark}

\begin{proposition}\label{prop: algorithm computes twisted Chern}
    If $B$ is a closed connected surface, the algorithm above computes the twisted Euler number associated to a twisted Euler class (up to choice of sign). \\
\end{proposition}

\begin{proof}
    Since $B$ is a compact connected surface, we have $\text{rank}(H_{2}(B;\mathbb{Z}_{\omega}))=1$, and any generator coincides with the fundamental class up to sign. 
    If $\mathcal{U}$ has no quadruple intersections, then $H_{2}(B;\mathbb{Z}_{\omega}) = Z_{2}(B;\mathbb{Z}_{\omega}) = \text{ker}\partial_{\omega,2}$, so the last column $r$ of the matrix $R$ above represents a generator of $H_{2}(B;\mathbb{Z}_{\omega})$ (specifically, the vector $r$ corresponds to the simplicial cycle $\sum_{(jkl)\in\mathcal{N}(\mathcal{U})^{(2)}}r_{jkl}\cdot U_{jkl}$). \\
    
    If $k > 1$, then the columns of $K$ form a basis for $\ker D_{\omega,2}$. Moreover, $\text{col}(D_{\omega,3})\subseteq \text{col}(K)\subseteq\text{col}(R)$, so only the last $k$ rows of $R^{-1}D_{\omega,3}$ are non-zero. The matrix $B$ expresses the columns of $D_{3,\omega}$ as linear combinations of the columns of $K$, so the kernel of $B$ consists of vectors which represent non-trivial 2-cycles. Since $\text{rank}(H_{2}(B;\mathbb{Z}_{\omega}))=1$ by assumption, we conclude $\ker B$ is spanned by the last column $v$ of $R'$ and $\mu_{B,\omega}=Kv$ represents a generator of $H_{2}(B;\mathbb{Z}_{\omega})$.\\

    Finally, pairing $\mu_{B}$ with $\widetilde{e}$ gives the twisted Euler number associated with $\widetilde{e}$ up to sign. If desired, one can calibrate the fundamental class representative $\mu_{B,\omega}$ obtained from the algorithm above by imposing a sign choice for the pairing of $\mu_{B,\omega}$ with a fixed Euler class representative.\\
\end{proof}

\section{Persistence Of Characteristic Classes}\label{sec: Persistence}

Due to practical considerations such as non-uniform sampling density or the presence of outliers for which our feature map is ill-defined, we may have reason to believe that some of the local measurements incorporated into our model are more reliable than others. 
In other cases, we may want to make a deliberate choice to ignore or modify certain parts of the data so that the restricted discrete approximate cocycle representative becomes trivial. 
Given a family of discrete approximate local angle functions $\{f_{j}\}\in \Gamma_{\alpha}(\pi,\mathcal{U})$ and a witness $\Omega$, we therefore introduce a weight function $w:\mathcal{N}(\mathcal{U})\to \mathbb{R}_{\geq 0}$ which measures the alignment quality of each pair $f_{j}|_{\pi^{-1}(U_{j}\cap U_{k})}$ and $f_{k}|_{\pi^{-1}(U_{j}\cap U_{k})}$ by the corresponding $\Omega_{jk}$. 
We then construct an induced a  simplex-wise filtration $\mathcal{W}^\Omega$ on $\mathcal{N}(\mathcal{U})$ and compute the persistence of $w_{1}(\Omega)$ and $\tilde{e}(\Omega)$ with respect to $\mathcal{W}^\Omega$. 
In Section~\ref{sec: Coordinatization}, we show how to construct a coordinatization map which is in some sense compatible with the restriction of $\Omega$ to any subcomplex $W^{r}$ in $\mathcal{W}^\Omega$.\\

\subsection{Weights Filtration}\label{sec: Weights Filt}

\begin{definition}
    Suppose $\pi:X\to B$ is a function between metric spaces, $\mathcal{U}$ is a good open cover of $B$ and $\Omega\in \check{C}^{1}(\mathcal{U};\underline{O(2)}) \cong C^{1}(\mathcal{N}(\mathcal{U});O(2))$ witnesses that $\{f_{j}\}\in \Gamma_{\alpha}(\pi,\mathcal{U})$ for some $\alpha \geq 0$. 
    For each $(jk)\in\mathcal{N}(\mathcal{U})^{(1)}$, define 

\begin{equation}\label{eq: weight func}
    w(jk) = \frac{1}{|\pi^{-1}(U_{jk})|}\sum_{x\in\pi^{-1}(U_{jk})}|f_{j}(x) - \Omega_{jk}f_{k}(x)|    
\end{equation}

    \noindent Extend $w$ to a weight function $w:\mathcal{N}(\mathcal{U})\to \mathbb{R}_{\geq 0}$ by setting $w(j) = 0$ for all $j$ and inductively defining $w$ for higher simplices by $w(\sigma) = \max\limits_{\sigma'\subsetneq \sigma}w(\sigma')$. Order the simplices of $\mathcal{N}(\mathcal{U})$ by (increasing) weight, then by (increasing) dimension, then lexicographically. This ordering induces a simplex-wise filtration $\mathcal{W}^{\Omega}$ of $\mathcal{N}(\mathcal{U})$, called the \textbf{weights filtration}  of $\mathcal{N}(\mathcal{U})$ associated with $\Omega$, since all 0-simplices appear first, and any higher simplex appears later than each of its faces. \\
\end{definition}

\begin{remark}
    Though we chose~\eqref{eq: weight func} as a measurement of the alignment quality of $f_{j}|_{\pi^{-1}(U_{jk})}$ and $f_{k}|_{\pi^{-1}(U_{jk})}$ via $\Omega_{jk}$, there are plenty of other reasonable choices one could make depending on context. For instance, one could use max error rather than average error. To ensure the weight function induces a total ordering of $\mathcal{N}(\mathcal{U})$, we assume $w$ is injective on 1-simplices (one can always obtain this setup by applying arbitrarily small perturbations). \\ 
\end{remark} 

\begin{definition}
    Suppose $\Omega\in \check{C}^{1}\left(\mathcal{U};\underline{O(2)}\right)$ witnesses that $\{f_{j}\}\in \Gamma_{\alpha}(\pi,\mathcal{U})$. Let $\mathcal{W}^{\Omega} = \{W_{r}\}_{r=1}^{|\mathcal{N}(\mathcal{U})|}$ denote the associated weights filtration. Given a cochain $\Lambda\in \check{C}^{p}(\mathcal{U};\mathcal{F})$ for a presheaf of groups $\mathcal{F}$ and $p \geq 1$, we define the \textbf{cobirth} of $\Lambda$ to be the pair $(r,w(\sigma_{r}))$, where $r$ is the largest index such that $\Lambda|_{W_{r}}$ is a cocycle (assume $p\leq 1$ if $\mathcal{F}$ is a sheaf of non-abelian groups). 
    Similarly, we define the $\textbf{codeath}$ of $\Lambda$ to be the pair $(r, w(\sigma_{r}))$ for $r$ the largest index such that $\Lambda|_{W_{r}}$ is a coboundary. \\

\end{definition}

\begin{figure}[h!]
    \centering
    \begin{subfigure}{.31\textwidth}
        \centering
        \includegraphics[width=\textwidth]{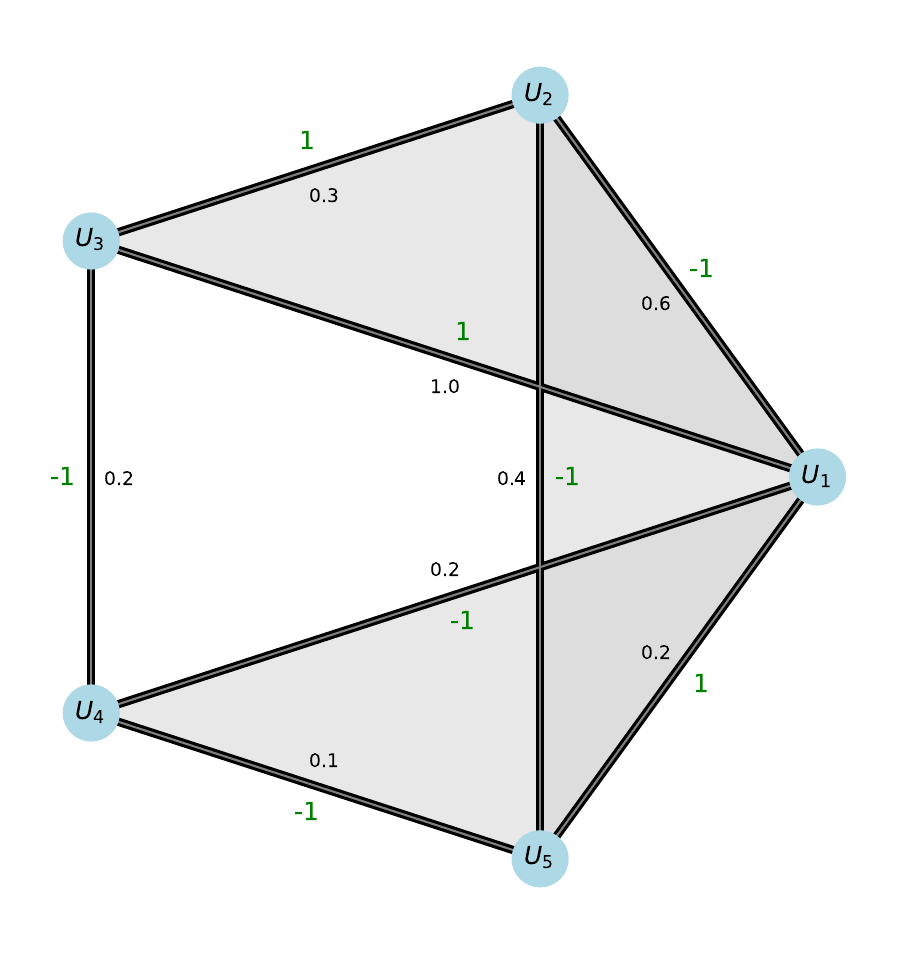}
        \caption{Full nerve}
    \end{subfigure}
    \hspace{0.01\textwidth}
    \begin{subfigure}{0.31\textwidth}
        \centering
        \includegraphics[width=\linewidth]{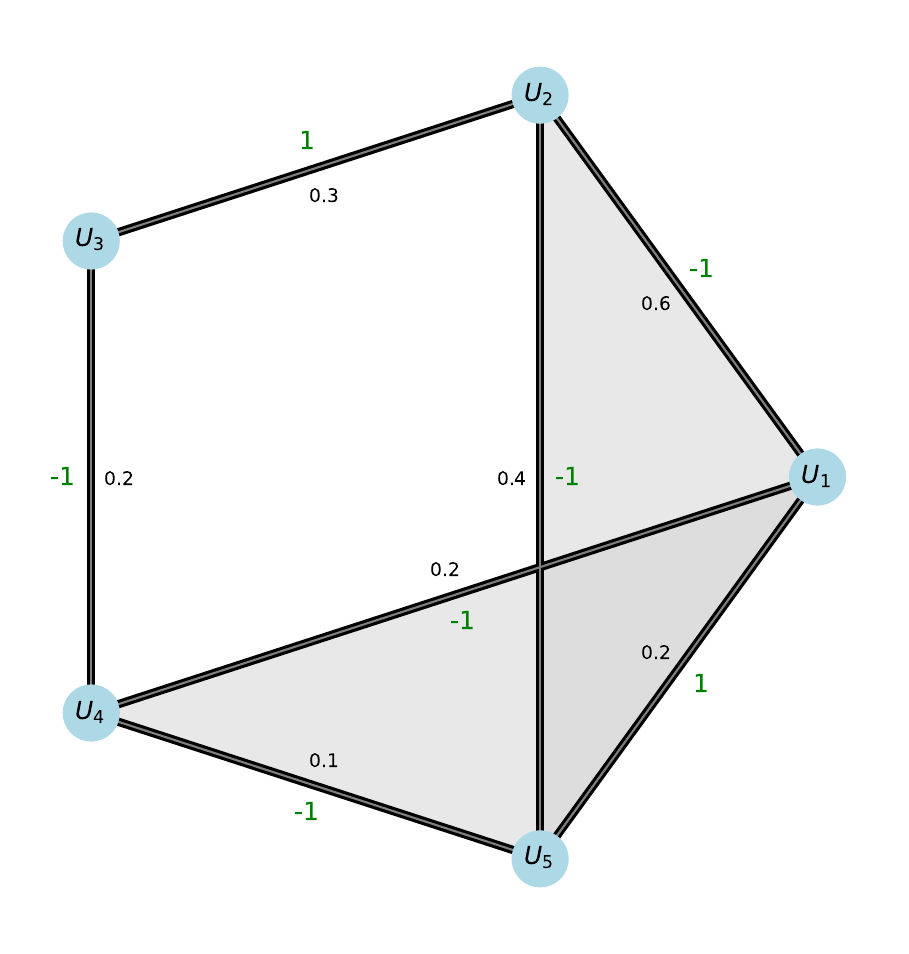}
        \caption{$\omega$ cobirth}
    \end{subfigure}
    \hspace{0.01\textwidth}
    \begin{subfigure}{0.31\textwidth}
        \centering
        \includegraphics[width=\linewidth]{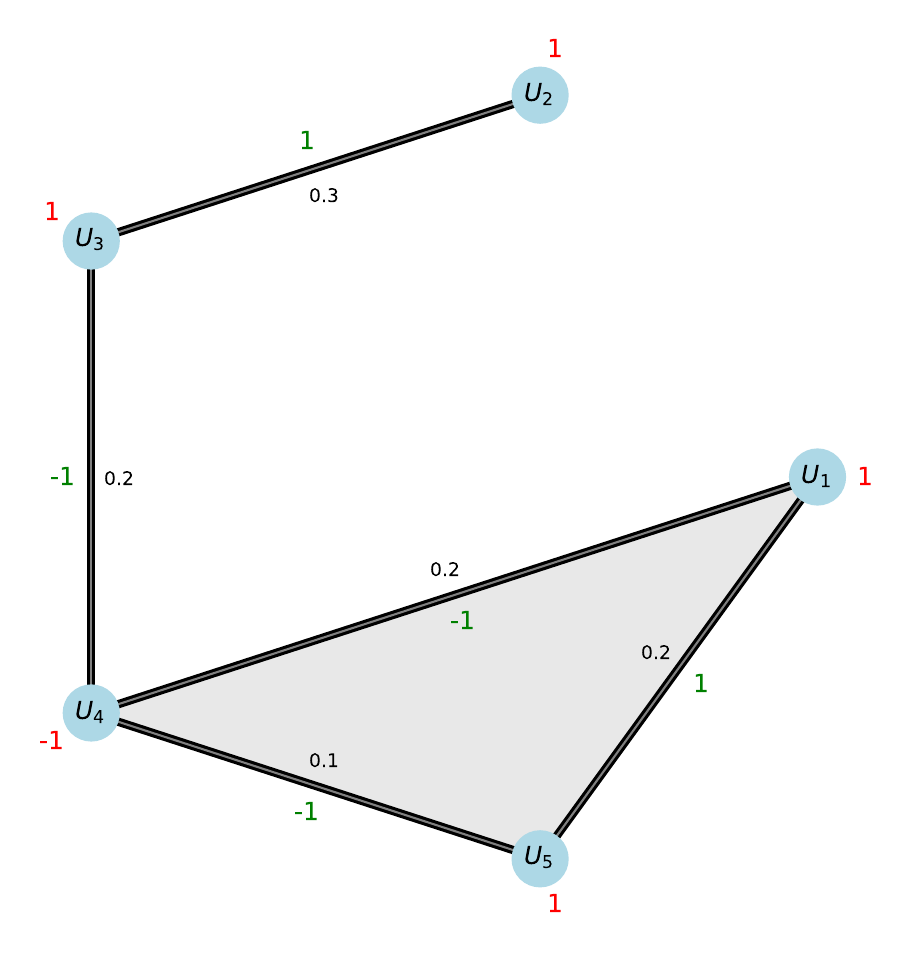}
        \caption{$\omega$ codeath}
    \end{subfigure}

    \caption{Various stages of the weights filtration on the nerve of an open cover. Each edge is labeled with its associated weight and the value assigned by a $\mathbb{Z}_{2}$-simplicial cochain $\omega$.}
    \label{fig: Persistence illustration}
\end{figure}

\begin{construction}\label{const: persistence computation}
    Suppose $R$ is a PID and $R_{\omega}$ is a local system with stalk $R$. 
    If $\mathcal{U}$ is finite, one can efficiently compute the cobirth and codeath of $\Lambda$ as follows: let $D$ denote the matrix representation of the coboundary operator $\delta_{\omega}:\check{C}^{*}(\mathcal{U};R_{\omega})\to \check{C}^{*}(\mathcal{U};R_{\omega})$ when the simplices are arranged in increasing order according to the weights filtration. 
    Define the vector $\hat{\Lambda}\in\mathbb{R}^{|\mathcal{N}(\mathcal{U})|}$ such that $\hat{\Lambda}_{r} = \Lambda(U_{\sigma_{r}})$ whenever $\sigma_{r}$ is a $p$-simplex and $\hat{\Lambda}_{r} = 0$ otherwise. It is easy to see that the cobirth of $\Lambda$ is 1 less than the smallest index $r_{b}$ such that $(D\hat{\Lambda})_{r_{b}}$ is nonzero (if none exists, the cobirth is $|\mathcal{N}(\mathcal{U})|$).\\

    \noindent To determine the codeath, write $S = LDR$, where $L$ and $R$ are invertible and $S$ is the Smith normal form of $D$. For $p > 0$, let $r_{p}$ denote the largest index of a row of $S$ corresponding to a $p$-simplex and containing a non-zero entry. The codeath of $\Lambda$ is 1 less than the largest index $r_{d} > r_{p}$ such that $(L\hat{\Lambda})_{r_{d}}$ is nonzero (if none exists, then the codeath equals the cobirth). \\

\end{construction}

\noindent Given a cochain $\Omega\in\check{C}^{1}(\mathcal{U};\underline{O(2)})$, we can use Construction~\ref{const: persistence computation} to compute the cobirth index $b_{sw}$ and codeath index $d_{sw}$ of $\omega = \det_{*}(\Omega)\in\check{C}^{1}(\mathcal{U};\underline{\mathbb{Z}_{2}})$ (here the local system is trivial). Then, for any $r\geq b_{sw}$, we can meaningfully ask whether $\tilde{e}(\Omega)|_{W^{r}} = \Delta_{\omega}(\Omega)|_{W^{r}}\in C^{1}(W^{r};\mathbb{Z}_{\omega})$ is a cocycle, so we compute the persistence of $\tilde{e}(\Omega)$ restricted to $W^{b_{sw}}$. \\

\subsection{Interpretation}

We would like to interpret each filtration stage $W^{r}\in \mathcal{W}$ as the nerve of a modified open cover $\mathcal{U}^{r}$ of a new base space $B^{r}$. For example, removing an edge $(jk)$ from $\mathcal{N}(\mathcal{U})$ effectively separates sets $U_{j}$ and $U_{k}$. For the purposes of coordinatization, there are various choices one could make about how to handle data in $\pi^{-1}(U_{j}\cap U_{k})$, and the appropriate choice depends on context. One option would be to simply remove these points from the dataset -- this would be a reasonable choice if the intersection is small and consists mostly of outliers with ill-defined projections. In other cases, one may want to keep the data while modifying $B$ and $\mathcal{U}$. For instance, if the cobirth of a class representative is $\infty$ (or all the weights are fairly small and close together), one may still want to introduce cuts in the base space in order to artificially create a trivial bundle which admits a global coordinate system. With this in mind, we offer a recursive approach for constructing a family of spaces $\{B^{r}\}$ and corresponding open covers $\{\mathcal{U}^{r}\}$ such that $W^{r} = \mathcal{N}(\mathcal{U}^{r})$ for each $r$. \\

\begin{construction}\label{const: filtered space}
    Let $\mathcal{W}^{\Omega}$ denote the weights filtration associated with $\Omega\in \check{C}^{k}(\mathcal{U};\underline{O(2)})$.
    Suppose $W^{r}\in\mathcal{W}^{\Omega}$ for $r > 1$, and we have a space $B^{r}$ and open cover $\mathcal{U}^{r} = \{U_{j}^{r}\}_{j\in J}$ such that $B^{r} = \bigcup_{j\in J}U^{r}_{j}$ and $W^{r} = \mathcal{N}(\mathcal{U}^{r})$. Let $\sigma_{r}\in\mathcal{N}(\mathcal{U})$ denote the simplex which appears in $W^{r}$ but not $W^{r-1}$.  Choose some $j\in \sigma_{r}$ (for instance, the last $j$ lexicographically) and define $U^{r-1}_{j} = U_{j}^{r} \cap (\overline{U^{r}_{\sigma_{r}}})^{c}$ and $B^{r-1} = B^{r} - (U^{r}_{j}\cap \partial U^{r}_{\sigma_{r}})$ (for all $k\neq j$, define $U^{r-1}_{k} = U^{r}_{k}$). It is easy to check that $U^{r-1}_{j}\subset B^{r-1}$ is open, $B^{r-1} = \bigcup_{j}U^{r-1}_{j}$ and $W^{r-1} = \mathcal{N}(\mathcal{U}^{r-1})$. In terms of the original space $B$, we have 

\begin{equation*}
    B^{r} = B - \bigcup_{s > r}(U_{j_{s}}^{s}\cap \partial U_{\sigma_{s}^{s}})
\end{equation*}

\end{construction}

\section{Coordinatization}\label{sec: Coordinatization}

Suppose $\mathcal{U} = \{U_{j}\}_{j=1}^{n}$ is a good finite open cover of a metric space $B$ and $\Omega\in\check{Z}_{\varepsilon}^{1}\left(\mathcal{U};\underline{O(2)}\right)$ witnesses that $\{f_{j}\}_{j=1}^{n}\in \Gamma_{\alpha}(\pi, \mathcal{U})$ for some function $\pi:X\to B$ ($X$ a finite metric space). 
Let $\mathcal{W}^\Omega = \{W^{r}\}$ denote the associated weights filtration of $\mathcal{N}(\mathcal{U})$. 
For each $r$, we show how to construct a coordinatization map $F^{r}:X\to V(2,d)\times_{O(2)}\mathbb{S}^{1}$ for $d\leq 2n$ which is in some sense compatible with $\Omega|_{W^{r}}$. 
This gives a menu of options of models for the dataset. 
In the special case when $\Omega|_{W^{r}}$ is an approximate coboundary, we construct a 'global trivialization' map $F^{r}:X\to B^{r}\times\mathbb{S}^{1}$. \\

\begin{enumerate}
    \item \textbf{Step 1: (Choose A Filtration Stage)} Choose a filtration stage $W^{r}\in \mathcal{W}^\Omega$ of $\mathcal{N}(\mathcal{U})$. Use Construction~\ref{const: filtered space} to construct the sets $\mathcal{U}^{r} = \{U_{j}^{r}\}_{j=1}^{n}$, then choose a partition of unity $\{\rho^{r}_{j}\}_{j=1}^{n}$ subordinate to $\mathcal{U}^{r}$.\\

\item \textbf{Step 2: (Dimensionality Reduction)} For each $j$, define $\widehat{\Psi}^{r}_{j}:U^{r}_{j}\to V(2,2n)$ by 

\begin{equation}\label{eq: dim red}
    \widehat{\Psi}^{r}_{j}(b) = \left(\begin{array}{ccc}
    \sqrt{\rho^{r}_{0}(b)} \Omega_{0j}\\
    \vdots\\
    \sqrt{\rho^{r}_{n}(b)} \Omega_{nj}\\
    \end{array}\right)
\end{equation}

Choose a value $d\leq 2n$ and apply Principal Stiefel Coordinates \cite{LeeEtAl2025} to the point cloud $X_{V} = \bigcup_{j}\Phi_{j}^{r}(U_{j}^{r})\subset V(2,2n)$ to obtain a map $\text{psc}:X_{V}\to V(2,d)$.\\ 

    \item \textbf{Step 3: (Classifying Map)} For each $j$, let $\widetilde{\Psi}_{j}^{r} = \text{psc}\circ\widehat{\Psi}_{j}^{r}:U_{j}^{r}\to V(2,d)$. Define $\widetilde{g}^{r}:B^{r}\to \mathbb{R}^{d\times d}$ by

\begin{equation}
    \widetilde{g}^{r}(b) = \sum_{j=1}^{n}\rho^{r}_{j}(b)\widetilde{\Psi}^{r}_{j}(b)\widetilde{\Psi}^{r}_{j}(b)^{T}
\end{equation}

\noindent and let $g^{r} = \Pi\circ\tilde{g}^{r}:B^{r}\to Gr(2,d)$.\\

\item \textbf{Step 4: (Cocycle Projection)} For each $j$, define $\Psi^{r}_{j}:U^{r}_{j}\to V(2,d)$ by $\Psi^{r}_{j}(b) = \widehat{\Pi}(g^{r}(b), \widetilde{\Psi}^{r}_{j}(b))$, where $\widehat{\Pi}$ is the map from Construction~\ref{const: Stiefel projection}. Then, define $\Omega^{r}\in\check{Z}^{1}(\mathcal{U}^{r};\mathcal{C}_{O(2)})$ by $\Omega^{r}_{jk}(b) = (\Psi^{r}_{j}(b))^{T} \Psi^{r}_{k}(b)$. \\

\item \textbf{Step 5: (Bundle Map)} Define $F^{r}_{j}:\pi^{-1}(U^{r}_{j})\to V(2,d)\times_{O(2)}\mathbb{S}^{1}$ by 
\begin{equation}
    F_{j}^{r}(x) = \Psi^{r}_{j}(\pi(x))\exp\left(\sum_{k=0}^{n}\rho_{k}^{r}(\pi(x))\log \Omega_{jk}^{r}(\pi(x))f_{k}(x)\right)
\end{equation}

\noindent The $F^{r}_{j}$'s agree on overlaps, yielding a globally defined map $F^{r}:X\to V(2,d)\times_{O(2)}\mathbb{S}^{1}$. \\

\end{enumerate}

\subsection{Coordinatization Pipeline For Trivial Bundles}

In the special case when a family $\{f_{j}\}_{j\in J}\in T_{\alpha}(\pi,\mathcal{U}^{r})$ is witnessed by a  discrete approximate coboundary $\Omega$, we construct a global trivialization map $\varphi:X\to B\times \mathbb{S}^{1}$. In particular, suppose $\omega = \text{det}_{*}(\Omega)\in\check{Z}^{1}(\mathcal{U}^{r};\underline{\mathbb{Z}_{2}})$ and $\tilde{e} = \Delta_{\omega}(\Omega)\in\check{Z}^{2}(\mathcal{U}^{r};\mathbb{Z}_{\omega})$ are coboundaries.\\

\begin{enumerate}
    \item Construct a solution to the linear system $\delta^{0}\phi = \omega$ over $\mathbb{Z}_{2}$. 
    Define $\hat{\Omega} = \phi\cdot \Omega\in\check{Z}^{1}\left(\mathcal{U}^{r};\underline{SO(2)}\right)$, and let $\hat{f}_{j} = \phi_{j}\circ f_{j}$ for all $j$.\\

    \item Define $e_{jkl} = \phi_{j}\tilde{e}_{jkl}$ for all $(jkl)\in\mathcal{N}(\mathcal{U}^{r})^{(2)}$  and construct a solution to the linear system $\delta^{1}\beta = e$ over $\mathbb{Z}$. 
    
    \item Choose a partition of unity $\{\rho_{j}\}_{j\in J}$ subordinate to $\mathcal{U}$ and any lift $\Theta\in\check{C}^{1}(\mathcal{U}^{r};\underline{\mathbb{R}})$ of $\hat{\Omega}$ via $\text{exp}_{*}$. Define $\mu\in \check{C}^{1}(\mathcal{U}^{r};\mathcal{C}_{\mathbb{R}})$ by 

\begin{equation}
    \mu_{j}(b) = \sum_{k=1}^{n}\rho_{k}(b)(\Theta_{kj} - \beta_{kj})
\end{equation}

\noindent and let $\bar{f}_{j} = \exp (\mu_{j})\circ \hat{f}_{j}$ for all $j$.\\

\item Define $F:X\to B\times\mathbb{S}^{1}$ by 

\begin{equation}
    F(x) = \left(\pi(x),\ \exp\left(\sum_{j=1}^{n}\rho_{j}(\pi(x))\log \tilde{f}_{j}(x)\right)\right)
\end{equation}

\end{enumerate}

\begin{remark}
    Since trivial bundles are flat, we can often use Singer's method~\cite{Singer} in practice to obtain a locally-constant true coboundary $\widetilde{\Omega}$ which is close to an approximate one. Though we have no theoretical guarantees on stability, the fact that the result is locally-constant guarantees that the resulting trivialization map does not have unnecessary variation. \\
\end{remark}

\section{Experiments}\label{sec: Implementations}

Here we demonstrate the practical viability and utility of the algorithms we have described in this paper through a series of implementations on real and synthetic datasets. 
A summary of the results of each experiment can be found in Table~\ref{tab: Experiments Results}. 
Note that the roughness value $\alpha$ in each case was well outside the range required to meet our theoretical hypotheses, which suggests that our approach is stable for a larger class of objects than those covered by the present analysis. 
Further details about our implementation on optical flow data can be found in the companion paper~\cite{turow2026extended}. 
Additional examples, implementation details and documentation are provided in the accompanying open-source software package. 
\\

\begin{table}[h!]
\centering
\begin{tabular}{|c|c|c|c|c|c|}
\hline
\textbf{Experiment} & \textbf{Base Space} & \textbf{Ambient Dim} &  \textbf{Sample Size} & \textbf{SW1} & $\widetilde{\eta}$\\
\hline
Optical Flow Patches  & $\mathbb{RP}^{1}$ & 18 & 25,000 & 0 & 0 \\

Folded Klein Bottle  & $\mathbb{S}^{1}$ & 8 & 10,000 & $\neq 0$ & 0 \\

Prism Densities & $\mathbb{RP}^{2}$ & 15,408 & 10,000 & $\neq 0$ & $\pm 3$\\

\hline
\end{tabular}
\caption{A summary of the setups for the experiments. 
Each entry in the SW1 column indicates whether the Stiefel-Whitney class of the corresponding circle bundle model vanishes. 
Entries in the $\widetilde{\eta}$ column indicate the twisted Euler numbers of the models.}
\label{tab: Experiments Setup}
\end{table}

\begin{table}[h!]
\centering
\begin{tabular}{|c|c|c|c|c|c|}
\hline
\textbf{Experiment} & $\eps$ & $\delta$ & $\alpha$ & $\eps_{\text{cocycle}}$ & $d_{Z}(\Omega,\Pi_{Z}(\Omega))$ \\
\hline
Optical Flow Patches  & 0.207 & 0 & 0.207 & 0 & 0 \\

Folded Klein Bottle & 0.207 & 0 & 0.207 & 0 & 0 \\ 

Prism Densities & 0.380 & 0.545 & 0.837 & 0.312 & 0.295 \\

\hline
\end{tabular}
\caption{Summary of experimental results.
The values of $\eps$, $\delta$ and $\alpha$ for each system of discrete approximate angle functions are shown. 
For each experiment, the constructed witness $\Omega$ was found to lie in $\check{Z}^{1}_{\eps_{\text{cocycle}}}(\mathcal{U};\underline{O(2)})$.} 
\label{tab: Experiments Results}
\end{table}

\begin{table}[h!]
\centering
\begin{tabular}{|c|c|c|c|c|c|c|c|c|}
\hline
\textbf{Experiment} & $\#(\mathcal{U})$ & $w_{\text{max}}$ & $b_{\omega}$ & $d_{\omega}$ & $b_{\widetilde{e}}$ & $d_{\widetilde{e}}$ & $\#(W^{r})^{(1)}$ & $\#(W^{r})^{(2)}$\\

\hline
Optical Flow Patches & 12 & 0.088 & $\infty$ & $\infty$ & $\infty$ & $\infty$ & 12/12 & 0/0\\

Folded Klein Bottle & 12 & 0.088 & $\infty$ & $\infty$ & $\infty$ & $\infty$ & 12/12 & 0/0\\

Prism Densities & 60 & 0.176 & $\infty$ & 0.176 & $\infty$ & 0.088 & 67/177 & 11/118\\
\hline
\end{tabular}
\caption{Summary of persistence information for the computed characteristic class representatives in each experiment.
Here $w_{\text{max}}$ denotes the maximum weight among all simplices in $\mathcal{N}(\mathcal{U})$.
The cobirth and codeath weights of the Stiefel-Whitney and twisted Euler classes are respectively denoted by $b_{\omega}$, $d_{\omega}$, $b_{\widetilde{e}}$ and $d_{\widetilde{e}}$.
$W^{r}$ denotes the stage of the weights filtration corresponding to $\min\{d_{\omega},d_{\widetilde{e}}\}$; the number of 1 and 2-simplices in $W^{r}$ are denoted by $\#(W^{r})^{(1)}$ and $\#(W^{r})^{(2)}$, respectively.}
\label{tab: Persistence Results}
\end{table}

\subsection{Optical Flow Patches}\label{sec: optical flow}

Optical flow refers to the apparent motion of objects between consecutive frames of a video \cite{Horn_Opt_Flow}. 
This perceived motion is characterized by an assignment of a vector $v\in\mathbb{R}^{2}$ to each pixel in each frame. 
The vector at pixel $(x,y)$ in frame $t$, points to the location in frame $t+1$ where pixel $(x,y)$ appears to move. 
The estimation of optical flow from a video is central to many computer vision tasks, including motion tracking, object segmentation, and video compression \cite{Autonomous_Vehicles}, \cite{Opt_Flow_Applications}. The Sintel dataset \cite{Sintel}, derived from the open-source animated short film “Sintel” by the Blender Foundation, was created as a high-quality benchmark for testing and training optical flow models. \\

\begin{figure}[h]
    \centering
    \includegraphics[width=0.55\textwidth]{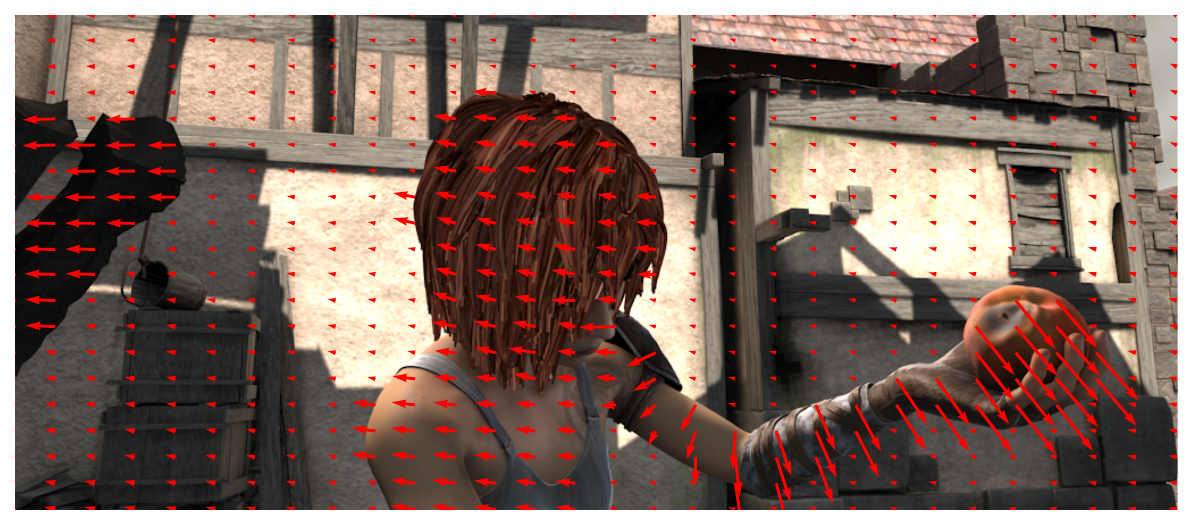}
    \caption{A frame from the Sintel video with sample (scaled) optical flow vectors shown}
    \label{fig: labeled frame}
\end{figure}

In~\cite{opt_flow_torus}, Adams et al. studied the nonlinear statistics of 3 x 3 high-contrast optical flow patches from \cite{Sintel}. 
They proposed that a dense core subset $X$ of (normalized) high-contrast patches is concentrated around a 2-manifold in $\mathbb{R}^{18}$ homeomorphic to the 2-dimensional torus $\mathbb{T}^{2}$. 
To support their claim, they introduced a feature map $\pi_{\text{flow}}:X\to \mathbb{RP}^{1}$ which roughly measures a patch's predominant axis of flow (see Figure~\ref{fig: sample_predom_dirs}). 
They then offered evidence that the data in $X$ over any small interval in $\mathbb{RP}^{1}$ has circular topology,  supporting the hypothesis that an underlying circle bundle structure is present. \\

\begin{figure}[h]
    \centering
    \includegraphics[width=0.95\textwidth]{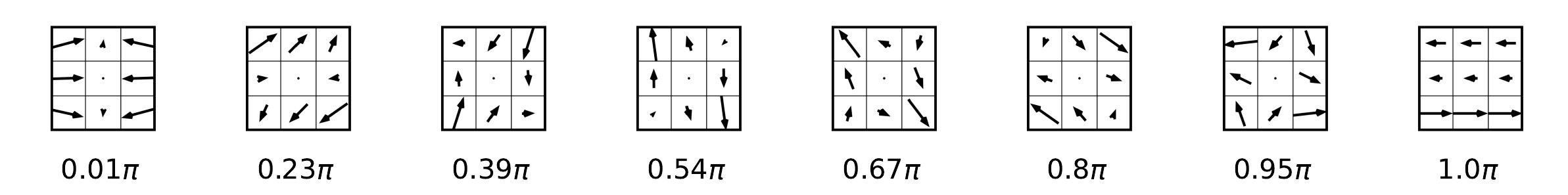}
    \caption{A sample of mean-centered and normalized high-contrast optical flow patches labeled by predominant direction (represented as an angle in $[0,\pi)$)}
    \label{fig: sample_predom_dirs}
\end{figure}

By Proposition~\ref{prop: 2-manifold class}, there are exactly two non-isomorphic circle bundles over $\mathbb{RP}^{1}\cong\mathbb{S}^{1}$ which can be distinguished computationally by their Stiefel-Whitney classes in $H^{1}(\mathbb{RP}^{1};\mathbb{Z}_{2})\cong\mathbb{Z}_{2}$. The total spaces of the two bundles are the torus and the Klein bottle. 
To confirm and provide coordinates for the torus model proposed by~\cite{opt_flow_torus}, we used the same pre-processing to obtain a sample of 50,000 normalized high-contrast optical flow patches from the Sintel dataset and refine to a sample $X_{\text{flow}}$ of 25,000 patches from the dense core subset $X(300,50)$. 
We then computed the predominant axis of flow $\pi_{\text{flow}} (x)\in\mathbb{RP}^{1}$ for each patch $x\in X_{\text{flow}}$.
Figure~\ref{fig: Global persistence} shows the persistence diagram for a random subsample of 1,000 points from $X_{\text{flow}}$. For comparison, we also computed persistence diagrams for a random 1,000-point sampling from $X(300,30)$ and a noisy synthetic sampling of the theoretical proposed torus model.
Observe that the synthetic dataset exhibits the Betti signature of $\mathbb{T}^{2}$ ($\beta_{0} = 1$, $\beta_{1} = 2$, $\beta_{2} = 1$) over a large range of scales, while the most persistent Betti signature for the samples of $X(300,50)$ and $X(300,30)$ is $\beta_{0} = 1$, $\beta_{1} = 1$, $\beta_{2} = 0$. \\

\begin{figure}[ht]
    \centering
    \begin{subfigure}[t]{0.32\textwidth}
        \centering
        \includegraphics[width=\linewidth]{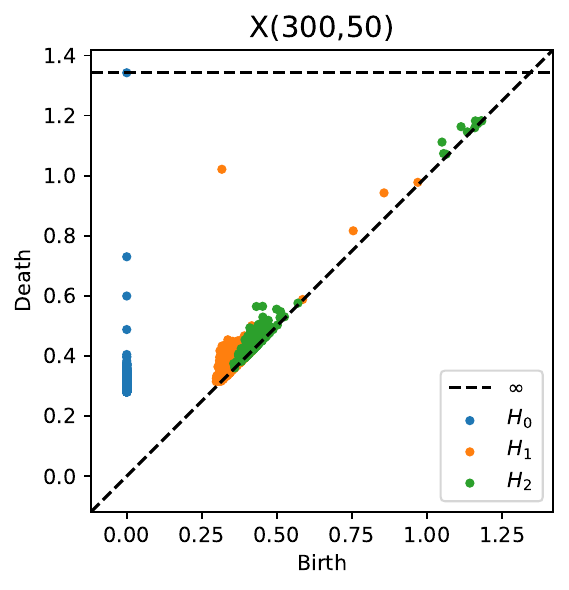}
    \end{subfigure}
    \begin{subfigure}[t]{0.32\textwidth}
        \centering
        \includegraphics[width=\linewidth]{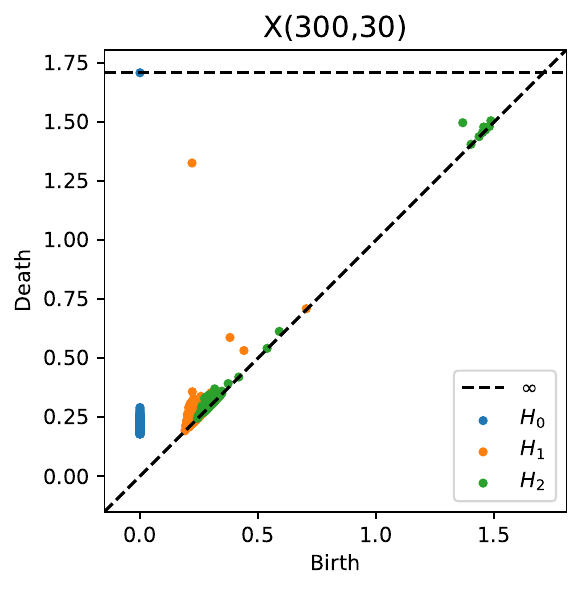}
    \end{subfigure}
    \begin{subfigure}[t]{0.32\textwidth}
        \centering
        \includegraphics[width=\linewidth]{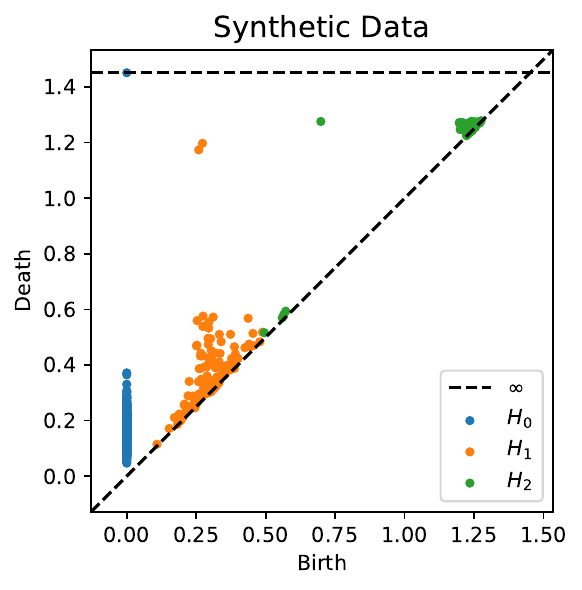}
    \end{subfigure}
    \caption{Persistence diagrams for the top 50\% and 30\%  densest subsets of the original dataset (left, middle), and those of a synthetically generated dataset  from the hypothesized torus model (right).  } 
    \label{fig: Global persistence}
\end{figure}

On the other hand, as observed in~\cite{opt_flow_torus}, the \textit{local} picture of $X_{\text{flow}}$ with respect to $\pi_{\text{flow}}$ tells a different story. 
We covered $\mathbb{RP}^{1}$ with 12 open intervals of the form $U_{j} = (\frac{j\pi}{12}-\frac{\pi}{16}, \frac{j\pi}{12} + \frac{\pi}{16})$ (note that $\frac{\pi}{16}$ is the maximal radius for which the nerve of $\mathcal{U} = \{U_{j}\}_{j=1}^{12}$ has no triple intersections). Figure~\ref{fig: Opt Flow Local Picture} shows the first two principal components and the persistence diagrams for the data in each $\pi_{\text{flow}}^{-1}(U_{j})$, strongly suggesting the presence of an underlying circle bundle structure. \\

\begin{figure}[h!]
    \centering
    \includegraphics[width=0.9\textwidth]{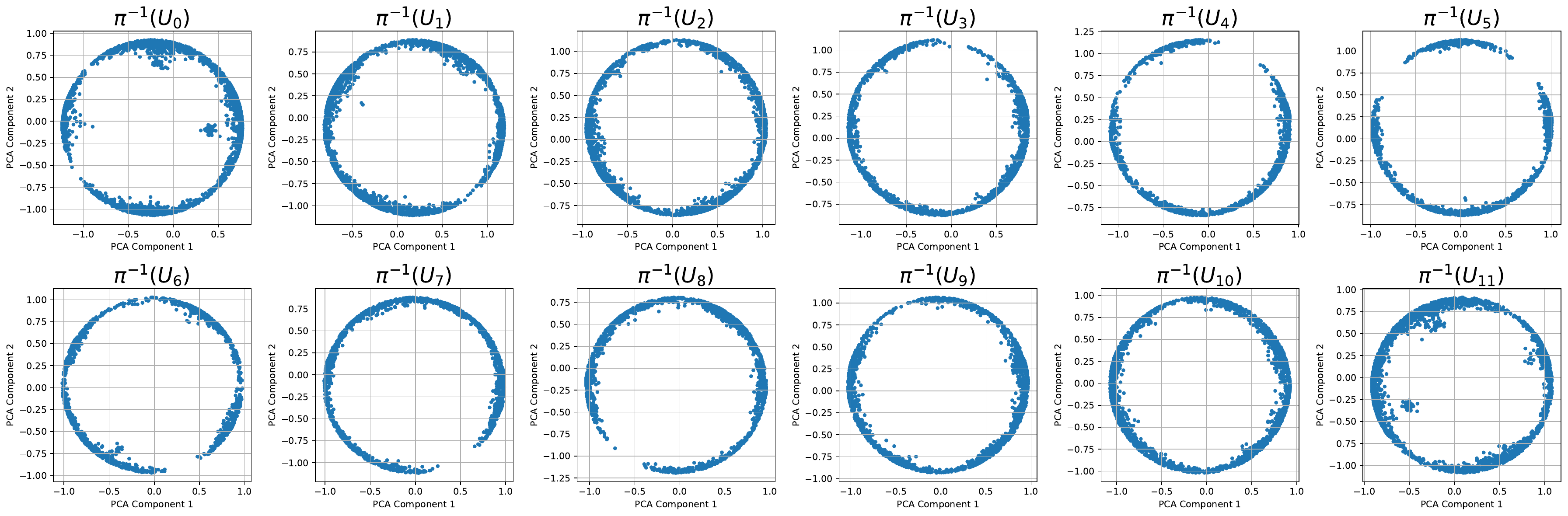}
    \vspace{5mm}
    \includegraphics[width=0.9\textwidth]{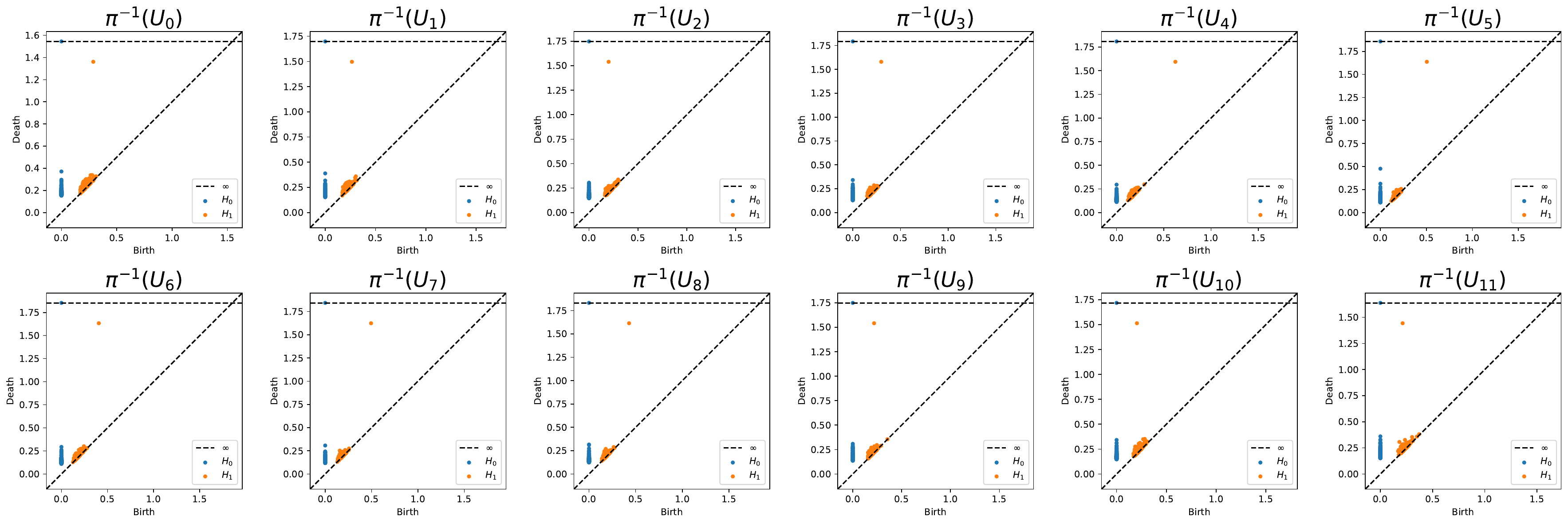}
    \caption{PCA and persistence diagrams for the sets $\{\pi^{-1}(U_{j})\}_{j=1}^{12}$} 
    \label{fig: Opt Flow Local Picture}
\end{figure}

We computed a family $\{f_{j}\}_{j=1}^{12}$ of discrete approximate local trivializations for $X_{\text{flow}}$ subordinate to $\mathcal{U}$ using PCA, then obtained a witness $\Omega\in\check{C}^{1}(\mathcal{N}(\mathcal{U});O(2))$ using the procedure outlined in Construction~\ref{const: maxmin Procrustes}. 
Figure~\ref{fig: Optical Flow Nerve} shows a visualization of $\mathcal{N}(\mathcal{U})$; the weight $w(jk)$ of each edge $(jk)$ is shown in black, and the values assigned by the Stiefel-Whitney class representative $\omega = \det_{*}(\Omega)$ are shown in red. A $0$-cochain $\phi\in C^{0}(\mathcal{N}(\mathcal{U});\mathbb{Z}_{2})$ satisfying $\delta^{0}\phi = \omega$ is shown in blue, confirming a topological torus structure underlying the data. 

\begin{figure}[h!]
    \centering
    \begin{subfigure}{.45\textwidth}
        \centering
        \raisebox{8.5mm}{
        \includegraphics[width=\linewidth]{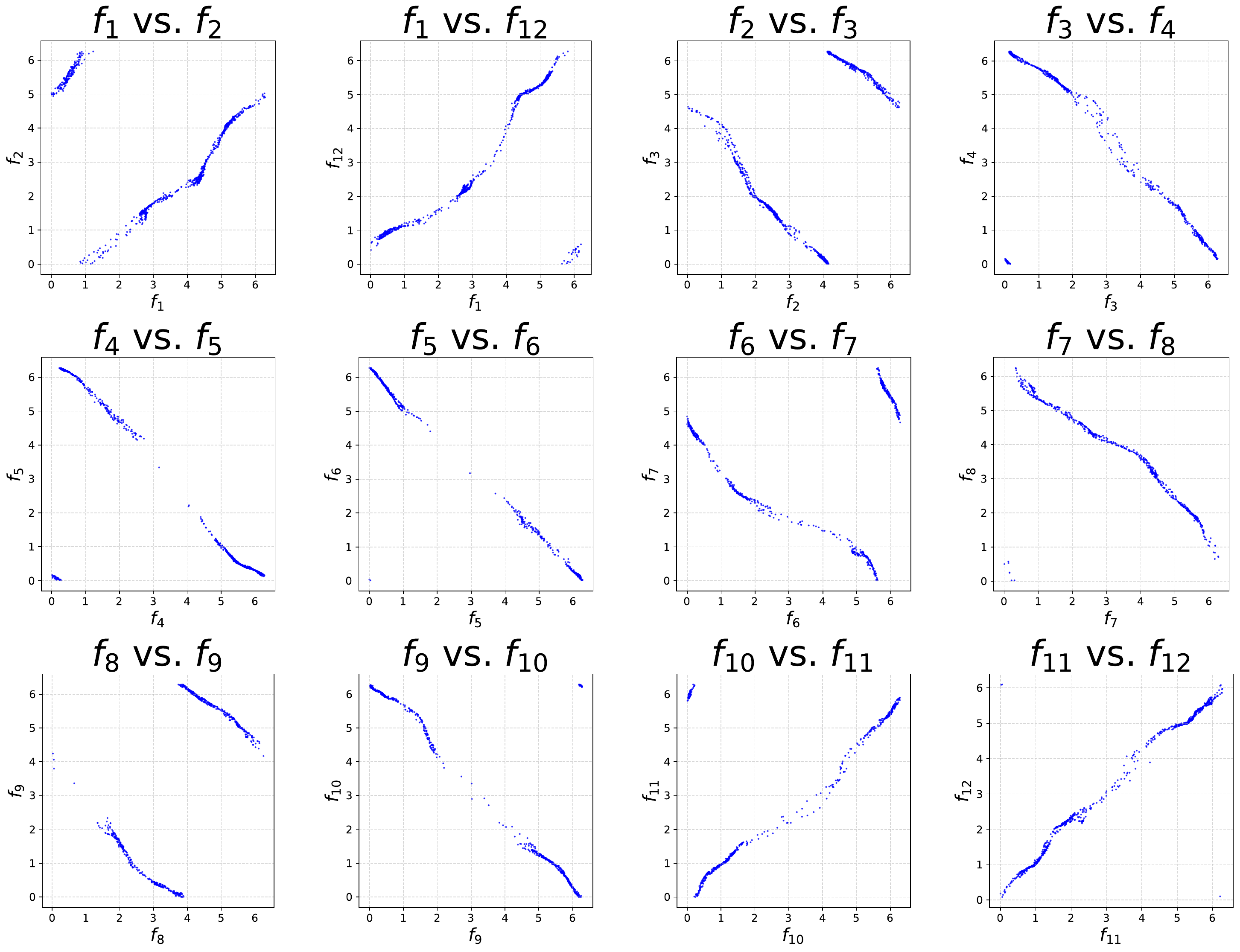}
        }
    \end{subfigure}
    \hspace{0.05\textwidth}
    \begin{subfigure}{.45\textwidth}
        \centering
        \includegraphics[width=\textwidth]{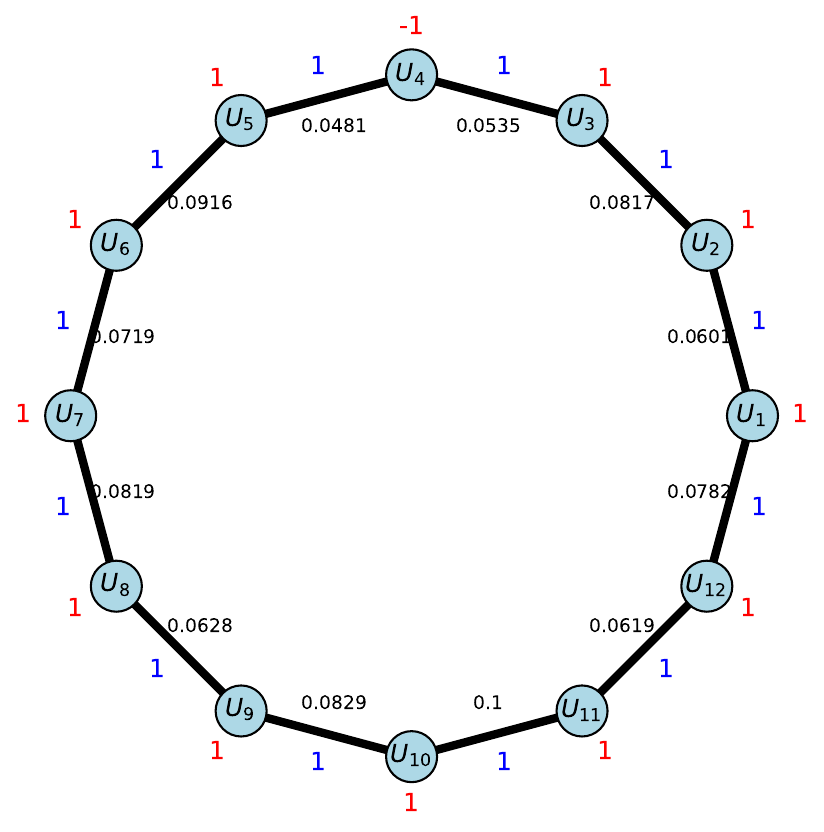}
    \end{subfigure}
    \caption{(Left) Correlations on intersections between the discrete approximate local trivializations $\{f_{j}\}_{j=1}^{12}$ computed for $X_{\text{flow}}$. (Right) The nerve of the open cover $\mathcal{U} = \{U_{j}\}_{j=1}^{12}$ of $\mathbb{RP}^{1}$. Each edge is labeled by its associated weight. The components of the Stiefel-Whitney class representative $\omega$ are shown in blue, and a potential $\phi$ for $\omega$ is shown in red.}
    \label{fig: Optical Flow Nerve}
\end{figure}

Once we verified the global triviality of the underlying bundle structure, we used the coordinatization pipeline described in Section~\ref{sec: Coordinatization} to obtain a global coordinate system for $X_{\text{flow}}$. 
It is important to note that a significant portion of the patches had ambiguous predominant flow directions despite being high-contrast (see Figure~\ref{fig: Weak Predom Dirs Samples}). 
Such patches are not accounted for in the torus model proposed by~\cite{opt_flow_torus}. 
Upon further investigation, we found that the fibers of the predominant direction map $\pi_{\text{flow}}$ are geometrically more like cylinders than circles, and that this extra degree of freedom locally parametrizes variations in directional strength among patches in $X_{\text{flow}}$.   
The generalized model we describe in our upcoming work  also accounts for other families of high-contrast patches not featured in the original torus model which are particularly relevant for motion tracking analysis.\\

\begin{figure}[ht]
    \centering
    \includegraphics[width=0.47\textwidth]{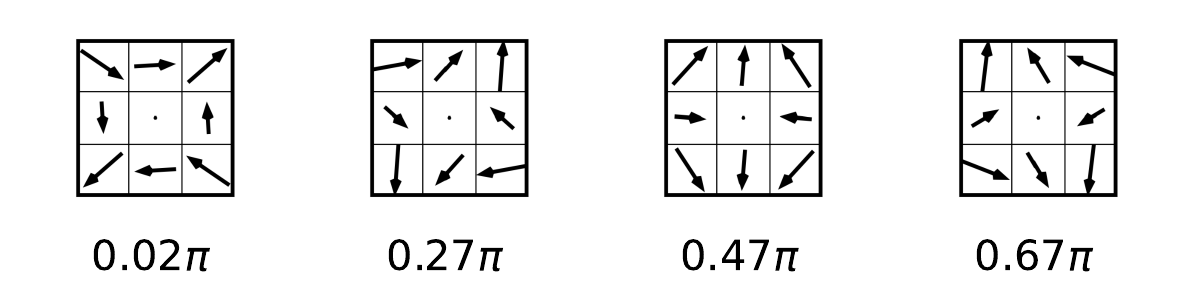}
    \hspace{7mm}
    \includegraphics[width=0.47\textwidth]{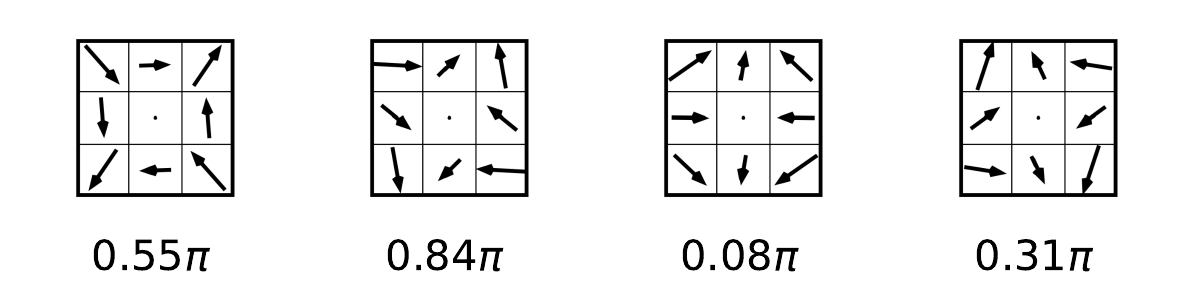}
    \caption{A sample of high-contrast optical flow patches for which the predominant direction is essentially arbitrary} 
    \label{fig: Weak Predom Dirs Samples}
\end{figure}

Figure~\ref{fig: Optical Flow Torus} shows a selection of coordinatized patches from $X_{\text{flow}}$ which are highly directional (see the upcoming paper for our formal definition of directional strength). Notice that patches near the top and bottom of the diagram roughly coincide, reflecting the circular topology of the fibers. Patches near the left and right edges also roughly coincide, which reflects the global torus structure.\\
\vspace{0.2cm}
\begin{figure}[h!]
\centering
\begin{tikzpicture}
  \node[anchor=south west, inner sep=0] (img) at (0,0)
    {\includegraphics[width=0.62\textwidth]{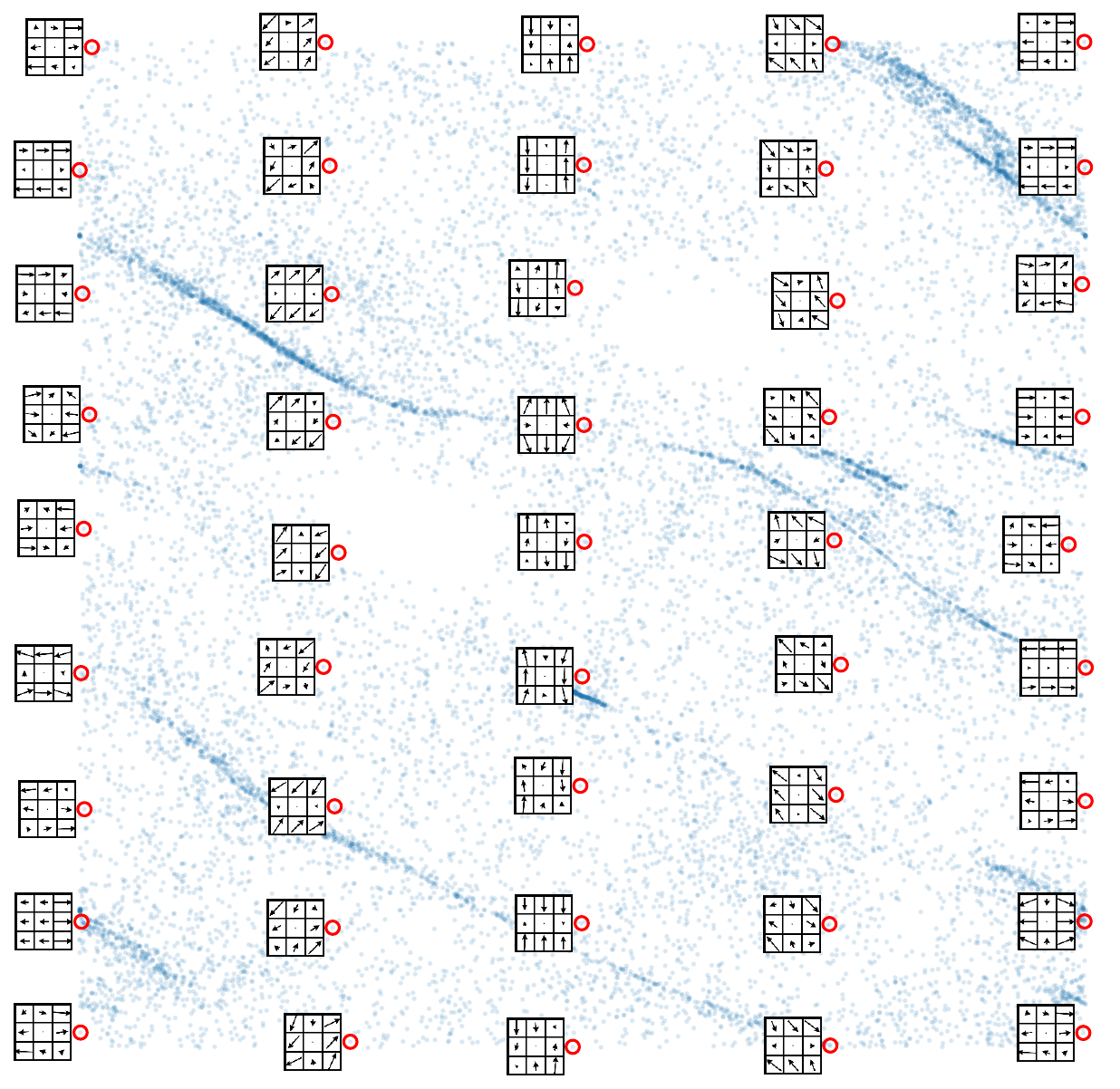}};
  \begin{scope}[x={(img.south east)}, y={(img.north west)}]

    \tikzset{torusarrow/.style={->, line width=0.6pt, draw=black!60}}

    \draw[torusarrow] (0.08,1.02) -- (0.92,1.02);
    \draw[torusarrow] (0.08,-0.02) -- (0.92,-0.02);

    \draw[torusarrow] (-0.02,0.08) -- (-0.02,0.92);
    \draw[torusarrow] (1.02,0.08) -- (1.02,0.92);

  \end{scope}
\end{tikzpicture}
\caption{A sample of coordinatized patches from $X_{\text{flow}}$ with high directional strength, arranged by predominant flow direction (x) and assigned fiber coordinate (y). Compare with Figure 7 in~\cite{opt_flow_torus}}
\label{fig: Optical Flow Torus}
\end{figure}

From a modeling perspective, the results of our analysis of $X_{\text{flow}}$ demonstrate how the interplay between local and global measurements depends crucially on the geometry of the chosen feature map. 
On the other hand, a careful fiberwise analysis helped us to both confirm an existing model and identify subtle, informative additional structure in the data which had previously been missed. 
For additional details, see our companion paper~\cite{turow2026extended}.\\ 

\subsection{Folded Klein Bottle}\label{sec: Folded KB}

We demonstrate our pipeline on a noisy synthetic sampling of a 2-manifold $M\subset\mathbb{R}^{8}$ with the topology of a Klein bottle.  
Unlike our previous example, the dataset does not come equipped with any obvious choice of projection map, so we use the DREiMac library's topologically-flavored circular coordinates algorithm \cite{DREiMac}, \cite{Sparse_CC} to obtain an $\mathbb{S}^{1}$-valued base map.  
The data in the fibers of the resulting projection map is concentrated around topological (but not geometric) circles. 
We therefore also use the DREiMac algorithm to construct local circular coordinates. \\

The model manifold $M$ is defined as follows: let $\gamma:\mathbb{R}\to \mathbb{R}^{4}$ be defined by $\gamma(t) = (\sin (t), \cos(t), \sin(2t),0)$ for all $t\in\mathbb{R}$, and let $F = \gamma(\mathbb{R})\subset\mathbb{R}^{4}$.  
Note that $F$ is topologically a circle, but geometrically it is folded. 
In particular, $F$ is invariant under a reflection in the $x$ and $y$ coordinates. 
Now, define $R:[0,2\pi]\to SO(4)$ by 

$$R(t) = \left(\begin{array}{cccc}
a & b & 0 & d\\
b & a & 0 & -d\\
0 & 0 & 1 & 0\\
-d & d & 0 & c\\
\end{array}\right)$$

where 

$$\tau = \frac{t}{2},\hspace{1cm}a = \frac{1 + \cos(\tau)}{2},\hspace{1cm}b = \frac{1 - \cos(\tau)}{2}$$

$$c = \cos(\tau),\hspace{1cm}d =  \frac{\sqrt{2}}{2}\sin(\tau)$$

Observe in particular that $R(0) = I$, $R(2\pi)$ is the matrix which maps $(x,y,z,w)$ to $(y,x,z,-w)$ and $R(t)$ fixes the $z$-axis for all $t$. Finally, define $\widetilde{M}$ by 

$$\widetilde{M} = \{(u,v)\in\mathbb{R}^{4}\times\mathbb{R}^{4}: u = \gamma (t), \ v\in R(t)F\}$$

and apply an $SO(8)$ rotation to obtain the manifold $M$ (the rotation is simply to intermix the coordinates).\\

We generated a dataset $X_{\text{kb}}\subset\mathbb{R}^{8}$ by sampling $10,000$ points from $M$ with a Gaussian noise level of $\sigma = 0.05$. 
Figure~\ref{fig: klein_rips} shows persistence diagrams with $\mathbb{Z}_{2}$- and $\mathbb{Z}_{3}$-coefficients for a subsample of $X_{\text{kb}}$ (recall that the Klein bottle has Betti numbers $\beta_{0} =0$, $\beta_{1} = 1$, $\beta_{2} = 0$, with 2-torsion classes in homology dimensions 1 and 2):

\begin{figure}[h!]
    \centering
    \includegraphics[width=0.75\textwidth]{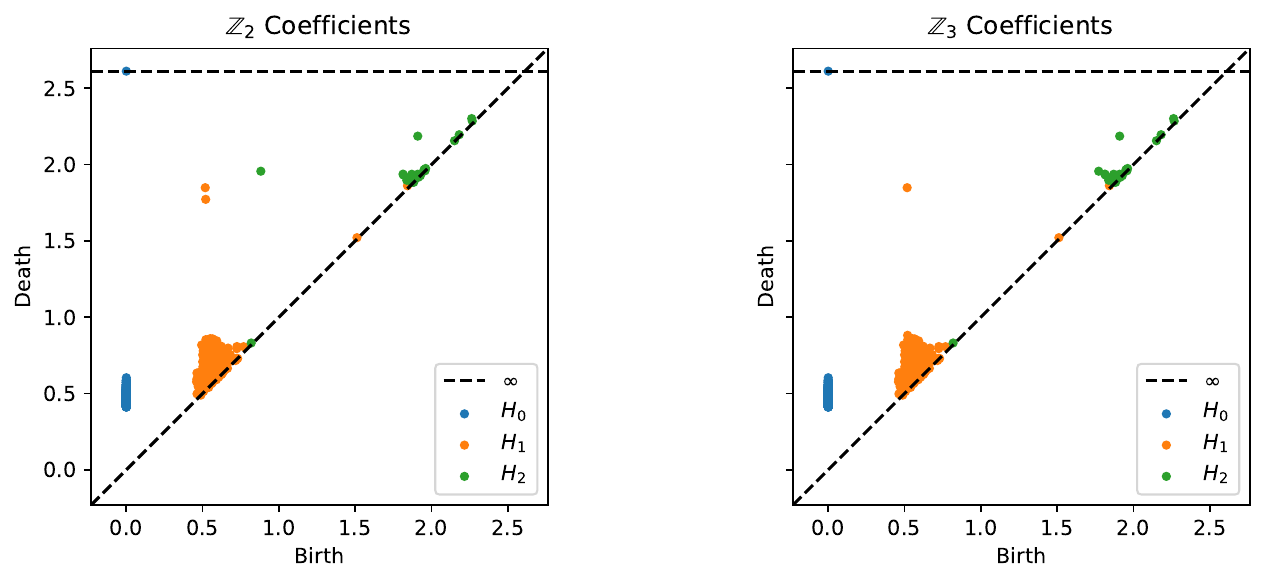}
    \caption{Persistence diagrams for a subsample of $X_{kb}$}
    \label{fig: klein_rips}
\end{figure}

We used the $\mathbb{Z}_{3}$-persistent class as input for the DREiMac circular coordinates algorithm to produce a map $\pi_{\text{kb}}:X_{\text{kb}}\to \mathbb{S}^{1}$.
We then constructed a cover of $\mathbb{S}^{1}$ by metric balls $\mathcal{U} = \{U_{j}\}_{j=1}^{12}$ around equally-spaced landmarks. 
Figure~\ref{fig: klein_pca} shows PCA projections of the data in $\pi^{-1}_{\text{kb}}(U_{j})$ for several $j$, confirming that PCA could not be used to produce local circular coordinates. 
The points are colored according to angular coordinates computed using DREiMac:

\begin{figure}[h!]
    \centering
    \includegraphics[width=0.9\textwidth]{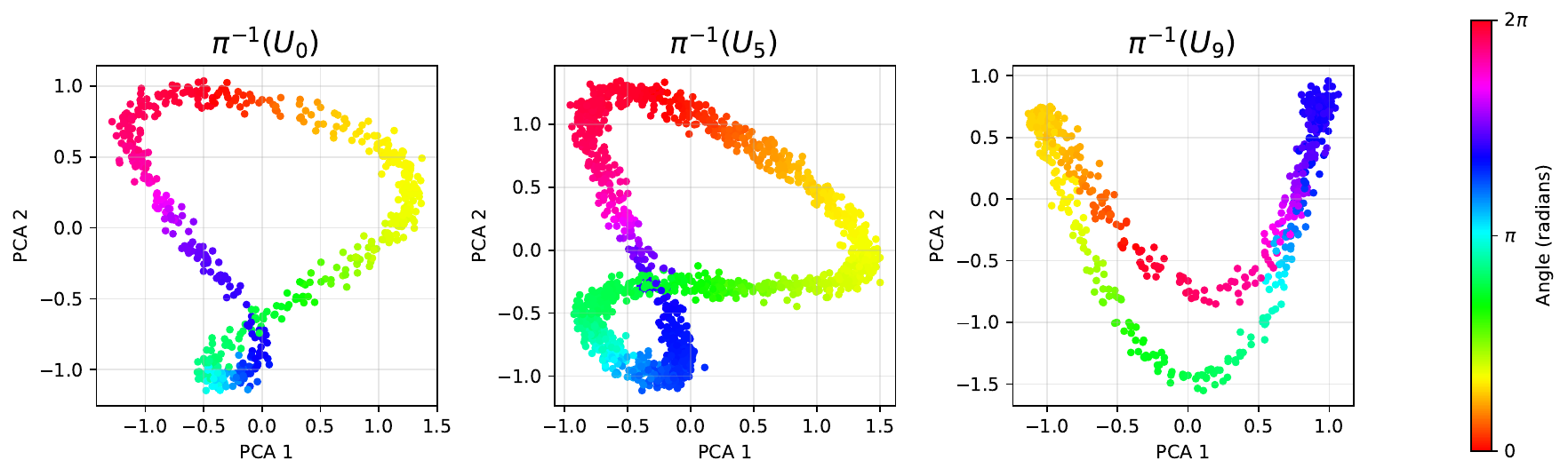}
    \caption{PCA projections of $\pi_{\text{kb}}^{-1}(U_{j})$ for several $j$.}
    \label{fig: klein_pca}
\end{figure}

A characteristic class computation then confirmed the anticipated global Klein bottle structure; a visualization of the nerve decorated by the non-trivial Stiefel-Whitney class representative is shown in Figure~\ref{fig: klein_nerve}.\\ 

\begin{figure}[h!]
    \centering
    \includegraphics[width=0.45\textwidth]{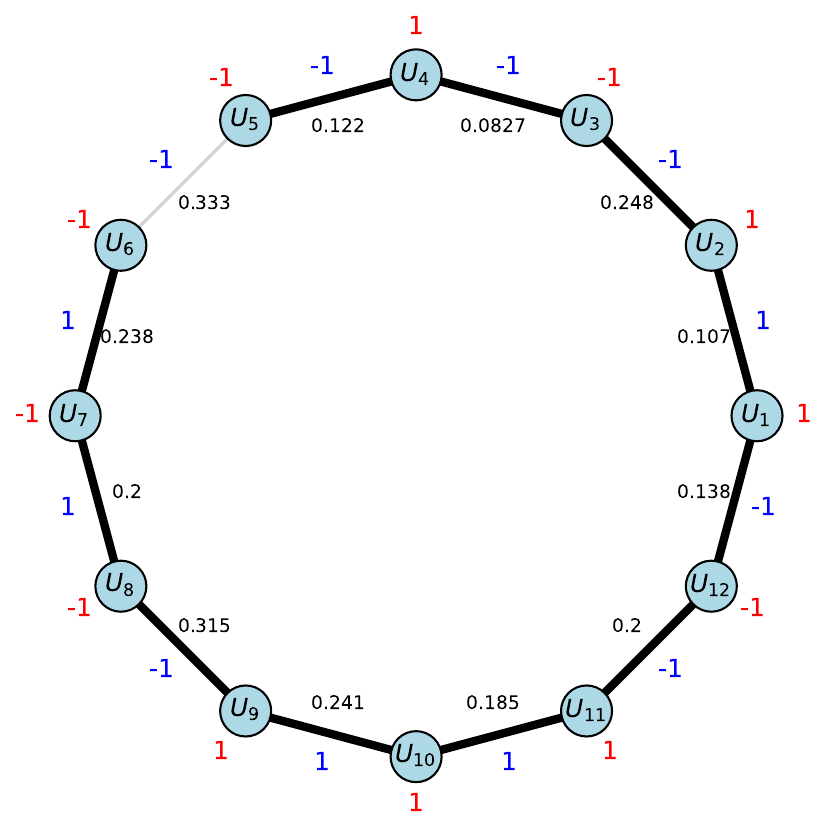}
    \caption{The nerve of the open cover $\mathcal{U} = \{U_{j}\}_{j=1}^{12}$ of $\mathbb{S}^{1}$. Each edge is labeled by its associated weight. 
    The components of the Stiefel-Whitney class representative $\omega$ are shown in blue. 
    Note the class is non-trivial on the full nerve, reflecting the underlying Klein bottle structure.  After the heaviest edge is removed, the class becomes trivial; a potential for the restricted cocycle is shown in red.}
    \label{fig: klein_nerve}
\end{figure}

\subsection{3D Densities}\label{sec: Densities}

Many applications of computer vision and imaging involve data in the form of density or intensity functions on some region of space or time. The ambient dimension for such data is inevitably very large, and distances computed using the standard Euclidean metric often fail to reflect one's intuition about how densities should compare (for example, two densities which differ by a small translation may be very far apart with respect to Euclidean distance). One would ideally opt for an Earth-mover-type distance such as the Wasserstein metric~\cite{rubner2000emd}, but these metrics are prohibitively expensive to compute, so an effective local-to-global approach to topological modeling which circumvents the need to construct a full distance matrix is highly desirable. 
Furthermore, for our experiment below, we sample from a 3-manifold embedded in a space of 3D densities, and persistence diagrams for dimension 3 are computationally intractable using any metric; the only way to certify the global topology is with local computations.\\

Consider the space $\mathcal{D}$ of real-valued functions on the unit ball $\mathbb{B}^{3}$. The space $\mathcal{D}$ has a natural right $SO(3)$-action defined by $(\rho\cdot A)(x) = \rho(A^{T}x)$ for all $\rho\in \mathcal{D}$, $A\in SO(3)$ and $x\in \mathbb{B}^{3}$. In particular, the orbit of any  $\rho_{0}\in\mathcal{D}$ is diffeomorphic to a quotient of $SO(3)$ by the corresponding stabilizer subgroup. For our demonstration, we sampled the $SO(3)$-orbit of a density $\rho_{\text{prism}}\in\mathcal{D}$ concentrated around the boundary of a triangular prism $P\subset\mathbb{B}^{3}$ (see Figure~\ref{fig: prism samples}). Explicitly, $\rho_{\text{prism}}:\mathbb{B}^{3}\to \mathbb{R}$ is given by 

\begin{equation}
    \rho_{\text{prism}}(p) = \exp\left(-\frac{d(p,\partial P)^{2}}{2\sigma^{2}}\right)
\end{equation}

where $d(p,\partial P)$ denotes the (Euclidean) distance from $p$ to the boundary of $P$, and $\sigma = 0.05$. Note that $\rho_{\text{prism}}$ has a 3-fold rotational symmetry in the $yz$-plane as well as a 2-fold rotational symmetry along the $x$-axis. The orbit $\mathcal{O}_{\text{prism}}$ is therefore homeomorphic to the quotient $L(6,1)/\mathbb{Z}_{2}$, where $L(6,1)$ denotes the lens space $\mathbb{S}^{3}/\mathbb{Z}_{6}\cong SO(3)/\mathbb{Z}_{3}$. \\

\begin{figure}[h!]
    \centering
    \includegraphics[width=0.9\textwidth]{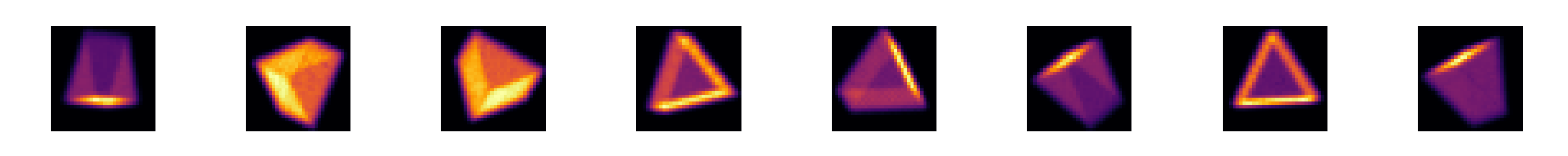}
    \vspace{0.5em} 
    \includegraphics[width=0.9\textwidth]{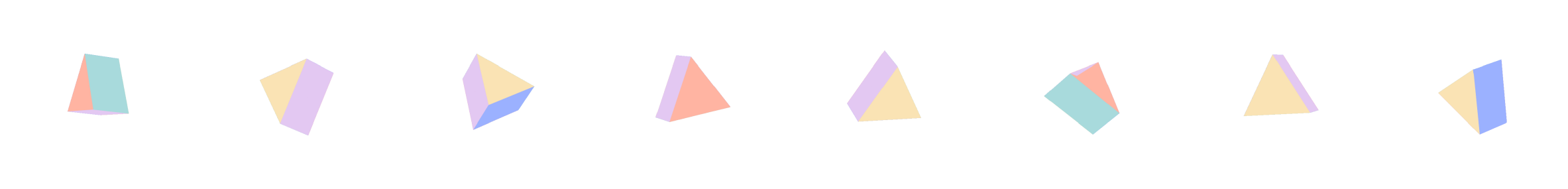}
    \caption{A sample of discrete 3D prism densities from $X_{\text{prism}}$. Each density is represented by its 2D projection along the $x$-axis as well as the corresponding rotation of the prism $P$ used to generate $\rho_{\text{prism}}$. Note that any two rotated prisms which are related by an element of $\text{stab}(\rho_{\text{prism}})$ result in indistinguishable 3D density functions.} 
    \label{fig: prism samples}
\end{figure}

Consider the feature map $\pi:\mathcal{O}_{\text{prism}}\to \mathbb{RP}^{2}$ defined as follows: given any $\rho\in \mathcal{O}_{\text{prism}}$, let $M_{\rho}$ be the covariance matrix given by

\begin{equation*}
    M_{\rho} = \int_{\mathbb{B}^{3}}\rho(x)(x-\mu_{\rho})(x-\mu_{\rho})^{T}dx
\end{equation*}

where $\mu_{\rho} = \int_{\mathbb{B}^{3}}x\ \rho(x)dx$ is the center-of-mass. Then, let $\pi(\rho)\in\mathbb{RP}^{2}$ be the axis corresponding to the smallest eigenvalue of $M_{\rho}$. One can interpret $\pi(\rho)$ as the axis along which $\rho$ has the smallest variance. In particular, for any $\rho\in \mathcal{O}_{\text{prism}}$, $\pi(\rho)$ is precisely the axis of symmetry (explicitly, $\pi(\rho_{\text{prism}}\cdot A) = A^{T}e_{1}$). In Section~\ref{sec: S2 and RP2 bundles}, we show $(\mathcal{O}_{\text{prism}}, \pi, \mathbb{RP}^{2})$ is a non-orientable circle bundle with twisted Euler number $\pm 3$. \\

We generated a synthetic dataset by applying 5,000 random rotations to $\rho_{\text{prism}}$ and sampling the resulting densities on a lattice $L$ (in particular, we restricted a 32 x 32 x 32 lattice on $[-1,1]^{3}$ to $\mathbb{B}^{3}$). The sampled density values were then flattened to vectors $X_{\text{prism}}\subset\mathbb{R}^{32^{3}}$. Figure~\ref{fig:cum_var_plot} shows the PCA cumulative explained variance $X_{\text{prism}}$, clearly illustrating that the embedding of the underlying 3-manifold is highly non-linear. Figure~\ref{fig: Prism density fat fiber} shows a PCA visualization of a 'fat fiber' of the dataset over a small neighborhood in $\mathbb{RP}^{2}$. \\

\begin{figure}[h]
    \centering    \includegraphics[width=0.6\textwidth]{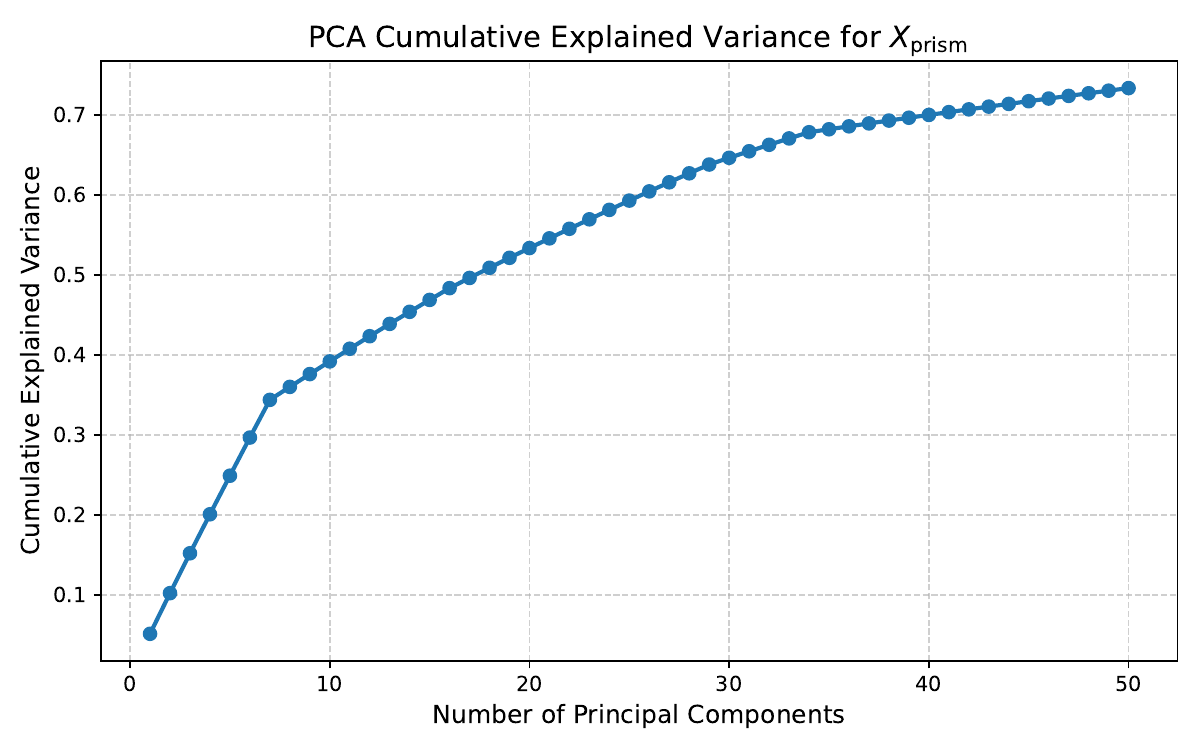}
    \caption{ }
    \label{fig:cum_var_plot}
\end{figure}

\begin{figure}[h!]
    \centering
    \begin{subfigure}{.45\textwidth}
        \centering
        \includegraphics[width=\textwidth]{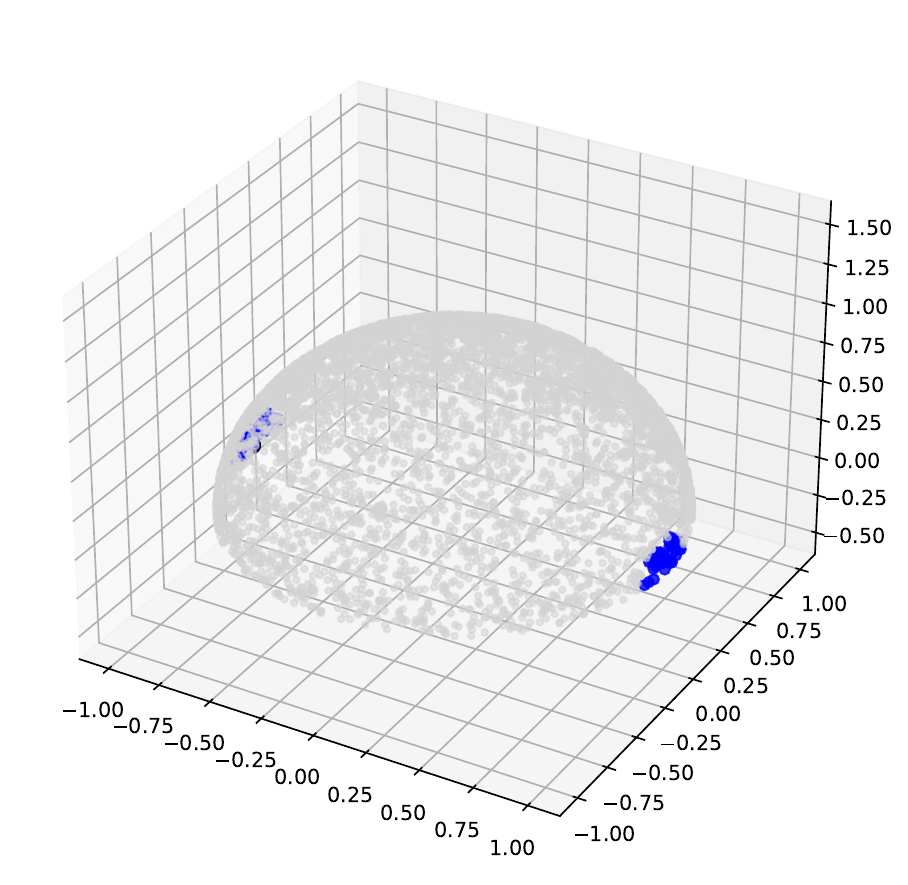}
    \end{subfigure}
    \hspace{0.05\textwidth}
    \begin{subfigure}{.45\textwidth}
        \centering
        \includegraphics[width=\linewidth]{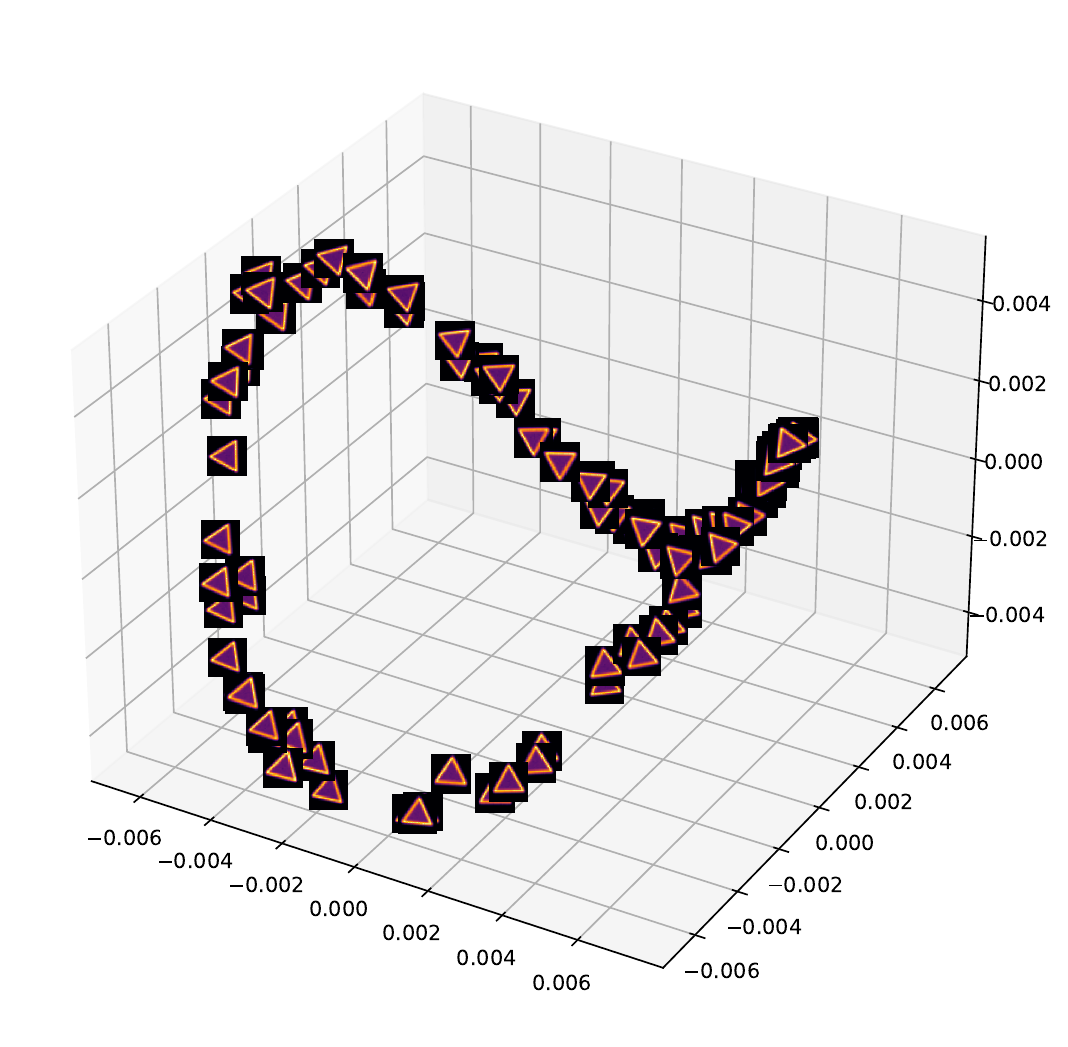}
    \end{subfigure}
    \caption{A PCA visualization of the data in a 'fat fiber' of $X_{\text{prism}}$ (right) over a small neighborhood in $\mathbb{RP}^{2}$ (left), portrayed here by a hemisphere with boundary identifications.}
    \label{fig: Prism density fat fiber}
\end{figure}

One can estimate the covariance matrix $M_{\rho}$ of a density $\rho$ from its discretization $v_{\rho}$ as 

\begin{equation*}
    \widetilde{M}_{\rho} = \sum_{l\in L}v_{\rho}(l)(l-\widetilde{\mu}_{\rho})(l-\widetilde{\mu}_{\rho})^{T}
\end{equation*}

where $\tilde{\mu}_{\rho} = \sum_{l\in L}l\ v_{\rho}(l)$. Setting $\pi(v_{\rho})\in\mathbb{RP}^{2}$ equal to the axis associated with the minimum eigenvalue of $\widetilde{M}_{\rho}$ yields a well-defined map $\pi:X_{\text{prism}}\to\mathbb{RP}^{2}$.\\

To classify the resulting discrete approximate circle bundle and infer the topology of the 3-manifold underlying $X_{\text{prism}}$, we constructed a cover $\mathcal{U} = \{U_{j}\}_{j=1}^{60}$ of $\mathbb{RP}^{2}$; as suggested by Figure~\ref{fig: Prism density fat fiber}, we found (to our surprise) that PCA was sufficient to compute local circular coordinates on the small portion of $X_{\text{prism}}$ in each $\pi^{-1}(U_{j})$.\\

Figure~\ref{fig: Prism Nerve Full}(a) shows the 1-skeleton of $\mathcal{N}(\mathcal{U})$. 
Each vertex $j\in \mathcal{N(U)}$ is labeled by the rotated prism used to generate a sample $x_{j}\in X_{\text{prism}}$ whose base projection $\pi(x_{j})\in\mathbb{RP}^{2}$ was near the center of $U_{j}$. 
Note in particular that the prisms at nodes which are connected by an edge have similar axes of symmetry. 
The edge representing each 1-simplex $(jk)\in\mathcal{N}(\mathcal{U})$ is colored according to the value assigned by the Stiefel-Whitney class representative $\omega = \det_{*}(\Omega)\in Z^{1}(\mathcal{N}(\mathcal{U});\mathbb{Z}_{2})$ (blue for $1$, red for $-1$). 
We verified that $\omega$ was not a coboundary, confirming that one cannot consistently choose an orientation for the data (since $H^{1}(\mathbb{RP}^{2};\mathbb{Z}_{2})\cong\mathbb{Z}_{2}$, verifying non-triviality is sufficient to uniquely determine the orientation class). 
With respect to $\omega$, the twisted Euler number of $\Omega$ was found to be $\pm 3$, in agreement with our theoretical model. 
This confirmed the global topology of the underlying 3-manifold up to homeomorphism. \\

Using the algorithm described in Section~\ref{sec: Persistence}, we computed the cobirth and codeath of each characteristic class representative (see Table~\ref{tab: Persistence Results}). 
Note that each representative was found to be a cocycle on the entire nerve of $\mathcal{U}$. 
The twisted Euler class representative became a coboundary after a single edge was removed from $\mathcal{N}(\mathcal{U})$.
This makes sense, because introducing a cut in $\mathbb{RP}^{2}$ yields a space which is homotopy equivalent to $\mathbb{S}^{1}$, hence has no homology in dimension $2$. 
On the other hand, the orientation class representative only restricted to a coboundary after 110 of the 177 1-simplices in $\mathcal{N}(\mathcal{U})$ were removed. 
One can interpret this result as a reflection of the compatibly non-orientable structures of the bundle and the base space. 
More specifically, the restriction of the underlying bundle $(\mathcal{O}_{\text{prism}}, \pi, \mathbb{RP}^{2})$ to any non-nullhomotopic loop (e.g. an equator $\mathbb{RP}^{1}\subset \mathbb{RP}^{2}$) is non-orientable, hence isomorphic to the Klein bottle. 
Thus, the restriction of an orientation class representative $\omega\in Z^{1}(\mathcal{N}(\mathcal{U});\mathbb{Z}_{2})$ to any subcomplex of $\mathcal{N}(\mathcal{U})$ which includes a non-nullhomotopic loop must still be non-trivial. \\

\begin{figure}[h!]
    \centering
    \begin{subfigure}{.4\textwidth}
        \centering
        \includegraphics[width=\textwidth]{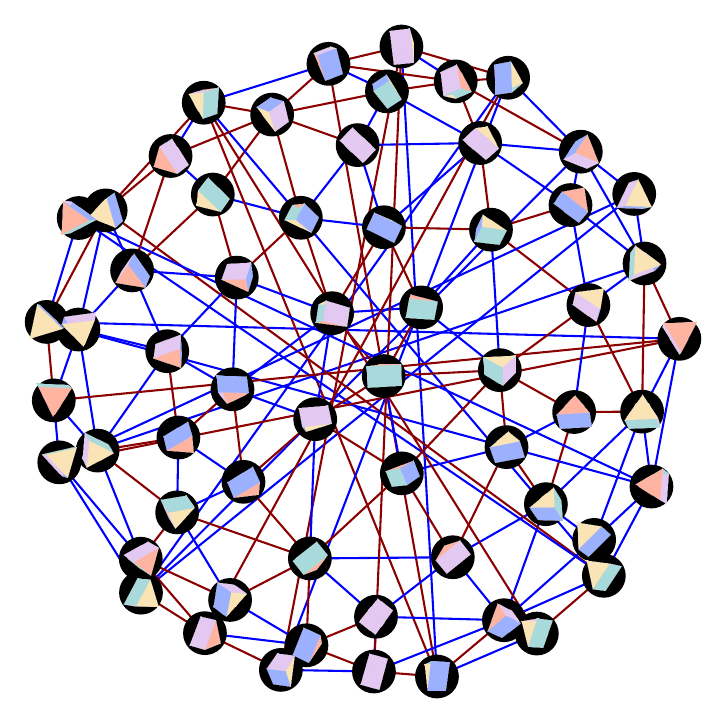}
        \caption{Full Nerve}
        \label{fig: Prism Nerve Full}
    \end{subfigure}
    \hspace{0.05\textwidth}
    \begin{subfigure}{.4\textwidth}
        \centering
        \includegraphics[width=\linewidth]{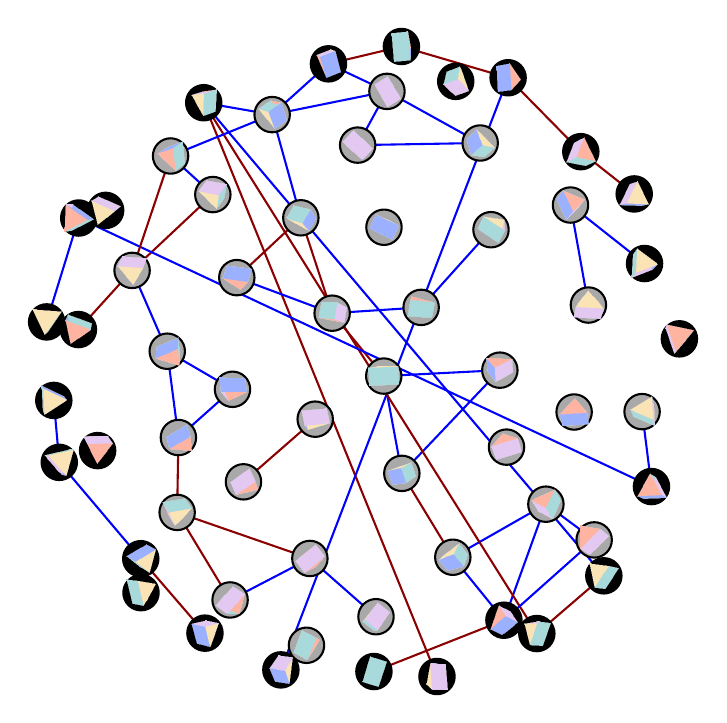}
        \caption{Maximal Subcomplex}
        \label{fig: Prism Nerve Subcomplex}
    \end{subfigure}
    \caption{The 1-skeleton of the nerve $\mathcal{N}(\mathcal{U})$ of the open cover of $\mathbb{RP}^{2}$ used to analyze $X_{\text{prism}}$. Edges are colored according to the values of the Stiefel-Whitney class representative (blue for 1, red for -1). The figure on the right shows the 1-skeleton of the maximal subcomplex of $\mathcal{N}(\mathcal{U})$ on which the characteristic class representatives restricted to coboundaries; the nodes which correspond to open sets containing the base projections of points in $\widetilde{X}_{\text{prism}}$ are highlighted.}
    \label{fig: Prism nerves}
\end{figure}

To illustrate further, we isolated the set of densities $\widetilde{X}_{\text{prism}}\subset X_{\text{prism}}$ whose base point projections satisfied $|z| < 0.1$, then projected these base points to the equator $\mathbb{RP}^{1}\subset\mathbb{RP}^{2}$. Using the same cover of $\mathbb{RP}^{1}$ as we used to analyze the optical flow dataset, we recomputed local coordinates and approximate transition matrices and verified that the resulting discrete approximate cocycle produced a non-trivial orientation class representative.
After dropping the heaviest edge from the nerve, we used the coordinatization algorithm described in Section~\ref{sec: Coordinatization} to compute global coordinates for $\widetilde{X}_{\text{prism}}$ compatible with the resulting subcomplex.\\

Figure~\ref{fig: Prism Density KB} shows a sample of $\widetilde{X}_{\text{prism}}$ arranged according to assigned base angles in $\mathbb{RP}^{1}$ and assigned fiber angles in $\mathbb{S}^{1}$. 
Observe that the base angle assigned to each density roughly captures its axis of symmetry, and that this axis steadily traverses a rotation by $\pi$ as the base angle varies from $0$ to $\pi$. The fiber angles assigned to densities in any narrow vertical strip roughly parametrize a rotation by $\frac{2\pi}{3}$ about a common axis of symmetry. 
On the other hand, careful inspection shows that the direction of rotation abruptly flips near the dotted vertical line, reflecting a location in the nerve where a cut was made before coordinatization. This also confirms the Klein bottle structure of the restricted bundle underlying $\widetilde{X}_{\text{prism}}$.\\

\begin{figure}[h!]
\centering

\begin{subfigure}[t]{0.45\textwidth}
\centering
\begin{tikzpicture}
  \node[anchor=south west, inner sep=0] (img) at (0,0)
    {\includegraphics[width=\textwidth]{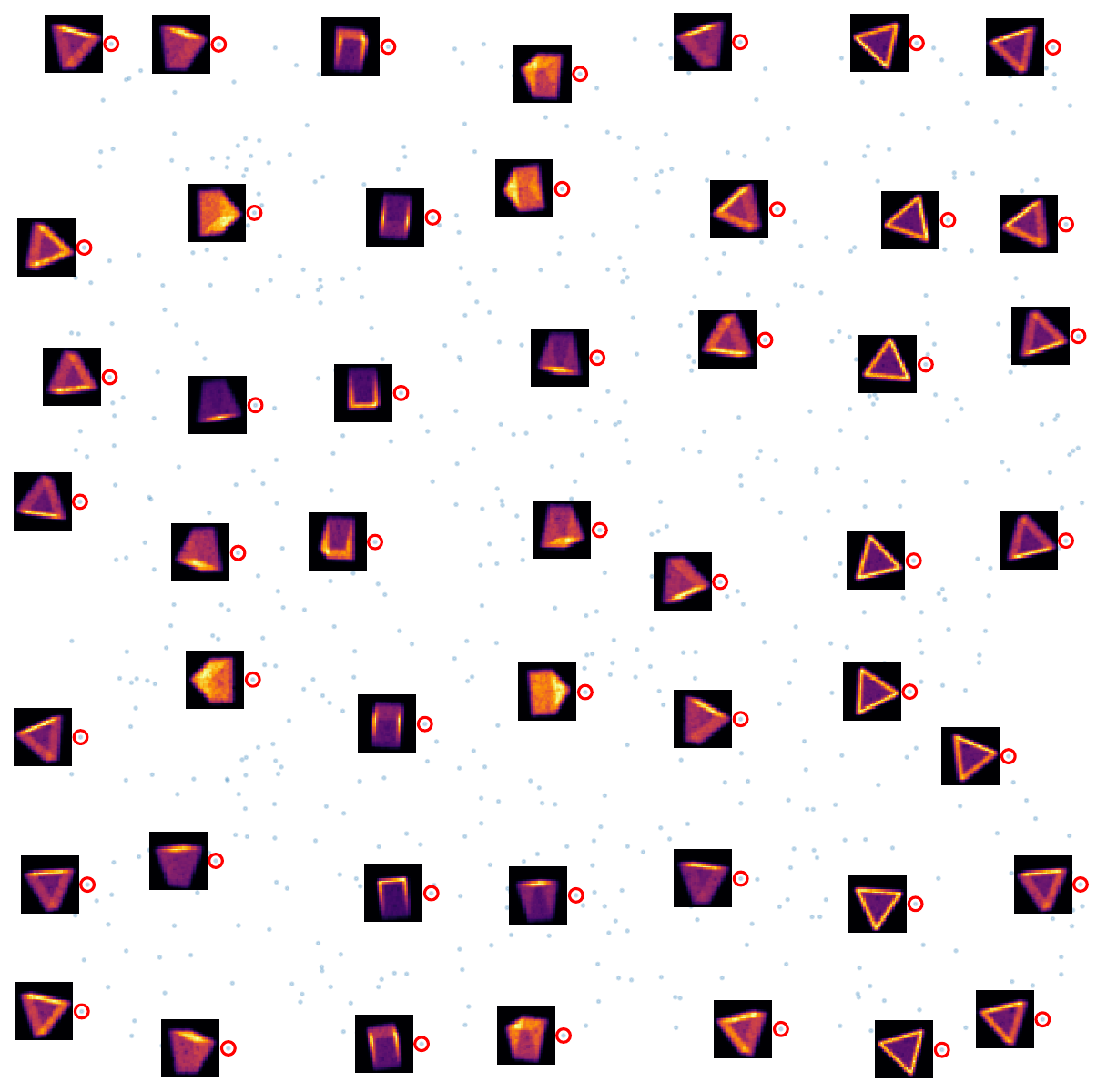}};
  \begin{scope}[x={(img.south east)}, y={(img.north west)}]
    \tikzset{kbArrow/.style={->, line width=0.45pt, draw=black!55}}

    \draw[kbArrow] (0.08,1.02) -- (0.92,1.02);
    \draw[kbArrow] (0.08,-0.02) -- (0.92,-0.02);

    \draw[kbArrow] (-0.02,0.08) -- (-0.02,0.92);
    \draw[kbArrow] (1.02,0.92) -- (1.02,0.08);
  \end{scope}
\end{tikzpicture}
\caption{2D Density Projections}
\label{fig: Density KB}
\end{subfigure}
\hspace{0.08\textwidth} 
\begin{subfigure}[t]{0.45\textwidth}
\centering
\begin{tikzpicture}
  \node[anchor=south west, inner sep=0] (img) at (0,0)
    {\includegraphics[width=\textwidth]{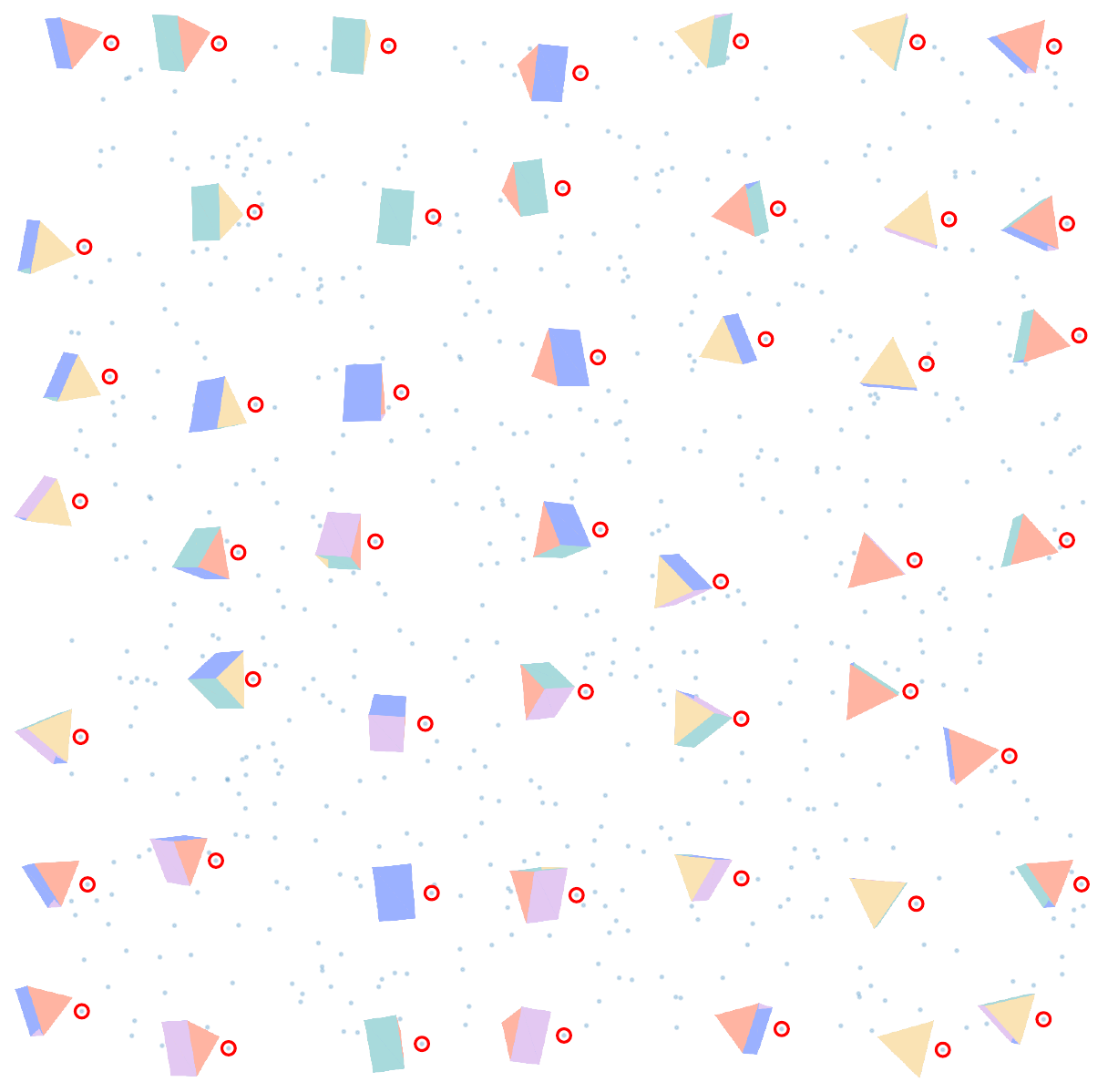}};
  \begin{scope}[x={(img.south east)}, y={(img.north west)}]
    \tikzset{kbArrow/.style={->, line width=0.45pt, draw=black!55}}

    \draw[kbArrow] (0.08,1.02) -- (0.92,1.02);
    \draw[kbArrow] (0.08,-0.02) -- (0.92,-0.02);

    \draw[kbArrow] (-0.02,0.08) -- (-0.02,0.92);
    \draw[kbArrow] (1.02,0.92) -- (1.02,0.08);
  \end{scope}
\end{tikzpicture}
\caption{Underlying Prism Meshes}
\label{fig: Prism Mesh KB}
\end{subfigure}

\caption{A collection of coordinatized samples from $\widetilde{X}_{\text{prism}}$ arranged according to base projection angle (x) and assigned fiber coordinate (y).}
\label{fig: Prism Density KB}
\end{figure}

Finally, to demonstrate our bundle map algorithm, we used the procedure described in Section~\ref{sec: Coordinatization} to construct a point cloud embedding $F:X_{\text{prism}}\to V(2,60)\times_{O(2)}\mathbb{S}^{1}\subset\mathbb{R}^{120}$ for the full dataset which respects the global topology implied by the characteristic classes (recall these were well-defined on the entire nerve $\mathcal{N}(\mathcal{U})$). 
One could potentially reduce the ambient dimension further by applying Principal Stiefel Coordinates to the dataset of frames in $V(2,60)$ used to produce $F$. 
The mean projection errors for different projection dimensions are shown in Figure~\ref{fig: PSC} .\\ 

\begin{figure}[h]
    \centering    \includegraphics[width=0.5\textwidth]{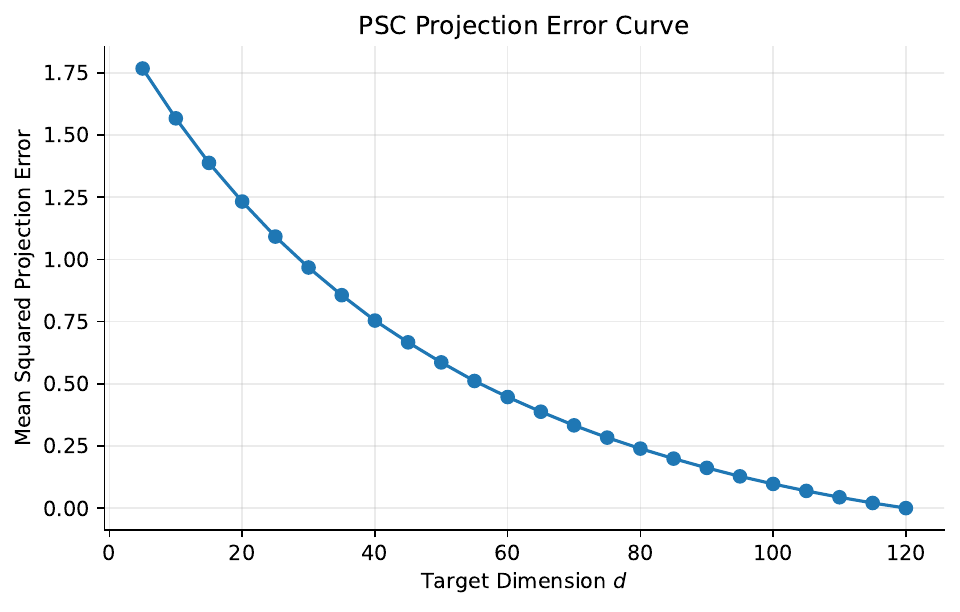}
    \caption{ }
    \label{fig: PSC}
\end{figure}

\section{Conclusion and Future Work}\label{sec: Future Work}

We have introduced a discrete, approximate analog of circle bundles, motivated by data-theoretic applications. We established conditions under which such an object can be identified with a true circle bundle and described algorithms for recovering the classification of the underlying bundle. 
We also proposed a coordinatization scheme that mimics the construction of a bundle map to a universal bundle, which can be tuned to capture different global topologies using a persistence-based approach on characteristic classes.  

Beyond circle bundles, the theory and algorithms presented here could be extended in several ways. 
One could consider discrete, approximate analogs of bundles with other fibers, such as disjoint unions of circles $\coprod_j \mathbb{S}^1_j$, higher-dimensional spheres $\mathbb{S}^n$, or tori $\mathbb{T}^n$. Another promising direction is the study of more general topological objects, such as fibrations or stratified bundles, which allow for singularities in the base or fiber. These generalizations open the door to applications involving more complex or singular data structures.

\renewcommand{\thesection}{A}
\section{Technical Details}\label{sec: Technical Details}

Here we present some background definitions and results which are used throughout the paper. \\

\subsection{Approximate Isometries Of The Circle}

The main result in this section is Lemma~\ref{lemma: chart pair approx transition}, which roughly states that discrete approximate local trivialization functions on a common domain are approximately related by an isometry of $\mathbb{S}^{1}$.
We first review some definitions and introduce technical results which will be used in the proof.\\

\begin{definition}
    Let $\sdist$ denote geodesic distance on $\mathbb{S}^{1}$, and let $d_{\infty}$ be the sup metric on the space $\text{Fun}(\mathbb{S}^{1},\mathbb{S}^{1})$. Note that for $t\in[0,2\pi]$ one has 
    \[
    \sdist(e^{is},e^{it})=\min\{|s-t|,\,2\pi-|s-t|\}\in[0,\pi]
    \]
\end{definition}

The next three lemmas are used in the proof of Proposition~\ref{prop: approx to exact isometry}:\\

\begin{lemma}\label{lemma: homeos opposite deg}
    Suppose $f,g\in\text{Homeo}(\mathbb{S}^{1})$ such that $\deg(f)\neq \deg(g)$. 
    Then, there exists some $z\in\mathbb{S}^{1}$ such that $f(z) = -g(z)$. 
\end{lemma}

\begin{proof}
    Let $h = (-g)^{-1}\circ f\in\text{Homeo}(\mathbb{S}^{1})$. Note that multiplication by $-1$ is a degree-1 homeomorphism, so the properties of degree imply
\[
\deg(h) \quad =\quad  \deg ((-g)^{-1})\cdot \deg(f) \quad =\quad  \deg(-g)\cdot \deg(f) \quad = \quad \deg(g)\cdot \deg(f) \quad =\quad  -1
\]
    One can show that every orientation-reversing homeomorphism of $\mathbb{S}^{1}$ has a fixed point (this is an easy consequence of the Intermediate Value Theorem applied to a continuous lift), and if $z_{0}$ is a fixed point of $h$, then $(-g^{-1})\circ f(z_{0}) = z_{0}$ (or equivalently, $f(z_{0}) = -g(z_{0})$).\\   
\end{proof}

\begin{lemma}\label{lemma: metric map is 1-Lipschitz}
    Let $(X,d)$ be a metric space. For any fixed $x_{0}\in X$, the assignment $x\longmapsto d(x_{0},x)$ is $1$-Lipschitz:
\[
|d(x_{0},x) - d(x_{0},x')|\leq d(x,x')\hspace{1cm}x,x'\in X
\]
\end{lemma}

\begin{proof}
    Define $f(x) = d(x_{0},x)$ for all $x\in X$. 
    For any $x,x'\in X$, the triangle inequality implies
\[
d(x,x')\geq d(x_{0},x) - d(x_{0},x') \hspace{1cm}d(x,x')\geq d(x_{0},x') - d(x_{0},x)
\]
    The result is immediate. \\
\end{proof}

\begin{lemma}\label{lemma: rigid pair general}
    Let $a,b\in\mathbb{S}^{1}$ such that $\sdist (a,b) = \frac{\pi}{2}$. 
    Suppose $z,w\in \mathbb{S}^{1}$ such that

\[
\max\{\left| \sdist (z,a) - \sdist (w,a)\right|, \quad \left| \sdist (z,b) - \sdist (w,b)\right|\}\quad \leq \quad \delta
\]

for some $\delta <\frac{\pi}{2}$. Then $\sdist (z,w) \leq \delta$.\\
\end{lemma}

\begin{proof}
    Since $\sdist$ is $O(2)$-invariant, it suffices to prove the lemma for the case $a = 1$ and $b = i$. \\

    Let $s,t\in [0,2\pi)$ and $z=e^{it}, w=e^{is}\in\mathbb{S}^{1}$. 
    We first show $\sdist (e^{it},e^{is})\leq \delta$ in the case when $t < \frac{\pi}{2}$ and $s <\pi$. We then argue that the hypothesis $\delta <\frac{\pi}{2}$ implies $e^{is}$ and $e^{it}$ cannot lie in opposite quadrants and use $O(2)$-invariance to conclude that the identity holds for all possible $s,t$.\\
    
    Suppose $t\in[0,\frac{\pi}{2})$. If $s\in[0,\pi)$, then 

\[
\sdist (e^{it},e^{is}) \quad = \quad |t-s| \quad = \quad \left|\sdist(e^{it},1) - \sdist(e^{is},1)\right| \quad \leq \quad \delta 
\]

    The case where $s\in [\frac{3\pi}{2},2\pi)$ reduces to the case above by applying a reflection, so $\sdist (e^{it},e^{is})\leq \delta$ in this case by $O(2)$-invariance of $\sdist$. 
    The case where $s\in [\pi, \frac{3\pi}{2})$ is forbidden by the hypothesis $\delta < \frac{\pi}{2}$. 
    To see this, 
    suppose $s\in [\pi, \frac{3\pi}{2})$. 
    Then, 

\begin{equation}\label{eq: 1 bound}
\delta \quad \geq \quad \left|\sdist(e^{it},1) - \sdist(e^{is},1)\right|\quad  =\quad  |t - (2\pi - s)| \quad =\quad  |2\pi - (s+t)|
\end{equation}

\begin{equation}\label{eq: i bound}
\delta \quad \geq \quad \left|\sdist(e^{it},i) - \sdist(e^{is},i)\right| \quad =\quad  \left|\left(\frac{\pi}{2} - t\right) - \left(s - \frac{\pi}{2}\right)\right| \quad =\quad  |\pi - (s + t)|
\end{equation}

    Since $\delta < \frac{\pi}{2}$, inequalities~\eqref{eq: 1 bound} and~\eqref{eq: i bound} are contradictory. 
    We conclude $\sdist(e^{it},e^{is})$ for all feasible $s$ when $t\in [0,\frac{\pi}{2})$.
    Thus, by $O(2)$-invariance of $\sdist$, the result holds for all possible $s,t\in [0,2\pi)$.\\
\end{proof}

\begin{proposition}\label{prop: approx to exact isometry}
    Suppose $h:\mathbb{S}^{1}\to \mathbb{S}^{1}$ is any function satisfying $\text{dist}(h)\leq \eps$ for some $\eps < \frac{\pi}{4}$. Then,\\

    \begin{enumerate}
        \item There exists some $\Omega\in O(2)$ such that $d_{\infty}(\Omega,h)\leq 2\eps$,\\
        
        \item The determinant of any such $\Omega$ is the same,\\
        
        \item The map $d_{\infty}(h,\cdot):O(2)\to \mathbb{R}_{\geq 0}$ has a unique minimizer.\\
    \end{enumerate}
\end{proposition}

\begin{proof}
    First, observe that there is a unique $\Omega\in O(2)$ which satisfies $\Omega\circ h(1)=1$ and $\sdist(\Omega\circ h(i), i) \leq \eps$. Let $\tilde{h} = \Omega\circ h$; to prove \textbf{1} above, it suffices to show $d_{\infty}(\tilde{h},\text{Id})\leq 2\eps$. 
    To see this, note that for any $z\in \mathbb{S}^{1}$, we have 

\begin{equation}\label{eq: approx iso 1}
\left|\sdist (\tilde{h}(z),1) - \sdist (z,1)\right|\quad = \quad \left|\sdist (h(z),h(1)) -\sdist(z,1)\right|\quad \leq \quad \eps
\end{equation}

    and

\begin{align*}
    \left|\sdist (\tilde{h}(z),i) - \sdist (z,i)\right|\quad &\leq \quad \left|\sdist (\tilde{h}(z),i) - \sdist (\tilde{h}(z),\tilde{h}(i))\right|\quad + \quad \left|\sdist (\tilde{h}(z),\tilde{h}(i)) - \sdist (z,i)\right|\\
    &\leq \quad \sdist (i, \tilde{h}(i))\quad + \quad  \left|\sdist (h(z),h(i)) -\sdist(z,i)\right|\\
    &\leq \quad 2\eps
\end{align*}
    
    By Lemma~\ref{lemma: rigid pair general}, $\sdist (\tilde{h}(z),z)\leq 2\eps$ for all $z\in\mathbb{S}^{1}$, so $d_{\infty}(\tilde{h},\text{Id})\leq 2\eps$. \\
    
    \textbf{2.} Suppose $\Omega'\in O(2)$ satisfies $d_{\infty}(h,\Omega')\leq 2\eps$ and $\det(\Omega')\neq\det(\Omega)$. 
    By the triangle inequality, 

\[
d_{\infty}(\Omega',\Omega)\quad \leq \quad d_{\infty}(\Omega',h) \ + \ d_{\infty}(h,\Omega)\quad\leq \quad 4\eps\quad <\quad\pi 
\]
    
    On the other hand, Lemma~\ref{lemma: homeos opposite deg} implies $d_{\infty}(\Omega',\Omega^{-1}) = \pi$, a contradiction. 
    We conclude $\det(\Omega') = \det(\Omega^{-1})=\det(\Omega)$. \\

    \textbf{3.} Let $L_{h}:O(2)\to \mathbb{R}_{\geq 0}$ be defined by $L_{h}(\Omega) = d_{\infty}(h,\Omega)$. 
    By compactness, at least one minimizer of $L_{h}$ exists. 
    By $\textbf{1}$, the set of minimizers is contained in $S_{\eps}=L_{h}^{-1}([0,2\eps])$, and $\textbf{2}$ implies $S_{\eps}$ is contained in a single connected component of $O(2)$. 
    We assume minimizers of $L_{h}$ have positive determinant (if not, the same result is obtained by first pre-composing $h$ with any reflection in $O(2)$).\\ 
    
    For $t\in [0,2\pi)$, define $s(t)\in [0,2\pi)$ such that $h(e^{it}) = e^{is(t)}$, and let $g(t) = s(t) - t$.
    If $R_{\theta}$ minimizes $L_{h}$ for some $\theta\in [0,2\pi)$, then
    
\[
\sdist(e^{ig(t)}, e^{i\theta}) \quad =\quad  \sdist (e^{i(s(t) - t)}, e^{i\theta}) \quad =\quad  \sdist(e^{i s(t)},e^{i(\theta + t)}) \quad \leq \quad L_{h}(R_{\theta})\quad \leq \quad 4\eps \quad < \quad \frac{\pi}{2}
\]
    for all $t\in [0,2\pi)$. 
    It follows that the image $e^{ig([0,2\pi))}$ is contained in an arc of diameter at most $4\eps < \pi$ containing $e^{i\theta}$. The unique minimizer of $L_{h}$ is then $R_{\theta^{*}}$, where $e^{i\theta^{*}}$ is the (unique) midpoint of the smallest arc containing $e^{ig([0,2\pi))}$. 
    From uniqueness of the minimizer, we conclude $\theta^{*} = \theta\mod 2\pi$.\\
\end{proof}

We finally arrive at the main result of this section: \\

\begin{lemma}\label{lemma: Appendix chart pair approx transition}
    Let $\pi:X\to B$ be a function between metric spaces. 
    Suppose $\{\varphi_{j}:(\pi^{-1}(U_{j}),\pi,U_{j})\to (U_{j}\times\mathbb{S}^{1}_{r_{j}},p_{B},U_{j})\}_{j=1}^{2}$ are discrete $(r,\beta,\eps)$-approximate local trivializations for some $r_{1},r_{2} \geq r$ and open sets $U_{1},U_{2}\subseteq B$ such that $\bar{B}_{\beta}(\pi(x))\subset U_{12}=U_{1}\cap U_{2}$ for some $x\in X$. 
    Then, 

\[
    |r_{1}-r_{2}|\leq \frac{2}{\pi}(\eps + \beta)
\]
    
    and there exists some $\Omega_{12}\in O(2)$ such that 

\[
    |f_{1}(x) - \Omega_{12} f_{2}(x)| \quad \leq \quad \tilde{\tau}(\varphi_{1},\varphi_{2})
\]
 
    for all $x\in \pi^{-1}(U_{12})$, where $f_{j}=p_{\mathbb{S}^{1}}\circ\varphi_{j}$ for $j = 1,2$ and $\widetilde{\tau}(\varphi_{1},\varphi_{2})$ is defined by~\eqref{eq: pair eta and tau}.\\
\end{lemma}

\begin{proof}
    For $j = 1,2$, let $\psi_{j}:B\times\mathbb{S}^{1}\to X$ be an almost-inverse for $\varphi_{j}$. 
    Let $\pi_{j} = p_{B}\circ\varphi_{j}$ and $f_{j} = p_{\mathbb{S}^{1}}\circ\varphi_{j}$.
    Choose some $x_{12}\in \pi^{-1}(U_{1}\cap U_{2})$ such that $\bar{B}_{\beta}(\pi(x_{12}))\subset U_{1}\cap U_{2}$, let $b_{12} = \pi(x_{12})$ and let $b_{j} = p_{B}\circ\varphi_{j}(x_{12})$ for $j = 1,2$. 
    For any $z\in\mathbb{S}^{1}$, we have $d_{B}(\pi\circ \psi_{2}(b_{2},z), b_{12})
    \leq \frac{\beta}{2}$, so $\psi_{2}(b_{2},z)\in \bar{B}_{\frac{\beta}{2}}(b_{12})\subseteq \bar{B}_{\beta}(b_{12})\subseteq U_{jk}$.
    We therefore have a well-defined function $\Lambda_{12}:\mathbb{S}^{1}_{r_{2}}\to\mathbb{S}^{1}_{r_{1}}$ given by $\Lambda_{12}(z) = f_{1}\circ\psi_{2}(b_{2},z)$ for all $z\in\mathbb{S}^{1}$. \\
    
    We now estimate $\text{dist}(\Lambda_{12})$ in order to invoke Proposition~\ref{prop: approx to exact isometry}. 
    Let $T_{12}:\mathbb{S}^{1}\to U_{1}\times\mathbb{S}^{1}$ be defined by $T_{12}(z) = \varphi_{1}\circ\psi_{2}(b_{2},z)$ for all $z\in\mathbb{S}^{1}$. 
    Note that for any $z\in\mathbb{S}^{1}$, 
    
\[
d_{B}(\pi\circ\psi_{2}(b_{2},z), b_{12}) \quad  \leq \quad d_{B}(\pi\circ \psi_{2}(b_{2},z), b_{2})+d_{B}(b_{2},b_{12}) \quad  \leq \quad \beta    
\]
    
    so $\psi_{2}(b_{2},z)\in \bar{B}_{2\beta}(b_{12})\subset \pi^{-1}(U_{1})$ and $T_{12}(z)$ is well-defined. Moreover, for any $z,z'\in\mathbb{S}^{1}$, 

\begin{align*}
    d_{B}(p_{B}\circ T_{12}(z), \ p_{B}\circ T_{12}(z')) \ &= \
    d_{B}(\pi_{1}\circ\psi_{2}(b_{2},z), \ \pi_{1}\circ\psi_{2}(b_{2},z'))\\
    &\leq d_{B}(\pi_{1}\circ\psi_{2}(b_{2},z), \ \pi \circ\psi_{2}(b_{2},z))\\
    &+ d_{B}(\pi\circ\psi_{2}(b_{2},z),  b_{2}) \\
    &+ d_{B}(b_{2}, \pi \circ\psi_{2}(b_{2},z'))\\
    &+ d_{B}(\pi\circ\psi_{2}(b_{2},z'), \pi_{1}\circ\psi_{2}(b_{2},z'))\\
    &\leq 2\beta 
\end{align*}   
     
Since $\text{dist}(\varphi_{1}),\text{dist}(\psi_{2})\leq \eps$, we have 

\[
|d_{\infty}^{r_{1}}(T_{12}(z), T_{12}(z')) - \sdist^{r_{2}}(z,z')| \quad \leq \quad  2\eps
\]

which in particular implies

\begin{equation}\label{eq: Lambda upper}
    \sdist^{r_{1}}(\Lambda_{12}(z), \Lambda_{12}(z')) \quad \leq \quad d^{r_{1}}_{\infty}(T_{12}(z), T_{12}(z')) \quad \leq \quad \sdist^{r_{2}}(z,z') + 2\eps
\end{equation}

for any $z,z'\in \mathbb{S}^{1}$. On the other hand, since

\begin{equation*}
    d^{r_{1}}_{\infty}(T_{12}(z), T_{12}(z')) \quad  = \quad \max\{d_{B}(p_{B}\circ T_{12}(z), \ p_{B}\circ T_{12}(z')), \hspace{5mm} d^{r_{1}}_{\mathbb{S}^{1}}(\Lambda_{12}(z), \ \Lambda_{12}(z'))\}
\end{equation*}

we can apply the algebraic identity $v\geq \max\{u,v\}-u$ (for any $u,v\geq 0$) to obtain

\begin{align*}
    \sdist^{r_{1}}(\Lambda_{12}(z), \  \Lambda_{12}(z')) \ &\geq \ d_{\infty}^{r_{1}}(T_{12}(z), T_{12}(z')) - d_{B}(p_{B}\circ T_{12}(z), \ p_{B}\circ T_{12}(z'))\\
    &\geq (\sdist^{r_{2}}(z,z')-2\eps) - 2\beta\\
    &=\sdist^{r_{2}}(z,z') - 2(\eps + \beta)
\end{align*}

Combining with~\eqref{eq: Lambda  upper}, we conclude $\text{dist}(\Lambda_{12}) \leq 2(\eps + \beta)$ as a function from $\mathbb{S}^{1}_{r_{2}}$ to $\mathbb{S}^{1}_{r_{1}}$. 
As an immediate consequence, we find that for any $z,z'\in\mathbb{S}^{1}$,

\[
\pi r_{2} - 2(\eps+\beta)\quad \leq \quad \sdist^{r_{1}}(\Lambda_{12}(z),\Lambda_{12}(z')) \quad \leq \quad \pi r_{1}
\]

which implies $\pi r_{2} - 2(\eps+\beta)\leq \pi r_{1}$, or $\pi (r_{2} - r_{1})\leq 2(\eps+\beta)$. Swapping the roles of $r_{1}$ and $r_{2}$ yields $\pi (r_{1} - r_{2}) \leq 2(\eps+\beta)$, so 

\[
 |r_{2} - r_{1}| \quad \leq\quad  \frac{2}{\pi}(\eps+\beta)
\]

Now, observe

\[
2(\eps+\beta) \quad \geq\quad  \left|\sdist^{r_{1}} (\Lambda_{12}(z),\Lambda_{12}(z')) - \sdist^{r_{2}} (z,z')\right| \quad = \quad \left|r_{1}\sdist (\Lambda_{12}(z),\Lambda_{12}(z')) - r_{2}\sdist (z,z')\right|
\]

which implies

\begin{align*}
    \left|\sdist (\Lambda_{12}(z),\Lambda_{12}(z')) - \sdist (z,z')\right| \quad &\leq\quad   \left|\sdist (\Lambda_{12}(z),\Lambda_{12}(z')) - \frac{r_{2}}{r_{1}}\ \sdist(z,z')\right| + \left|\frac{r_{2}}{r_{1}}\sdist(z,z') - \sdist(z,z')\right|\\
    &\leq \quad \frac{2}{r_{1}}(\eps+\beta) + \frac{\sdist(z,z')}{r_{1}}|r_{2} - r_{1}|\\
    & \leq \quad \frac{2}{r_{1}}(\eps+\beta) + \frac{\pi}{r_{1}}|r_{2} - r_{1}|\\
    &\leq \quad \frac{4}{r_{1}}(\eps+\beta) 
\end{align*}

Since $r_{1} > \frac{16}{\pi}(\eps+\beta)$ by definition of a discrete $(r,\beta,\eps)$-approximate local trivialization, the inequality above implies that, as a function from $\mathbb{S}^{1}$ to $\mathbb{S}^{1}$ (with the usual unscaled geodesic distance), $\text{dist}(\Lambda_{12}) < \frac{\pi}{4}$. 
Proposition~\ref{prop: approx to exact isometry} then implies that the loss function $\Omega\longmapsto d_{\infty}(\Lambda_{12}, \Omega)$ has a unique minimizer $\Omega_{12}$ in $O(2)$, and that $d_{\infty}(\Lambda_{12},\Omega_{12})\leq \frac{8}{r_{1}}(\eps+\beta)$. \\

Now, let $\tilde{x}\in\pi^{-1}(U_{12})$ be arbitrary. 
Let $\varphi_{2}(\tilde{x})=(\tilde{b},\tilde{z})\in U_{2}\times\mathbb{S}^{1}_{r_{2}}$ and let $\tilde{y}=\psi_{2}(b_{2},\tilde{z}_{2})\in \pi^{-1}(U_{12})$.
Observe,
\begin{align*}
    d_{B}(\tilde{b}_{2},b_{2})\quad &\leq \quad d_{B}(\tilde{b}_{2}, \pi(\tilde{x}))  +  d_{B}(\pi(\tilde{x}), \pi(x_{12})) + d_{B}(\pi(x_{12}), b_{2})\\
    &\leq \quad \frac{\beta}{2} + \text{diam}(U_{12}) + \frac{\beta}{2} \\
    &= \quad \text{diam}(U_{12}) + \beta
\end{align*}

By the codistortion property of $\varphi_{2}$ and $\psi_{2}$,

\[
d_{X}(\tilde{x},\tilde{y})\quad \leq \quad d_{\infty}^{r_{2}}(\varphi_{2}(\tilde{x}), (b_{2}, \tilde{z}_{2}))  + \eps\quad \leq \quad d_{B}(\tilde{b}_{2},b_{2}) + \eps\quad \leq \quad \text{diam}(U_{12}) + \eps + \beta
\]

and since $\text{dist}(\varphi_{1})\leq \eps$, 

\[
    \sdist^{r_{1}}(f_{1}(\tilde{x}), f_{1}(\tilde{y}))\quad \leq \quad d_{\infty}^{r_{1}}(\varphi_{1}(\tilde{x}), \varphi_{1}(\tilde{y}))\quad \leq \quad d_{X}(\tilde{x},\tilde{y}) + \eps\quad \leq \quad \text{diam}(U_{12}) + 2\eps + \beta
\]

or, equivalently, 

\[
\sdist(f_{1}(\tilde{x}), f_{1}(\tilde{y}))\quad \leq \quad \frac{1}{r_{1}}\left(\text{diam}(U_{12}) + 2\eps + \beta\right)
\]

Finally, noting that $f_{1}(\tilde{y}) = \Lambda_{12}(f_{2}(\tilde{x}))$, we obtain

\begin{align*}
    \sdist (f_{1}(\tilde{x}), \Omega_{12} f_{2}(\tilde{x}))\quad &\leq \quad \sdist (f_{1}(\tilde{x}), f_{1}(\tilde{y})) + \sdist (f_{1}(\tilde{y}), \Omega_{12}f_{2}(\tilde{x}))\\
    &\leq \quad \frac{1}{r_{1}}(\text{diam}(U_{12}) + 2\eps + \beta) + \sdist(\Lambda_{12}(f_{2}(\tilde{x})), \Omega_{12}f_{2}(\tilde{x}))\\
    &\leq \quad \frac{1}{r_{1}}(\text{diam}(U_{12})+2\eps+\beta+8(\eps+\beta))\\
    &= \quad \frac{1}{r_{1}}(\text{diam}(U_{12})+10\eps +9\beta)\\
    &\leq \quad \tau (\varphi_{1},\varphi_{2})
\end{align*}
\end{proof}

\subsection{Circle Bundles Over $\mathbb{RP}^{2}$}\label{sec: S2 and RP2 bundles}

In this section, we provide some details of the construction and classification of the circle bundle underlying the dataset used in Section~\ref{sec: Densities}.
The bundle is obtained as a quotient of the classical principal $\mathbb{S}^{1}$-bundle over $\mathbb{S}^{2}$ with total space $SO(3)$, briefly described below:\\

For any $v\in\mathbb{S}^{2}$ and $\theta\in\mathbb{R}$ let $R_{v,\theta}\in SO(3)$ denote the counterclockwise rotation by $\theta$ about $v$. 
Fixing a $v_{0}\in \mathbb{S}^{2}$, we obtain a right $\mathbb{S}^{1}$-action defined by $A\cdot e^{i\theta} = AR_{v_{0},\theta}$. One can easily verify that this action is free and that the orbits are precisely the fibers of the map $\pi:SO(3)\to \mathbb{S}^{2}$ defined by $\pi(A) = Av_{0}$. 
We conclude that $(SO(3),\pi, \mathbb{S}^{2})$ is a principal $\mathbb{S}^{1}$-bundle. 
Moreover, the isomorphism class of this bundle does not depend on the choice of $v_{0}$. 
To see this, observe that for any $v_{0},v_{1}\in\mathbb{S}^{2}$ and $B\in SO(3)$ such that $Bv_{0} = v_{1}$, the map $F:SO(3)\to SO(3)$ by $F(A) = AB^{-1}$ is an equivariant bundle isomorphism.
We therefore take $v_{0} = e_{1}$ without loss of generality in what follows.\\

The group $SO(3)$ admits a (universal) double cover $\text{Spin}(3)\cong\mathbb{S}^{3}\subset\mathbb{H}$ via the map $\text{Adj}:\mathbb{S}^{3}\to SO(3)$ defined by $\text{Adj}_{q}(v)=qvq^{-1}$. 
Here we are identifying $\mathbb{S}^{2}$ with the pure imaginary unit quaternions via $(x,y,z)\leftrightarrow xi+yj+zk$. 
Thinking of $\mathbb{S}^{1}$ as a subgroup of the unit quaternions, we have a free right $\mathbb{S}^{1}$-action on $\mathbb{S}^{3}$ by group multiplication. 
Moreover, the continuous surjection $\widetilde{\pi}:\mathbb{S}^{3}\to \mathbb{S}^{2}$ defined by $\widetilde{\pi}(q) = \text{Adj}_{q}(e_{1})$ is invariant under this action, so $(\mathbb{S}^{3},\widetilde{\pi},\mathbb{S}^{2})$ is a principal $\mathbb{S}^{1}$ bundle (called the Hopf bundle). 
In particular, we have the following commutative diagram:

\begin{equation}\label{eq: Hopf SO(3) diagram}
\begin{tikzcd}
\mathbb{S}^{3} \arrow[rr, "\text{Adj}"] \arrow[dr, swap, "\tilde{\pi}"'] & & \mathrm{SO}(3) \arrow[dl, "\pi"] \\
& \mathbb{S}^{2} &
\end{tikzcd}
\end{equation}

Note, however, that the two principal circle bundles are not isomorphic. Indeed, $\text{Adj}$ is not a homeomorphism, nor is it $\mathbb{S}^{1}$-equivariant (in fact, one has $\text{Adj}_{q\cdot e^{i\theta}} = \text{Adj}_{q}\cdot e^{i(2\theta)}$). 
The Hopf bundle famously has Euler number -1 with respect to the standard orientation on $\mathbb{S}^{2}$ (for a proof, see~\cite{husemoller}). 
The following proposition immediately implies that the Euler number of the principal bundle $(SO(3),\pi,\mathbb{S}^{2})$ is $-2$, and will also be useful for classifying the circle bundle model underlying the density dataset in Section~\ref{sec: Implementations}: \\

\begin{proposition}\label{prop: Euler p-cover}
   Let $\xi = (E,\pi,B)$ be a principal $\mathbb{S}^{1}$-bundle, and for any $p\in\mathbb{Z}$, let $\varphi_{p}:\mathbb{S}^{1}\to \mathbb{S}^{1}$ be the group homomorphism defined by $\varphi_{p}(z) = z^{p}$. Then $(\varphi_{p})_{*}\xi$ is a principal $\mathbb{S}^{1}$-bundle over the same base space with Euler number $e((\varphi_{p})_{*}\xi)=pe(\xi)$ and total space $E\times_{\varphi_{p}}\mathbb{S}^{1}\cong E/\mathbb{Z}_{p}$.\\        
\end{proposition}

\begin{proof}
    See, for example, Chapter 4 of ~\cite{husemoller}.\\
\end{proof}
 
\begin{corollary}
    The principal $\mathbb{S}^{1}$-bundle $(SO(3),\pi,\mathbb{S}^{2})$ described above has Euler number $-2$ with respect to the standard orientation on $\mathbb{S}^{2}$.
\end{corollary}

\begin{proof}
    Let $\xi_{H} = (\mathbb{S}^{3},\widetilde{\pi},\mathbb{S}^{2})$ and $\xi_{SO(3)} = (SO(3),\widetilde{\pi},\mathbb{S}^{2})$ be the principal $\mathbb{S}^{1}$-bundles shown in~\eqref{eq: Hopf SO(3) diagram}, and consider the group homomorphism  $\varphi:\mathbb{S}^{1}\to\mathbb{S}^{1}$ defined by $\varphi(z) = z^{2}$. The total space of $\varphi_{*}\xi_{H}$ is $\mathbb{S}^{3}/\mathbb{Z}_{2}\cong SO(3)$, and the adjoint map precisely satisfies $\text{Adj}_{q\cdot z} = \text{Adj}_{q}\cdot \varphi(z)$. We can therefore identify $\xi_{SO(3)}$ with $\varphi_{*}\xi_{H}$, so by Proposition~\ref{prop: Euler p-cover}, $e(\xi_{SO(3)}) = 2e(\xi_{H}) = -2$.\\
\end{proof}

Since principal $\mathbb{S}^{1}$-bundles over $\mathbb{S}^{2}$ are completely determined by their Euler numbers, Proposition~\ref{prop: Euler p-cover} actually implies \textit{every} such bundle can be realized as a quotient of the Hopf bundle. 
The quotient 3-manifold $\mathbb{S}^{3}/\mathbb{Z}_{p}$ is often denoted by $L(p,1)$ and called a \textbf{lens space}. 
Note that for any $p\in\mathbb{Z}$, we have $L(2p,1)\cong SO(3)/\mathbb{Z}_{p}$ and $(\varphi_{2p})_{*}\xi_{\text{Hopf}}\cong (\varphi_{p})_{*}\xi_{SO(3)}$.\\

The following proposition suggests how to construct the circle bundles over $\mathbb{RP}^{2}$:\\

\begin{proposition}\label{prop: bundle from equivariant}
    Suppose $\xi = (E,\pi,B)$ is a principal $G$-bundle, and $E$ and $B$ have free, proper right actions by a topological group $H$ such that: 

\begin{enumerate}
    \item $\pi:E\to B$ is $H$-equivariant, \\
    
    \item There is a representation $\varphi:H\to \text{Aut}(G)$ such that $(e\cdot g)\cdot h = (e\cdot h)\cdot \varphi_{h}(g)$ for all $e\in E$, $g\in G$ and $h\in H$.\\
\end{enumerate}    

Then, the map $\pi':E/H\to B/H$ defined by $\pi'([e]) = [\pi(e)]$ is a fiber bundle with fiber $G$ and structure group $G\rtimes_{\varphi}H$.\\

\end{proposition}

\begin{proof}
    We give a sketch -- the details are fairly straightforward and are left to the reader.  \\

    First, observe that the hypotheses imply $E$ admits a free right $(G\rtimes_{\varphi}H)$-action $e\cdot (g,h) = (e\cdot h)\cdot g$, and that the orbits of this action are precisely the fibers of the map $\Pi:E\to B/H$ defined by $\Pi (e) = [\pi(e)]$. That $\Pi$ is locally trivial follows directly from this property of $\pi$, so $(E,\Pi, B/H)$ is a principal $(G\rtimes_{\varphi}H)$-bundle. Define a left $(G\rtimes_{\varphi}H)$-action on $G$ by $(g,h)\cdot g' = \varphi_{h}(g')g^{-1}$ and form the associated bundle $(E\times_{G\rtimes_{\varphi} H}G, \Pi', B/H)$ with fiber $G$. One can check that the assignment $(e,g)\longmapsto [e\cdot g]$ is a homeomorphism which fits into the commutative diagram below:

\begin{equation}
\begin{tikzcd}
E\times_{G\rtimes_{\varphi} H}G \arrow[rr] \arrow[dr, swap, "\Pi'"] & & E/H \arrow[dl, "\pi'"] \\
& B/H &
\end{tikzcd}
\end{equation}

    where $\pi':E/H\to B/H$ is the map defined in the proposition. We conclude that $(E/H, \pi', B/H)$ is a fiber bundle with fiber $G$ and structure group $G\rtimes_{\varphi}H$.\\
\end{proof}

\begin{corollary}\label{cor: rp2 bundles}
    For any integer $p$, let $(L(2p,1),\pi_{2p},\mathbb{S}^{2})$ be the principal $\mathbb{S}^{1}$-bundle over $\mathbb{S}^{2}$ with Euler number $-2p$, and consider the free right $\mathbb{Z}_{2}$-actions on $\mathbb{S}^{2}$ and $L(2p,1)$ defined respectively by $v\cdot (-1) = -v$ and $[q]\cdot (-1) = [qj]$. The map $\pi_{2p}':L(2p,1)/\mathbb{Z}_{2}\to \mathbb{RP}^{2}$ defined by $\pi_{2p}'([e]) = [\pi_{2p}(e)]$ is a non-orientable circle bundle with (twisted) Euler number $\pm p$. \\
\end{corollary}

Now, suppose $X$ is a Hausdorff space which admits a continuous right $SO(3)$-action. The orbit $\mathcal{O}_{x}\subset X$ of any $x\in X$ is homeomorphic to the quotient of $SO(3)$ by the stabilizer subgroup $\text{stab}(x) = \{g\in SO(3): x\cdot g = x\}$. In particular, the map $T_{x}:SO(3)/\text{stab}(x)\to \mathcal{O}_{x}$ defined by $T_{x}([A]) = x\cdot A$ is a homeomorphism. \\

If $\text{stab}(x)$ is trivial, then $\mathcal{O}_{x}\cong SO(3)$, and $\mathcal{O}_{x}$ inherits a principal $\mathbb{S}^{1}$-structure defined by the right $\mathbb{S}^{1}$-action $y\cdot e^{i\theta} = y\cdot R_{e_{1},\theta}$. The induced projection map $\widetilde{\pi}:\mathcal{O}_{x}\to \mathbb{S}^{2}$ is defined by $\widetilde{\pi}(x\cdot A) = \pi (A) = Ae_{1}$ for all $A\in SO(3)$, and $T_{x}$ is an equivariant bundle isomorphism. \\ 

Now suppose $x\in X$ has a $p$-fold rotational symmetry in the sense that $\text{stab}(x)\cong \mathbb{Z}_{p}$ for some positive integer $p$. The orbit $\mathcal{O}_{x}$ is then homeomorphic to the lens space $L(2p,1)\cong SO(3)/\mathbb{Z}_{p}$. Just as before, $\mathcal{O}_{x}$ inherits a principal $\mathbb{S}^{1}$-structure such that $T_{x}$ is an equivariant bundle isomorphism. We conclude $(\mathcal{O}_{x},\widetilde{\pi},\mathbb{S}^{2})$ is a principal $\mathbb{S}^{1}$-bundle with Euler number $-2p$.\\

Lastly, suppose that in addition to a $p$-fold rotational symmetry in the plane perpendicular to $e_{1}$, $x\in X$ also has an additional 2-fold rotational symmetry with respect to this plane (as for the triangle prism density functions in Section~\ref{sec: Densities}). 
Then $\text{stab}(x)$ is the non-abelian group generated by $\{R_{e_{1},\frac{2\pi}{p}}, R_{e_{2},\pi}\}$, where $R_{e_{1},\theta}R_{e_{2},\pi} = R_{e_{2},\pi}R_{e_{1},-\theta}$. 
In particular, observe that $R_{e_{2},\pi}$ induces a free $\mathbb{Z}_{2}$-action on $SO(3)/\langle R_{e_{1},\frac{2\pi}{p}}\rangle$ given by $[A]\cdot (-1) = [AR_{e_{2},\pi}]$. 
Moreover, one can easily verify that $\pi([A]\cdot (-1)) = -\pi ([A])$ and $([A]\cdot R_{e_{1},\theta})\cdot (-1) = ([A]\cdot (-1))\cdot R_{e_{1},-\theta}$ for all $A\in SO(3)$. Thus, by Proposition~\ref{prop: bundle from equivariant}, the induced projection map $\pi':SO(3)/\langle R_{e_{1},\frac{2\pi}{p}}, R_{e_{2},\pi}\rangle\to \mathbb{RP}^{2}$ given by $\pi'([A]) = [\pi(A)]$ is a non-orientable circle bundle. 
In fact, using the adjoint map to identify $\mathbb{S}^{3}/\mathbb{Z}_{2}$ with $SO(3)$, this bundle is canonically isomorphic to the circle bundle $\pi'_{2p}:L(2p,1)/\mathbb{Z}_{2}\to \mathbb{RP}^{2}$ described in Corollary~\ref{cor: rp2 bundles} (note in particular that $\text{Adj}_{j} = R_{e_{2},\pi}$). \\

\subsection{Miscellaneous}\label{sec: Miscellaneous}

Here we provide several miscellaneous results which are used throughout the paper.\\ 

\begin{proposition}\label{prop: reflections}
    Suppose $r,R\in O(2)$ with $\det(r) = -1$ and $\det(R) = 1$ ($r$ is a reflection and $R$ is a rotation). Then, $rRr = R^{-1}$.\\
\end{proposition}

\begin{proof}
    Let $\text{Adj}_{r}(R) = rRr^{-1}$. Note that $SO(2)\subset O(2)\subset GL(2,\mathbb{R})$, so $r$ and $R$ are invertible (in particular, $r = r^{-1}$). Moreover, $\text{det}(\text{Adj}_{r}(R)) = 1$, so $\text{Adj}_{r}$ is a map from $SO(2)$ to itself. In fact, $\text{Adj}_{r}$ is clearly an involution, so we must have $\text{Adj}_{r}\in \text{Aut}(SO(2))\cong\mathbb{Z}_{2}$. On the other hand, $\text{Adj}_{r}$ cannot be the identity since this would imply $rR = Rr$, which we know is false by the semidirect decomposition of $O(2)$. We conclude that $\text{Adj}_{r}$ is the inversion map on $SO(2)$, so $rRr = R^{-1}$. \\
\end{proof}

\begin{proposition} \label{prop: Frobenius delta surjectivity}
    Suppose $A\in M_{n\times n}$ and $X\subset\mathbb{S}^{n-1}\subset\mathbb{R}^{n}$ such that $|Ax| < \varepsilon$ for all $x\in X$. If $d_{H}(X,\mathbb{S}^{n-1}) < \delta$ (with respect to the Euclidean metric), then $|A|_{F} < \frac{\varepsilon\sqrt{n}}{1 - \delta}$. \\   
\end{proposition}

\begin{proof}
    Choose any $v\in \mathbb{S}^{n-1}$, and let $x\in X$ such that $|v - x| < \delta$. Then,

\begin{equation*}
    |Av| \ \leq \ |A(v-x)| \ + \ |Ax| \ < \ \frac{|A|_{F}}{\sqrt{n}}\delta \ + \ \varepsilon
\end{equation*}

\noindent Thus,

\begin{equation*}
    |A|_{F} \ < \ \sqrt{n}\left(\frac{|A|_{F}}{\sqrt{n}}\delta \ + \ \varepsilon\right)
\end{equation*}

\noindent or equivalently, $|A|_{F} < \frac{\varepsilon\sqrt{n}}{1 - \delta}$. \\
\end{proof}

\begin{proposition}\label{prop: Karcher}
    Suppose $V = \{v_{j}\}_{j=0}^{n}\subset\mathbb{S}^{1}\subset\mathbb{R}^{2}$ such that $|v_{j}-v_{0}| < \sqrt{2}$ for all $j$. Then $V$ has a unique weighted Karcher mean with respect to the geodesic distance on $\mathbb{S}^{1}$ and any weights. In particular, given weights $w_{0},...,w_{n}$, the unique Karcher mean is given by 
    
\begin{equation*}
    \mu = \exp \left(\sum_{j=0}^{n} w_{j} \log (v_{j})\right)
\end{equation*}
    
    \noindent where $\log$ denotes any branch of $\log$ whose domain contains $V$. \\
\end{proposition}

\begin{proof}
    The fact that $|v_{j}-v_{0}|<\sqrt{2}$ for all $j$ implies $d_{\mathbb{S}^{1}}(v_{j},v_{0}) < \frac{\pi}{2}$ for all $j$, so by the triangle inequality, the diameter of $V$ (with respect to $d_{\mathbb{S}^{1}}$) is less than $\pi$. 
    Since $\mathbb{S}^{1}$ is geodesically convex on any subset of diameter less than $\pi$, $V$ has a unique weighted Karcher mean with respect to any weights. 
    In particular, we can choose lifts $\{\phi_{j}\}_{j=0}^{n}\subset\mathbb{R}$ of $\{v_{j}\}_{j=0}^{n}$ which lie in an interval $I\subset\mathbb{R}$ of length less than $\frac{1}{2}$. The Frechet variance of $V$ around any $p = e^{i\phi_{p}}\in\exp(I)\subset \mathbb{S}^{1}$ is given by 

\begin{align*}
    F_{V}(p) &= \sum_{j=0}^{n}w_{j}d_{\mathbb{S}^{1}}(p,v_{j})^{2} \\
    &= \sum_{j=0}^{n}w_{j}d_{\mathbb{S}^{1}}(e^{i\phi_{p}},e^{i\phi_{j}})^{2}\\
    &= \sum_{j=0}^{n}w_{j}d_{\mathbb{R}}(\phi_{p},\phi_{j})^{2} 
\end{align*}

    \noindent which is the weighted variance of $\{\phi_{j}\}_{j=0}^{n}$ about $\phi_{p}$. This is uniquely minimized at $\phi_{p} = \sum_{j=0}^{n}w_{j}\phi_{j}$, so $F_{V}$ has a minimizer at $p = \exp(\sum_{j=0}^{n}w_{j}\phi_{j})$. 
    Since we have shown that weighted Karcher mean is unique under the hypotheses, $p$ is the unique weighted Karcher mean of $V$. \\ 
\end{proof}

\printbibliography

@article{turow2026extended,
  title={An Extended Topological Model for High-Contrast Optical Flow},
  author={Turow, Brad and Perea, Jose A.},
  journal={arXiv preprint arXiv:2603.06853},
  year={2026}
}

@article{opt_flow_torus,
  author = {Henry Adams and Jonathan Bush and Brittany Carr and Lara Kassab and Joshua Mirth},
  title = {A Torus Model For Optical Flow},
  journal = {Pattern Recognition Letters},
  year = {2020},
  volume = {129},
  pages = {304--310},
}

@inproceedings{Sintel,
title = {A naturalistic open source movie for optical flow evaluation},
author = {Butler, D. J. and Wulff, J. and Stanley, G. B. and Black, M. J.},
booktitle = {European Conf. on Computer Vision (ECCV)},
editor = {{A. Fitzgibbon et al. (Eds.)}},
publisher = {Springer-Verlag},
series = {Part IV, LNCS 7577},
month = oct,
pages = {611--625},
year = {2012}
}

@inproceedings{Sparse_CC,
  title={Sparse circular coordinates via principal $\mathbb{Z}$-bundles},
  author={Perea, Jose A.},
  booktitle={Topological Data Analysis: The Abel Symposium 2018},
  pages={435--458},
  year={2020},
  doi={https://doi.org/10.1007/978-3-030-43408-3_17},
  organization={Springer}
}

@article{Klein_bottle,
  title={On the Local Behavior of Spaces of Natural Images},
  author={Carlsson, Gunnar and Ishkhanov, Tigran and de Silva, Vin and Zomorodian, Afra},
  journal={International Journal of Computer Vision},
  volume={76},
  pages={1--12},
  year={2008},
  doi={10.1007/s11263-007-0056-x},
  publisher={Springer Science+Business Media}
}

@article{Euclidean_Vector_Bundles,
  title={Approximate and discrete Euclidean vector bundles},
  author={Scoccola, Luis and Perea, Jose A.},
  journal={Forum of Mathematics, Sigma},
  volume={11},
  pages={e20},
  year={2023},
  doi={10.1017/fms.2023.16},
  eprint={2104.07563},
  archivePrefix={arXiv},
  primaryClass={math.AT}
}

@inproceedings{Mapper,
  title={Topological Methods for the Analysis of High Dimensional Data Sets and 3D Object Recognition},
  author={Singh, Gurjeet and M{\'e}moli, Facundo and Carlsson, Gunnar},
  booktitle={Eurographics Symposium on Point-Based Graphics},
  year={2007},
  editor={Botsch, M. and Pajarola, R.},
  organization={Eurographics Association},
  address={Switzerland}
}

@article{DREiMac,
  doi = {10.21105/joss.05791},
  url = {https://doi.org/10.21105/joss.05791},
  year = {2023},
  publisher = {The Open Journal},
  volume = {8},
  number = {91},
  pages = {5791},
  author = {Jose A. Perea and Luis Scoccola and Christopher J. Tralie},
  title = {{DREiMac}: Dimensionality Reduction with Eilenberg-MacLane Coordinates},
  journal = {Journal of Open Source Software}
}

@book{husemoller,
  author    = {Dale Husemoller},
  title     = {Fiber Bundles},
  year      = {1994},
  edition   = {3rd},
  publisher = {Springer},
  series    = {Graduate Texts in Mathematics},
  volume    = {20},
  isbn = {9780387940874},
}

@book{bredon1997sheaf,
  title={Sheaf Theory},
  author={Bredon, Glen E.},
  series={Graduate Texts in Mathematics},
  volume={170},
  publisher={Springer},
  year={1997},
  edition={2nd},
}

@book{hatcher2002algebraic,
  author    = {Hatcher, Allen},
  title     = {Algebraic Topology},
  year      = {2002},
  publisher = {Cambridge University Press},
  address   = {Cambridge},
  url       = {https://pi.math.cornell.edu/~hatcher/AT/AT.pdf}
}

@article{eilenberg1947homology,
  author    = {Eilenberg, Samuel},
  title     = {Homology of Spaces with Operators. I},
  journal   = {Transactions of the American Mathematical Society},
  volume    = {61},
  number    = {3},
  year      = {1947},
  pages     = {378--417},
  doi       = {10.1090/S0002-9947-1947-0021313-4}
}

@book{MilnorStasheff,
  author    = {John W. Milnor and James D. Stasheff},
  title     = {Characteristic Classes},
  series    = {Annals of Mathematics Studies},
  volume    = {76},
  year      = {1974},
  publisher = {Princeton University Press},
  address   = {Princeton, NJ},
  note      = {Available at \url{https://press.princeton.edu/books/paperback/9780691081229/characteristic-classes}}
}

@article{LeeEtAl2025,
  author    = {Andrew Lee and Harlin Lee and Jose A. Perea and Nikolas Schonsheck and Madeleine Weinstein},
  title     = {O(k)-Equivariant Dimensionality Reduction on Stiefel Manifolds},
  journal   = {SIAM Journal on Mathematics of Data Science},
  volume    = {7},
  number    = {2},
  pages     = {410--437},
  year      = {2025},
  doi       = {10.1137/23M1603443},
  eprint    = {2309.10775},
  archivePrefix = {arXiv},
  primaryClass = {cs.CG},
  url       = {https://doi.org/10.1137/23M1603443}
}

@article{Tinarrage2022,
  author  = {Raphaël Tinarrage},
  title   = {Computing persistent Stiefel-Whitney classes of line bundles},
  journal = {Journal of Applied and Computational Topology},
  volume  = {6},
  number  = {1},
  pages   = {65--125},
  year    = {2022},
  issn    = {2367-1726},
  doi     = {10.1007/s41468-021-00080-4},
  url     = {https://doi.org/10.1007/s41468-021-00080-4}
}

@article{rubner2000emd,
  title={The earth mover's distance as a metric for image retrieval},
  author={Rubner, Yossi and Tomasi, Carlo and Guibas, Leonidas J},
  journal={International journal of computer vision},
  volume={40},
  number={2},
  pages={99--121},
  year={2000},
  publisher={Springer}
}

@article{Horn_Opt_Flow,
  title={Determining optical flow},
  author={Horn, Berthold KP and Schunck, Brian G},
  journal={Artificial Intelligence},
  volume={17},
  number={1-3},
  pages={185--203},
  year={1981},
  publisher={Elsevier}
}

@inproceedings{Autonomous_Vehicles,
  title={Object Scene Flow for Autonomous Vehicles},
  author={Menze, Moritz and Geiger, Andreas},
  booktitle={Proceedings of the IEEE Conference on Computer Vision and Pattern Recognition (CVPR)},
  pages={3061--3070},
  year={2015},
  organization={IEEE},
  doi={10.1109/CVPR.2015.7298925}
}

@article{Opt_Flow_Applications,
  title={Optical flow and scene flow estimation: A survey},
  author={Zhai, Mingliang and Xiang, Xuezhi and Lv, Ning and Kong, Xiangdong},
  journal={Pattern Recognition},
  volume={114},
  pages={107861},
  year={2021},
  publisher={Elsevier}
}

@article{tenenbaum2000isomap,
  title={A global geometric framework for nonlinear dimensionality reduction},
  author={Tenenbaum, Joshua B and de Silva, Vin and Langford, John C},
  journal={Science},
  volume={290},
  number={5500},
  pages={2319--2323},
  year={2000},
  publisher={American Association for the Advancement of Science}
}

@inproceedings{scoccola2022fibered,
  author =	{Scoccola, Luis and Perea, Jose A.},
  title =	{FibeRed: Fiberwise Dimensionality Reduction of Topologically Complex Data with Vector Bundles},
  booktitle =	{39th International Symposium on Computational Geometry (SoCG 2023)},
  pages =	{56:1--56:18},
  series =	{Leibniz International Proceedings in Informatics (LIPIcs)},
  ISBN =	{978-3-95977-273-0},
  ISSN =	{1868-8969},
  year =	{2023},
  volume =	{258},
  editor =	{Chambers, Erin W. and Gudmundsson, Joachim},
  publisher =	{Schloss Dagstuhl -- Leibniz-Zentrum f{\"u}r Informatik},
  address =	{Dagstuhl, Germany},
  URL =		{https://drops.dagstuhl.de/entities/document/10.4230/LIPIcs.SoCG.2023.56},
  URN =		{urn:nbn:de:0030-drops-179068},
  doi =		{10.4230/LIPIcs.SoCG.2023.56}
}

@article{Singer,
  title={Exact and Stable Recovery of Rotations for Robust Synchronization},
  author={Wang, Lanhui and Singer, Amit},
  journal={Information and Inference: A Journal of the IMA},
  year={2013},
  pages={1--53},
  doi={10.1093/imaiai/drn000},
  note={Received on 15 July 2013},
  publisher={Oxford University Press}
}

@inproceedings{Roweis2000,
  title={Nonlinear dimensionality reduction by locally linear embedding},
  author={Roweis, Sam T and Saul, Lawrence K},
  booktitle={Advances in neural information processing systems},
  volume={13},
  year={2000}
}

@article{Belkin2003,
  title={Laplacian eigenmaps for dimensionality reduction and data representation},
  author={Belkin, Mikhail and Niyogi, Partha},
  journal={Neural computation},
  volume={15},
  number={6},
  pages={1373--1396},
  year={2003},
  publisher={MIT Press}
}

@article{Coifman2006,
  title={Diffusion maps},
  author={Coifman, Ronald R and Lafon, Stéphane},
  journal={Applied and computational harmonic analysis},
  volume={21},
  number={1},
  pages={5--30},
  year={2006},
  publisher={Elsevier}
}

@inproceedings{Sapiro2001,
  title={Geometric partial differential equations and image analysis},
  author={Sapiro, Guillermo},
  booktitle={Mathematical Imaging and Vision},
  year={2001},
  publisher={Springer}
}

@article{cryoem_singer_sigworth,
  author    = {Singer, Amit and Sigworth, Fred J.},
  title     = {Computational Methods for Single-Particle Cryo-Electron Microscopy},
  journal   = {Annual Review of Biophysics},
  volume    = {43},
  pages     = {469--492},
  year      = {2014},
  doi       = {10.1146/annurev-biophys-051013-022846}
}

@article{gao_synchronization,
  author    = {Gao, Tingran and Brodzki, Jacek and Mukherjee, Sayan},
  title     = {The Geometry of Synchronization Problems and Learning Group Actions},
  journal   = {Discrete and Computational Geometry},
  year      = {2019},
  doi       = {10.1007/s00454-019-00100-2}
}

@online{tshishiku-poincare-local-coeffs,
  author   = {Tshishiku, Bena},
  title    = {Local coefficients and Poincar{\'e} duality},
  date     = {2016-07-15},
  note     = {Lecture by Maggie Miller; notes by Bena Tshishiku},
  url      = {https://bena-tshishiku.github.io/files/papers/Poincare-duality-local-coefficients.pdf},
  urldate  = {2025-09-06}
}

@article{Smale1957Vietoris,
  author    = {Smale, Stephen},
  title     = {A Vietoris mapping theorem for homotopy},
  journal   = {Proceedings of the American Mathematical Society},
  volume    = {8},
  number    = {3},
  year      = {1957},
  pages     = {604--610},
  doi       = {10.1090/S0002-9939-1957-0087106-9},
  mrnumber  = {0087106},
  url       = {https://www.ams.org/journals/proc/1957-008-03/S0002-9939-1957-0087106-9/}
}

@article{Karcher1977,
  author  = {Karcher, Hermann},
  title   = {Riemannian Center of Mass and Mollifier Smoothing},
  journal = {Communications on Pure and Applied Mathematics},
  volume  = {30},
  number  = {5},
  year    = {1977},
  pages   = {509--541},
  doi     = {10.1002/cpa.3160300502}
}

@book{EilenbergSteenrod1952,
  author    = {Eilenberg, Samuel and Steenrod, Norman},
  title     = {Foundations of Algebraic Topology},
  publisher = {Princeton University Press},
  address   = {Princeton, NJ},
  year      = {1952},
  mrnumber  = {0050886}
}

@article{JMLR:v24:21-0073,
  author  = {Ephy R. Love and Benjamin Filippenko and Vasileios Maroulas and Gunnar Carlsson},
  title   = {Topological Convolutional Layers for Deep Learning},
  journal = {Journal of Machine Learning Research},
  year    = {2023},
  volume  = {24},
  number  = {59},
  pages   = {1--35},
  url     = {http://jmlr.org/papers/v24/21-0073.html}
}

@book{miranda1995algebraic,
  title={Algebraic curves and Riemann surfaces},
  author={Miranda, Rick},
  volume={5},
  year={1995},
  publisher={American Mathematical Soc.}
}

\end{document}